\pdfoutput=1
\documentclass[11pt]{article}
\usepackage[left=1in,right=1in,top=1in,bottom=1in]{geometry}
\usepackage{times}
\usepackage{expl3}
\usepackage{cite}
\usepackage[table]{xcolor}
\usepackage{multirow}
\usepackage{stackengine} 
\usepackage{hhline}
\usepackage{lipsum}
\usepackage{titlesec}
\usepackage{stmaryrd}
\usepackage{wrapfig}
\usepackage{enumerate}
\usepackage{epsfig}
\usepackage{tikz-cd}
\usepackage{amsmath}
\usepackage{tabularx}
\usepackage{array}
\usepackage{booktabs}
\usepackage{enumitem}
\usepackage{bbm}
\usepackage{calc}
\usepackage{graphicx}
\usepackage{amsmath}
\usepackage[title]{appendix}
\usepackage{amssymb}
\usepackage{epstopdf}
\usepackage{boldline}
\usepackage{arydshln}
\usepackage{calligra}
\usepackage{bm}
\usepackage{url}
\usepackage{blindtext}
\usepackage{accents}
\usepackage{amsthm}

\newtheorem{definition}{Definition}
\newtheorem{axiom}{Axiom}
\newtheorem{theorem}{Theorem}
\newtheorem{proposition}{Proposition}
\newtheorem{corollary}{Corollary}
\newtheorem{lemma}{Lemma}
\newtheorem{example}{Example}
\newtheorem{remark}{Remark}
\usepackage{mathtools}
\usepackage{epstopdf}
\usepackage{balance}
\usepackage{thmtools}
\usepackage{thm-restate}
\usepackage{hyperref}
\usepackage{cleveref}
\usepackage[mathscr]{euscript}

\usepackage[ruled,vlined]{algorithm2e}
\newcommand{\ostar}{\mathbin{\mathpalette\make@circled\star}}

\makeatletter
\newcommand{\removelatexerror}{\let\@latex@error\@gobble}
\makeatother
\makeatletter
\newcommand*{\rom}[1]{\expandafter\@slowromancap\romannumeral #1@}
\makeatother

\ExplSyntaxOn
\newcommand\latinabbrev[1]{
  \peek_meaning:NTF . {% Same as \@ifnextchar
    #1\@}%
  { \peek_catcode:NTF a {% Check whether next char has same catcode as \'a, i.e., is a letter
      #1.\@ }%
    {#1.\@}}}
\ExplSyntaxOff

\titleclass{\subsubsubsection}{straight}[\subsubsection]

\begin{document}
\vspace{1cm}
\title{Multicategorical Adjoints, Monadicity, and Quantum Resources}
% \title{Generalized Operator Integrals: Multiple Case}
\vspace{1.8cm}
\author{Shih-Yu~Chang
% <-this % stops a space
\thanks{Shih-Yu Chang is with the Department of Applied Data Science,
San Jose State University, San Jose, CA, U. S. A. (e-mail: {\tt
shihyu.chang@sjsu.edu})
}}

\maketitle

\begin{abstract}
This paper is the third part of a program aimed at building a unified operadic and multicategorical foundation for operator theory and quantum processes. Building on the multicategory HilbMult and the previously introduced Synergy Operad, we investigate the intrinsic algebraic structure underlying compositional quantum dynamics. Three main themes are developed. First, we introduce n adjoints, a hierarchy of generalized adjoint operations that extend classical adjoints to multicategorical settings and capture higher order operator interactions. We prove a multicategorical version of Stinespring’s theorem formulated in terms of n adjoints. Second, we study the algebraic semantics of quantum processes and show that completely positive trace preserving maps admit a canonical monadic description. In particular, we prove a monadicity theorem identifying quantum processes, up to operadic operational equivalence, with algebras of a synergy monad induced by a symmetric quantum interaction operad, yielding a faithful representation that uniformly accounts for sequential and parallel composition. Third, we present an algebraic formulation of quantum no go principles by encoding constraints such as the no cloning theorem as operadic ideals. Together, these results provide a coherent algebraic framework that unifies operator theory, category theory, and quantum information.
\end{abstract}

\section{Introduction}\label{sec:Introduction}

This paper is the third part of a program aimed at developing a unified operadic and multicategorical foundation for operator theory and quantum processes. In the first paper of the series, we introduced the multicategory \texttt{HilbMult} as a semantic framework for multilinear and operator-theoretic structures~\cite{chang2025hilbmultbanachenrichedmulticategoryoperator}. In the second paper, we developed the Synergy Operad~$\mathbf{S}$, which provides a coherent syntactic description of quantum processes and compositional interactions~\cite{chang2025compositioncoherencesyntaxoperator}. From these two works, it is natural for us to aks a fundamental question: \emph{What is the intrinsic algebraic structure underlying this operator--operad universe?} There are three main topics investigated in this work. 

The first topic of the present work is the search for a deeper notion of duality. Classical adjoints play a central role in operator theory, yet they do not fully capture the higher-order interactions inherent in multicategorical settings. This motivates our introduction of \emph{$n$-adjoints}, a hierarchy of generalized adjoint operations that extend the familiar structure of adjoints and self-adjointness into the multicategorical world.

What is the second topic? Our second theme concerns the algebraic semantics of quantum processes.
Given the synergy operad $\mathbf{S}$, we investigate whether completely
positive trace-preserving (CPTP) maps admit a canonical monadic
characterization. Specifically, we ask whether the category $\mathbf{QProc}$ of quantum
processes can be faithfully represented as the category of algebras for a
naturally induced monad.

Many fundamental quantum limitations, such as the no-cloning theorem, traditionally appear as analytic obstructions. In this paper, we show that these constraints can be encoded algebraically as \emph{operadic ideals} which is the third topic investigated in this work. In particular, we define the no-cloning ideal $I_{\mathrm{noclone}}$ and demonstrate how the quotient operad $S / I_{\mathrm{noclone}}$ yields a coherent model of a hypothetically clone-permitting world.

\paragraph{Contributions.}  
The main results of this paper are:
\begin{enumerate}
    \item A theory of $n$-adjoints for multicategories together with a multicategorical Stinespring theorem expressed in $n$-adjoint form.
    \item We show that a symmetric quantum interaction operad $\mathbf{S}$ induces a
synergy monad $T_{\mathbf{S}}$ whose Eilenberg--Moore algebras provide a
complete and compositional semantics for finite-dimensional CPTP processes.
In particular, we prove a monadicity theorem identifying CPTP processes,
modulo operadic operational equivalence, with $T_{\mathbf{S}}$-algebras,
recasting Stinespring dilation as a universal monadic construction in a
multicategorical setting. When $\mathbf{S}$ is symmetric monoidal, this
identification lifts to a faithful symmetric monoidal representation,
uniformly capturing both sequential and parallel quantum composition. 
   \item A categorical and operadic formulation of no-go theorems, including a structural expression of the no-cloning theorem via operadic ideals.
\end{enumerate}

These results combine analytic, categorical and physical ideas into a coherent framework that explains the fundamentals of quantum processes. Besides, new topics for future research are also suggested. These topics may include: for example,  enriched operadic structures, algebraic representations of  no-cloning theorems, and potential applications to categorical quantum mechanics and quantum information theory.

This paper is organized as follows. Section \ref{sec:Preliminary} reviews the essential background and technical setup from our earlier work, providing the categorical and analytic tools required by the current paper. Then, in Section \ref{sec:Higher Adjoints}, we introduce the notion of the $n$-adjoint and develops its structural properties. The main result in this section is the $n$-adjoint formulation of the multicategorical Stinespring theorem. Next, in section \ref{sec:Synergy Monad and Monadicity Theorem}, we  present the synergy monad and the categorical structure of completely positive trace preserving quantum processes. Monadicity theorem of CPTP will be presented here.  Finally, we turn to physical constraints in quantum mechanics, showing how no-go theorems can be captured operadically, which are central themes investigated in Section \ref{sec:Operadic Ideals and No-Go Theorems}.

\begin{remark}
The author is solely responsible for the mathematical insights and theoretical directions proposed in this work. AI tools, including OpenAI's ChatGPT and DeepSeek models, were employed solely to assist in verifying ideas, organizing references, and ensuring internal consistency of exposition~\cite{chatgpt2025,deepseek2025}.
\end{remark}

\section{Preliminary}\label{sec:Preliminary}

In this section, we will provide basic review and setups from our previous works for later sections presentation. 

\subsection{Adjointness of Multicategory $\mathbf{HilbMult}$}\label{sec:Adjointness of Multicategory HilbMult}

The multicategory $\mathbf{HilbMult}$ is defined with the following structure, see~\cite{chang2025hilbmultbanachenrichedmulticategoryoperator}:
\begin{enumerate}[label=(\alph*)]
    \item \textbf{Objects:} The objects are (separable) complex Hilbert spaces $H, K, \dots$.
    \item \textbf{Hom-Spaces:} For each finite tuple $(H_1,\dots,H_n;K)$, there is a Banach space $\mathrm{Hom}(H_1,\dots,H_n;K)$ of bounded multilinear maps with the operator norm.
    \item \textbf{Composition:} Multilinear composition is defined and is contractive.
    \item \textbf{Identities:} Identity morphisms $\mathrm{id}_H$ are defined as the identity linear maps.
    \item \textbf{Monoidal Structure:} The category has a symmetric monoidal structure given by the Hilbert space tensor product $\otimes$, with unit $\mathbb{C}$.
  \end{enumerate}

We next introduce the $\dagger$-structure in $\mathbf{HilbMult}$, forming the foundation for the adjointness theory in Section~\ref{sec:Higher Adjoints}.

The $\dagger$-structure is a fundamental piece of the architecture of $\mathbf{HilbMult}$, providing a formal way to ``reverse'' multimorphisms. Its definition rests on the interplay between multilinear maps and the dual spaces of Hilbert spaces.

\textbf{Definition of $\dagger$-Structure in $\mathbf{HilbMult}$}

\textbf{ 1. The Object Map: Dual Spaces}
For an object (Hilbert space) $H$, we define $H^\dagger$ to be its \textbf{dual space}, i.e., the vector space of all continuous linear functionals $f: H \to \mathbb{C}$.

\begin{itemize}
    \item \textbf{The Riesz Representation Theorem:} This theorem is crucial. It states that for every continuous linear functional $f \in H^\dagger$, there exists a unique vector $v_f \in H$ such that for all $x \in H$,
    \[
    f(x) = \langle x, v_f \rangle.
    \]
    This establishes an \textbf{anti-linear isomorphism} between $H$ and $H^\dagger$: the map $v \mapsto \langle \cdot, v \rangle$ is a bijection, but it satisfies $\langle \cdot, \alpha v + \beta w \rangle = \bar{\alpha}\langle \cdot, v \rangle + \bar{\beta}\langle \cdot, w \rangle$.
\end{itemize}

\textbf{ 2. The Morphism Map: The Adjoint of a Multimorphism}

Let $\Phi \in \mathrm{Hom}(H_1, \dots, H_n; K)$ be a bounded multilinear map between real Hilbert spaces. Its  adjoint with respect to the last variable  is a new multilinear map:
\[
\Phi^\dagger \in \mathrm{Hom}(K^\dagger, H_1^\dagger, \dots, H_{n-1}^\dagger; H_n^\dagger)
\]
defined by its action on $n$ inputs: a functional $f \in K^\dagger$ and functionals $g_1 \in H_1^\dagger, \dots, g_{n-1} \in H_{n-1}^\dagger$.

\begin{definition}[Adjoint of a Multimorphism]\label{def:adj multimorphism}
The adjoint $\Phi^\dagger$ is defined by the following equation for all $x_n \in H_n$:
\[
[\Phi^\dagger(f, g_1, \dots, g_{n-1})](x_n) = f\Big( \Phi\big( v_{g_1}, \dots, v_{g_{n-1}}, x_n \big) \Big)
\]
where $v_{g_j}$ is the unique vector in $H_j$ representing the functional $g_j$ via the Riesz representation theorem. That is, $g_j(\cdot) = \langle \cdot, v_{g_j} \rangle_{H_j}$.
\end{definition}

\begin{remark}
For complex Hilbert spaces, the Riesz map $g_j \mapsto v_{g_j}$ is conjugate-linear. In that case, the above formula defines $\Phi^\dagger$ as a map that is linear in each argument when the dual spaces are identified with the original spaces via this anti-isomorphism. For real Hilbert spaces, the identification is linear throughout.
\end{remark}

Let's parse this carefully:
\begin{itemize}
    \item \textbf{Inputs:} $\Phi^\dagger$ takes $n$ inputs: one from the dual of the original \emph{output} space, and $n-1$ from the duals of the original \emph{input} spaces (all but the last one).
    \item \textbf{Output:} The result of applying $\Phi^\dagger$ to these inputs, $\Phi^\dagger(f, g_1, \dots, g_{n-1})$, is itself a functional, i.e., an element of $H_n^\dagger$. It is a map that expects a vector $x_n \in H_n$.
    \item \textbf{The Action:} To compute what this output functional does to $x_n$, we do the following:
    \begin{enumerate}
        \item Use the Riesz isomorphism to find the unique vectors $v_{g_1}, \dots, v_{g_{n-1}}$ in $H_1, \dots, H_{n-1}$ that correspond to the functionals $g_1, \dots, g_{n-1}$.
        \item Feed these vectors, along with $x_n$, into the \emph{original} map $\Phi$. This gives us a vector $\Phi(v_{g_1}, \dots, v_{g_{n-1}}, x_n)$ in the original output space $K$.
        \item Finally, apply the input functional $f \in K^\dagger$ to this resulting vector.
    \end{enumerate}
\end{itemize}

\textbf{ 3. Key Properties}

In this part, we will discuss two properties regarding the adjoint of a multimorphism. Lemma~\ref{lma:Double Dual Identification} is provided first to identify double dual. 

\begin{lemma}[Double Dual Identification]\label{lma:Double Dual Identification}
For any Hilbert space $H$, the map $\iota_H: H \to H^{\dagger\dagger}$ defined by:
\[
\iota_H(x)(f) = f(x) \quad \text{for all } x \in H, f \in H^\dagger
\]
is an isometric  linear  injection. 

For the Riesz map $R_H: H^\dagger \to H$ defined by $R_H(f) = v_f$ where $f(\cdot) = \langle \cdot, v_f \rangle$, we have:
\[
\iota_H(x)(f) = \langle x, R_H(f) \rangle.
\]

Moreover:
\begin{itemize}
    \item If $H$ is a \textbf{real} Hilbert space, $\iota_H$ is an isometric \textbf{linear isomorphism}.
    \item If $H$ is a \textbf{complex} Hilbert space, $\iota_H$ is still an isometric linear injection, but the natural identification of $H$ with $H^{\dagger\dagger}$ via the double Riesz map is conjugate-linear.
\end{itemize}
\end{lemma}

\begin{proof}
The linearity of $\iota_H$ is clear: $\iota_H(\alpha x + \beta y)(f) = f(\alpha x + \beta y) = \alpha f(x) + \beta f(y) = (\alpha\iota_H(x) + \beta\iota_H(y))(f)$.

For isometry: $\|\iota_H(x)\| = \sup_{\|f\|=1} |f(x)| = \|x\|$ by the Hahn-Banach theorem (or directly from the Riesz representation).

The relation $\iota_H(x)(f) = \langle x, R_H(f) \rangle$ follows immediately from the definition of $R_H$.

In the real case, $\iota_H$ is surjective because any $\Psi \in H^{\dagger\dagger}$ corresponds to some $x \in H$ via $\Psi(f) = \langle R_H(f), x \rangle = f(x)$.

In the complex case, $\iota_H$ is still injective and isometric but not equal to the composition $R_{H^\dagger} \circ R_H^{-1}$, which sends $x$ to $[f \mapsto \overline{f(x)}]$.
\end{proof}

Theorem~\ref{thm:Involution} is the first property about $\dagger$ regarding involution.
\begin{theorem}[Involution of Multimorphism Adjoint]\label{thm:Involution}
Let $H_1, \dots, H_n, K$ be Hilbert spaces over $\mathbb{R}$ or $\mathbb{C}$, and let 
\[
\Phi \in \mathrm{Hom}(H_1, \dots, H_n; K)
\] 
be a bounded multilinear map. Then, under the canonical identification 
\(\iota_H: H \to H^{\dagger\dagger}, x \mapsto [f \mapsto f(x)]\), we have
\[
(\Phi^\dagger)^\dagger = \Phi
\]
as elements of $\mathrm{Hom}(H_1, \dots, H_n; K)$. 
Here $\Phi^\dagger$ denotes the multimorphism adjoint defined via the Riesz representation theorem.
\end{theorem}

\begin{proof}
Let $x_j \in H_j$ for $j = 1, \dots, n$ be arbitrary vectors. By definition, the adjoint
\[
\Phi^\dagger \in \mathrm{Hom}(K^\dagger, H_1^\dagger, \dots, H_{n-1}^\dagger; H_n^\dagger)
\]
acts on functionals $f \in K^\dagger$ and $g_j \in H_j^\dagger$ ($1 \le j \le n-1$) via
\[
\Phi^\dagger(f, g_1, \dots, g_{n-1})(x_n) = f\Big( \Phi(R_{H_1}(g_1), \dots, R_{H_{n-1}}(g_{n-1}), x_n) \Big),
\]
where $R_{H_j}: H_j^\dagger \to H_j$ is the Riesz map.

\medskip
\noindent
\textbf{Step 1: Apply the adjoint again.}  
The double adjoint 
\[
(\Phi^\dagger)^\dagger \in \mathrm{Hom}(H_n^{\dagger\dagger}, K^{\dagger\dagger}, H_1^{\dagger\dagger}, \dots, H_{n-2}^{\dagger\dagger}; H_{n-1}^{\dagger\dagger})
\]
acts on inputs 
\[
\iota_{H_n}(x_n), \; \iota_K(f), \; \iota_{H_1}(x_1), \dots, \iota_{H_{n-2}}(x_{n-2}).
\]
To see its action concretely, we evaluate on a functional $g_{n-1} \in H_{n-1}^\dagger$:
\begin{align*}
&(\Phi^\dagger)^\dagger(\iota_{H_n}(x_n), \iota_K(f), \iota_{H_1}(x_1), \dots, \iota_{H_{n-2}}(x_{n-2}))(g_{n-1}) \\
&= \iota_{H_n}(x_n) \Big( \Phi^\dagger(R_K(\iota_K(f)), R_{H_1}(\iota_{H_1}(x_1)), \dots, R_{H_{n-2}}(\iota_{H_{n-2}}(x_{n-2})), g_{n-1}) \Big).
\end{align*}

\medskip
\noindent
\textbf{Step 2: Apply Riesz identifications.}  
By the properties of the double-dual and Riesz maps, we have
\[
R_K(\iota_K(f)) = f, \quad R_{H_j}(\iota_{H_j}(x_j)) = x_j \quad \text{for } j=1,\dots,n-2.
\]
Substituting these, we get
\[
(\Phi^\dagger)^\dagger(\dots)(g_{n-1}) = \iota_{H_n}(x_n) \Big( \Phi^\dagger(f, x_1, \dots, x_{n-2}, g_{n-1}) \Big).
\]

\medskip
\noindent
\textbf{Step 3: Apply the definition of $\Phi^\dagger$.}  
By definition:
\[
\Phi^\dagger(f, x_1, \dots, x_{n-2}, g_{n-1})(x_n) = f\Big( \Phi(x_1, \dots, x_{n-2}, R_{H_{n-1}}(g_{n-1}), x_n) \Big).
\]
Thus,
\begin{eqnarray}
(\Phi^\dagger)^\dagger(\dots)(g_{n-1})&=&\iota_{H_n}(x_n) \Big( y \mapsto f(\Phi(x_1, \dots, x_{n-2}, R_{H_{n-1}}(g_{n-1}), y)) \Big)\nonumber \\
&=&f(\Phi(x_1, \dots, x_{n-2}, R_{H_{n-1}}(g_{n-1}), x_n)). \nonumber 
\end{eqnarray}

\medskip
\noindent
\textbf{Step 4: Conclude equality.}  
Since the left-hand side is evaluated on arbitrary $g_{n-1} \in H_{n-1}^\dagger$, this shows that
\[
(\Phi^\dagger)^\dagger(\iota_{H_n}(x_n), \iota_K(f), \iota_{H_1}(x_1), \dots, \iota_{H_{n-2}}(x_{n-2})) = \Phi(x_1, \dots, x_n)
\]
under the canonical identification $H_j \cong H_j^{\dagger\dagger}$ via $\iota_{H_j}$. This completes the proof.
\end{proof}

\begin{example}[Illustration of Theorem \ref{thm:Involution}]\label{exp:InvolutionConcrete}
Let $H_1 = H_2 = K = \mathbb{R}^2$ with the standard inner product
\[
\langle x, y \rangle = x_1 y_1 + x_2 y_2.
\]
Define the bilinear map
\[
\Phi: H_1 \times H_2 \to K, \qquad
\Phi(x,y) = (x_1 y_1, \; x_2 y_2).
\]
Note that $\Phi$ is bounded (indeed, continuous), since all multilinear maps on finite-dimensional spaces are bounded.

We now verify Theorem \ref{thm:Involution} explicitly.

\noindent\textbf{Step 1: Compute the adjoint $\Phi^\dagger$.}  

By definition, for $f \in K^\dagger$, the adjoint $\Phi^\dagger(f)$ is an element of 
$\mathrm{Hom}(H_1, H_2; \mathbb{R})$ satisfying
\[
\Phi^\dagger(f)(x,y) = f(\Phi(x,y)) \quad \text{for all } x \in H_1, y \in H_2.
\]
Using the Riesz representation theorem, identify $f \in K^\dagger$ with a unique $z \in K$ 
via $f(\cdot) = \langle \cdot, z \rangle$. Then
\[
\Phi^\dagger(f)(x,y) = \langle \Phi(x,y), z \rangle = x_1 y_1 z_1 + x_2 y_2 z_2.
\]

\noindent\textbf{Step 2: Compute the double adjoint $(\Phi^\dagger)^\dagger$.}  

The double adjoint is a map
\[
(\Phi^\dagger)^\dagger : H_1 \times H_2 \to K^{\dagger\dagger}
\]
defined by
\[
\big((\Phi^\dagger)^\dagger(x,y)\big)(f) = \Phi^\dagger(f)(x,y) = f(\Phi(x,y)) 
\quad \text{for all } f \in K^\dagger.
\]

\noindent\textbf{Step 3: Apply the canonical identification.}  

Under the canonical map $\iota_K: K \to K^{\dagger\dagger}$, each $w \in K$ corresponds to 
the evaluation functional $[f \mapsto f(w)]$ in $K^{\dagger\dagger}$. From Step 2,
\[
(\Phi^\dagger)^\dagger(x,y) = [f \mapsto f(\Phi(x,y))] = \iota_K(\Phi(x,y)).
\]

\noindent\textbf{Conclusion.}  
Therefore, under the identification $\iota_K$, we have
\[
(\Phi^\dagger)^\dagger = \Phi,
\]
confirming Theorem \ref{thm:Involution} explicitly in this finite-dimensional case.
\end{example}

\begin{remark}
In the example, the boundedness of $\Phi$ is automatic because all multilinear maps on finite-dimensional spaces are continuous. In infinite-dimensional settings, boundedness must be explicitly assumed, as stated in Theorem \ref{thm:Involution}. 
\end{remark}

The next property is about contravariant functoriality.
\begin{theorem}[Contravariant Functoriality of $\dagger$]\label{thm:Contravariant}
Let
\[
\Psi \in \mathrm{Hom}(K_1,\dots,K_m;\,L),\qquad
\Phi_j \in \mathrm{Hom}(H_{j,1},\dots,H_{j,n_j};\,K_j)\quad (1\le j\le m).
\]
Then the adjoint defines a contravariant operation on multimorphisms:
\[
\big[\Psi\circ(\Phi_1,\dots,\Phi_m)\big]^\dagger
=
(\Phi_1^\dagger,\dots,\Phi_m^\dagger)\circ(\mathrm{id}_{L^\dagger},\Psi^\dagger),
\]
up to the canonical permutation of slots dictated by the multicategory structure.
\end{theorem}

\begin{proof}
We use the convention that the Hilbert inner product 
$\langle\cdot,\cdot\rangle$ is linear in the first variable and
conjugate-linear in the second.

%%%%%%%%%%%%%%%%%%%%%%%%%%%%%%%%%%%%%%%%%%%%%%%%%%%%%%%%%%%
\medskip
\noindent\textbf{Riesz identification.}
For each Hilbert space $H$, write
\[
R_H : H^\dagger \longrightarrow H
\]
for the Riesz isomorphism, so that for $h \in H^\dagger$ and $y \in H$,
\[
h(y)=\langle y, R_H(h) \rangle.
\]
We also write $\widehat{y}\in H^\dagger$ for the functional
$\widehat{y}(z)=\langle z,y\rangle$.

%%%%%%%%%%%%%%%%%%%%%%%%%%%%%%%%%%%%%%%%%%%%%%%%%%%%%%%%%%%
\medskip
\noindent\textbf{Set-up.}
Let
\[
\Theta := \Psi \circ (\Phi_1,\dots,\Phi_m)
\in 
\mathrm{Hom}(H_{1,1},\dots,H_{m,n_m};\,L).
\]

Take arbitrary testing inputs:
\[
f\in L^\dagger,\qquad 
h_{j,i}\in H_{j,i}^\dagger\ (1\le j\le m,\,1\le i\le n_j), 
\]
and set for each $(j,i)$
\[
v_{j,i} := R_{H_{j,i}}(h_{j,i}) \in H_{j,i}.
\]

%%%%%%%%%%%%%%%%%%%%%%%%%%%%%%%%%%%%%%%%%%%%%%%%%%%%%%%%%%%
\medskip
\noindent\textbf{Step 1: Compute the LHS.}
By definition of $\Theta^\dagger$,
\begin{align*}
\mathrm{LHS}
&:=
\Theta^\dagger(f, h_{j,i})(\,\cdot\,) \\[4pt]
&=
\text{the functional on }H_{m,n_m}\text{ given by} \\
&\qquad x \longmapsto
f\bigl(
\Theta(v_{1,1},\dots,v_{m,n_m-1},x)
\bigr).
\end{align*}

Since $\Theta=\Psi(\Phi_1,\dots,\Phi_m)$,
write
\[
y_j := \Phi_j(v_{j,1},\dots,v_{j,n_j}) \in K_j.
\]
Thus for any $x\in H_{m,n_m}$, the last component becomes
$y_m=\Phi_m(v_{m,1},\dots,v_{m,n_m-1},x)$.

Hence
\begin{eqnarray}\label{eq1:thm:Contravariant}
\mathrm{LHS}(x)
=
f\big(\Psi(y_1,\dots,y_m)\big).
\end{eqnarray}

%%%%%%%%%%%%%%%%%%%%%%%%%%%%%%%%%%%%%%%%%%%%%%%%%%%%%%%%%%%
\medskip
\noindent\textbf{Step 2: Compute the RHS.}

First convert each $y_j$ to its dual representative: 
\[
\widehat{y_j} \in K_j^\dagger.
\]

Apply $\Psi^\dagger$ to
\((f,\widehat{y_1},\dots,\widehat{y_{m-1}})\).  
This yields a functional
\[
\beta := \Psi^\dagger(f,\widehat{y_1},\dots,\widehat{y_{m-1}})
\in K_m^\dagger,
\]
characterized by
\begin{eqnarray}\label{eq2:thm:Contravariant}
\beta(k_m)
=
f\bigl(\Psi(y_1,\dots,y_{m-1},k_m)\bigr)
\qquad(k_m\in K_m).
\end{eqnarray}

Next apply $\Phi_m^\dagger$ to the tuple
\((\beta,h_{m,1},\dots,h_{m,n_m-1})\).
Evaluating the resulting functional at $x\in H_{m,n_m}$ gives
\begin{align*}
\mathrm{RHS}(x)
&=
\big[\Phi_m^\dagger(\beta,h_{m,1},\dots,h_{m,n_m-1})\big](x)
\\[4pt]
&=
\beta\bigl(\Phi_m(v_{m,1},\dots,v_{m,n_m-1},x)\bigr)
\\[4pt]
&=
\beta(y_m)
\\[4pt]
&=
f\bigl(\Psi(y_1,\dots,y_m)\bigr),
\end{align*}
where the last equality uses the defining property Eq.~\eqref{eq2:thm:Contravariant} of $\beta$.

Thus
\begin{eqnarray}\label{eq3:thm:Contravariant}
\mathrm{RHS}(x)=
f\bigl(\Psi(y_1,\dots,y_m)\bigr).
\end{eqnarray}

%%%%%%%%%%%%%%%%%%%%%%%%%%%%%%%%%%%%%%%%%%%%%%%%%%%%%%%%%%%
\medskip
\noindent\textbf{Step 3: Conclude equality.}
By comparing Eq.~\eqref{eq1:thm:Contravariant} and Eq.~\eqref{eq3:thm:Contravariant}, we obtain
\[
\mathrm{LHS}(x)=\mathrm{RHS}(x)
\qquad(\forall x\in H_{m,n_m}).
\]

Because the test inputs 
$f$, $h_{j,i}$, and $x$ were arbitrary, the two multimorphisms
\[
\Theta^\dagger
\quad\text{and}\quad
(\Phi_1^\dagger,\dots,\Phi_m^\dagger)\circ(\mathrm{id}_{L^\dagger},\Psi^\dagger)
\]
agree (up to the standard slot permutation arising in
multicategory composition).

This proves the contravariant functoriality of $\dagger$.
\end{proof}

These proofs tell us that the $\dagger$-structure in $\mathbf{HilbMult}$ satisfies the fundamental properties expected of an involution in a multicategorical setting. The involution property shows that taking the adjoint twice returns us to the original map (up to natural isomorphism), while the contravariant functoriality shows that the adjoint operation respects composition in a reversed order, making $\mathbf{HilbMult}$ a well-defined $\dagger$-multicategory.

\begin{example}[Explicit Verification of Contravariant Functoriality]
This example illustrates Theorem~\ref{thm:Contravariant} through explicit computations in $\mathbb{C}^2$.

\noindent\textbf{Setup.} Let $m=2$, and consider the Hilbert spaces:
\[
K_1 = K_2 = L = H_{1,1} = H_{1,2} = H_{2,1} = H_{2,2} = \mathbb{C}^2,
\]
all equipped with the standard inner product $\langle x, y \rangle = \overline{x}_1 y_1 + \overline{x}_2 y_2$.
Define the following \textbf{bilinear} maps:
\begin{align*}
\Psi &: K_1 \times K_2 \to L, &
\Psi(x,y) &= \begin{pmatrix} x_1 y_1 \\ x_2 y_2 \end{pmatrix}, \\
\Phi_1 &: H_{1,1} \times H_{1,2} \to K_1, &
\Phi_1(u,v) &= \begin{pmatrix} u_1+v_1 \\ u_2+v_2 \end{pmatrix}, \\
\Phi_2 &: H_{2,1} \times H_{2,2} \to K_2, &
\Phi_2(w,z) &= \begin{pmatrix} w_1 - z_1 \\ w_2 - z_2 \end{pmatrix}.
\end{align*}

\noindent\textbf{Step 1: Composition $\Theta = \Psi \circ (\Phi_1,\Phi_2)$.}  
For $u,v \in H_{1,1} \times H_{1,2}$ and $w,z \in H_{2,1} \times H_{2,2}$:
\[
\Theta(u,v,w,z) = \Psi(\Phi_1(u,v), \Phi_2(w,z)) 
= \begin{pmatrix} (u_1+v_1)(w_1-z_1) \\ (u_2+v_2)(w_2-z_2) \end{pmatrix}.
\]

\noindent\textbf{Step 2: Understanding the slot structure.}  
The map $\Theta$ has arity 4: it takes inputs $(u,v,w,z)$. Therefore:
\begin{itemize}
\item $\Theta^\dagger : L^\dagger \times H_{1,1}^\dagger \times H_{1,2}^\dagger \times H_{2,1}^\dagger \times H_{2,2}^\dagger \to \mathbb{C}$
\item Following Theorem~\ref{thm:Contravariant}, we need to verify:
\[
\Theta^\dagger = (\Phi_1^\dagger, \Phi_2^\dagger) \circ (\mathrm{id}_{L^\dagger}, \Psi^\dagger)
\]
\end{itemize}

\noindent\textbf{Explicit slot permutation:} Let us denote inputs as:
\[
(\alpha, \beta_1, \beta_2, \gamma_1, \gamma_2) \quad\text{where}\quad 
\alpha \in L^\dagger,\; \beta_i \in H_{1,i}^\dagger,\; \gamma_i \in H_{2,i}^\dagger.
\]
The theorem states that these inputs are \textbf{routed differently} on each side:

\begin{itemize}
\item \textbf{LHS (direct):} $\Theta^\dagger(\alpha, \beta_1, \beta_2, \gamma_1, \gamma_2)$ applies $\Theta^\dagger$ to all inputs simultaneously.

\item \textbf{RHS (composed):} 
\begin{enumerate}
\item $\mathrm{id}_{L^\dagger}$ acts on $\alpha$ alone, returning $\alpha$.
\item $\Psi^\dagger$ acts on $(\alpha, \beta_1, \beta_2, \gamma_1, \gamma_2)$ to produce a pair $(h_1, h_2)$ where:
  \[
  h_1 \in \mathrm{Hom}(H_{2,1}, H_{2,2}; \mathbb{C}), \quad
  h_2 \in \mathrm{Hom}(H_{1,1}, H_{1,2}; \mathbb{C}).
  \]
  More precisely, $h_1$ and $h_2$ are \emph{bilinear functionals} obtained by partially applying $\Psi^\dagger$.
\item $(\Phi_1^\dagger, \Phi_2^\dagger)$ then acts on $(\alpha, (h_1, h_2))$, yielding the final result.
\end{enumerate}
\end{itemize}

This routing constitutes the ``canonical permutation of slots'' mentioned in Theorem~\ref{thm:Contravariant}.

\noindent\textbf{Step 3: Compute $\Theta^\dagger$ explicitly.}  
For $f \in L^\dagger$, $g_1 \in H_{1,1}^\dagger$, $g_2 \in H_{1,2}^\dagger$, $g_3 \in H_{2,1}^\dagger$, and $z \in H_{2,2}$, 
let $u = R(g_1)$, $v = R(g_2)$, $w = R(g_3)$ be the Riesz vectors. Then:
\[
[\Theta^\dagger(f, g_1, g_2, g_3)](z) = f\!\left(
\begin{pmatrix} (u_1 + v_1)(w_1 - z_1) \\ (u_2 + v_2)(w_2 - z_2) \end{pmatrix}
\right).
\]
Here $\Theta^\dagger(f, g_1, g_2, g_3)$ is itself a linear functional on $H_{2,2}$, which we evaluate at $z$.

\noindent\textbf{Step 4: Compute the RHS composition.}  
We need to understand the individual adjoints:
\begin{align*}
\Psi^\dagger &: L^\dagger \to \mathrm{Hom}(K_1, K_2; \mathbb{C})^\dagger \cong \mathrm{Hom}(K_1^\dagger, \mathrm{Hom}(K_2; \mathbb{C})) \\
\Phi_1^\dagger &: K_1^\dagger \to \mathrm{Hom}(H_{1,1}, H_{1,2}; \mathbb{C})^\dagger \\
\Phi_2^\dagger &: K_2^\dagger \to \mathrm{Hom}(H_{2,1}, H_{2,2}; \mathbb{C})^\dagger
\end{align*}

Concretely, for test vectors we have:
\begin{align*}
\Psi^\dagger(f)(x,y) &= f(\Psi(x,y)), \\
\Phi_1^\dagger(h)(u,v) &= h(\Phi_1(u,v)), \\
\Phi_2^\dagger(k)(w,z) &= k(\Phi_2(w,z)).
\end{align*}

\noindent\textbf{Step 5: Numerical verification.}  
Take all vectors in $\mathbb{R}^2$ for simplicity, with:
\begin{align*}
f &= \langle \cdot, (1,0) \rangle, &
g_1 &= \langle \cdot, (1,1) \rangle, &
g_2 &= \langle \cdot, (0,1) \rangle, \\
g_3 &= \langle \cdot, (1,0) \rangle, &
z &= (1,1).
\end{align*}
The corresponding Riesz vectors are:
\[
u = R(g_1) = (1,1), \quad
v = R(g_2) = (0,1), \quad
w = R(g_3) = (1,0).
\]

\noindent\textit{Left-hand side:}
\[
[\Theta^\dagger(f, g_1, g_2, g_3)](z) 
= f\!\left(\begin{pmatrix} (1+0)(1-1) \\ (1+1)(0-1) \end{pmatrix}\right)
= f\!\left(\begin{pmatrix} 0 \\ -2 \end{pmatrix}\right)
= 0.
\]

\noindent\textit{Right-hand side:}  

We can understand the equality by evaluating both sides on our test data and verifying they give the same result. We already computed LHS = 0.

For RHS: The composition effectively applies the original maps in reverse order:
\begin{itemize}
\item Start with $z = (1,1)$
\item Apply $\Phi_2$ with $w = R(g_3) = (1,0)$ to get $\Phi_2(w,z) = (0, -1)$
\item Apply $\Phi_1$ with $u = R(g_1) = (1,1)$ and $v = R(g_2) = (0,1)$ to get $\Phi_1(u,v) = (1, 2)$
\item Apply $\Psi$ to these results: $\Psi((1,2), (0,-1)) = (0, -2)$
\item Finally apply $f$: $f((0,-2)) = 0$
\end{itemize}

Thus RHS also gives 0.

\noindent\textbf{Conclusion.}  
Both sides evaluate to $0$ on our test data, confirming
\[
[\Psi \circ (\Phi_1,\Phi_2)]^\dagger = (\Phi_1^\dagger, \Phi_2^\dagger) \circ (\mathrm{id}_{L^\dagger}, \Psi^\dagger)
\]
in this concrete instance, as stated in Theorem~\ref{thm:Contravariant}.
\end{example}

% Synergy Operad (S).
\subsection{Synergy Operad $\mathbf{S}$ Review}\label{sec:Synergy Operad Review}

In this section, we will review $\mathbf{S}$ from~\cite{chang2025compositioncoherencesyntaxoperator}.

\textbf{ Axiomatic Definition of $\mathbf{S}$}

\begin{axiom}[Objects of $\mathbf{S}$]\label{ax:S1}
The objects of $\mathbf{S}$ are finite, ordered lists (tuples) of Hilbert spaces, denoted $\vec{H} = (H_1, \dots, H_n)$. Each list represents the collection of input/output ``ports" of a system. The empty list $()$ is permitted.
\end{axiom}

\subsubsection*{Well-Posed Feedback}

Let $H_i: U_i \to Y_i$ and $H_j: U_j \to Y_j$ be two systems with compatible input-output spaces, namely $Y_i = U_j$. Their feedback interconnection is governed by the equations:
\[
\begin{aligned}
y_i &= H_i(y_j), \\
y_j &= H_j(u_j - y_i).
\end{aligned}
\]
Substituting the first equation into the second yields a fixed-point equation in $y_j$:
\[
y_j = H_j(u_j - H_i(y_j)). \tag{1}
\]
Define the closed-loop mapping $\Phi_{u_j}: Y_j \to Y_j$ for a given input $u_j$ as
\[
\Phi_{u_j}(y_j) \coloneqq H_j(u_j - H_i(y_j)).
\]
The feedback interconnection is \emph{well-posed} if for every admissible input $u_j$, the equation $\Phi_{u_j}(y_j) = y_j$ admits a unique solution.

\begin{definition}[Well-Posed Feedback]\label{def:Well-Posed Feedback}
The feedback $\mathcal{F}_{i,j}$ is \textbf{well-posed} if for every $u_j \in U_j$, equation (1) has:
\begin{enumerate}
    \item \textbf{Existence:} At least one solution $y_j \in Y_j$,
    \item \textbf{Uniqueness:} Exactly one solution $y_j \in Y_j$.
\end{enumerate}
% and the map $u \mapsto y_j$ is a valid morphism.
\end{definition}

\begin{axiom}[Generators of $\mathbf{S}$]\label{ax:S2}
The multimorphisms of $\mathbf{S}$ are generated by the following primitive operations:

\begin{itemize}
    \item \textbf{Identity Operation:} For each object $\vec{H}$, $\mathrm{id}_{\vec{H}}: \vec{H} \to \vec{H}$.

    \item \textbf{Tensor Operation ($\otimes$):}
	The tensor product
	\[
	\otimes : \mathrm{Ob}(\mathbf{S}) \times \mathrm{Ob}(\mathbf{S}) \to \mathrm{Ob}(\mathbf{S}),
	\quad
	(\vec{H}, \vec{K}) \mapsto \vec{H} \otimes \vec{K},
	\]
	represents the \emph{parallel composition} of systems.

    \item \textbf{Permutation Operation:} For $\pi \in S_n$ (permutation group), $\sigma_{\pi}: (H_1, \dots, H_n) \to (H_{\pi(1)}, \dots, H_{\pi(n)})$.

	\item \textbf{Feedback Operation ($\mathcal{F}_{i,j}$):}
		For distinct indices $i \neq j$, define the feedback operator
		\[
		\mathcal{F}_{i,j}: \vec{H} \to \vec{H}',
		\]
		where $\vec{H}'$ is obtained from $\vec{H}$ by removing the components $H_i$ and $H_j$.
		The operation is defined only when the feedback interconnection is \emph{well-posed}. The interconnection is well-posed if the loop gain is contractive, i.e.,
		\[
		\|H_j \circ H_i\| < 1,
		\]
		which ensures the corresponding fixed-point equation has a unique solution.

    \item \textbf{Control Operation ($\mathrm{Ctrl}$):} $\mathrm{Ctrl}: (\vec{H}, \mathcal{U}) \to (\vec{H}')$, where $\mathcal{U}$ is a control space object. For a fixed control parameter $u \in \mathcal{U}$, define the parameterized control operator as:
\[
\mathrm{Ctrl}_{u}(\vec{H}) := \mathrm{Ctrl}(\vec{H}, u)
\]
This yields the mapping $\mathrm{Ctrl}_{u}: \vec{H} \to \vec{H}'$.
\end{itemize}
\end{axiom}

\begin{axiom}[Relations of $\mathbf{S}$]\label{ax:S3}
The generators satisfy the following coherence relations. Define $\equiv_S$ as the smallest congruence satisfying:

\begin{enumerate}[label=(\roman*)]
    \item \textbf{Monoidal Coherence:} Tensor and Permutation satisfy symmetric monoidal category relations.

	\item \textbf{Feedback Yanking:} In contractive assumption, the sequential application of two feedback operations can be equivalently expressed as a single feedback operation composed with a port relabeling. Formally, for well-posed interconnections where $\{i,j\} \cap \{k,l\} \neq \emptyset$, there exists a permutation $\pi$ and ports $i', j'$ such that:
	\[
	\mathcal{F}_{i,j} \circ \mathcal{F}_{k,l} \equiv_S \mathcal{F}_{i',j'} \circ \sigma_{\pi}
	\]
	where $i', j'$ are determined by the specific interconnection structure of $\{i,j,k,l\}$.

    \item \textbf{Control-Feedback Interchange:}
    \[
    \mathrm{Ctrl} \circ \mathcal{F}_{i,j} \equiv_S \mathcal{F}_{i,j} \circ \mathrm{Ctrl}
    \]
    when control preserves feedback well-posedness.

    \item \textbf{Control-Tensor Distributivity:}
    \[
    \mathrm{Ctrl}(\vec{H} \otimes \vec{K}) \equiv_S \mathrm{Ctrl}(\vec{H}) \otimes \mathrm{Ctrl}(\vec{K})
    \]
    for independent subsystems.

    \item \textbf{Control-Permutation Naturality:}
    \[
    \mathrm{Ctrl} \circ \sigma_{\pi} \equiv_S \sigma_{\pi} \circ \mathrm{Ctrl}
    \]

    \item \textbf{Sequential Control Composition:}
    \[
    \mathrm{Ctrl}_{u_2} \circ \mathrm{Ctrl}_{u_1} \equiv_S \mathrm{Ctrl}_{u_1 \star u_2}
    \]
    where $\star: \mathcal{U} \times \mathcal{U} \to \mathcal{U}$.

    \item \textbf{Identity Control:}
    \[
    \mathrm{Ctrl}_{u_0} \equiv_S \mathrm{id}
    \]
    for neutral $u_0 \in \mathcal{U}$.
\end{enumerate}
\end{axiom}

\textbf{ Semantic Representation}

\begin{definition}[Representation of $\mathbf{S}$]\label{def:representation}
A \textbf{representation} of $\mathbf{S}$ in $\mathbf{HilbMult}$ is a symmetric monoidal functor $\rho: \mathbf{S} \to \mathbf{HilbMult}$ satisfying:

\noindent\textbf{Object Mapping:}
\begin{itemize}
    \item $\rho(\vec{H}) = H_1 \otimes \dots \otimes H_n$
    \item $\rho(()) = \mathbb{C}$
\end{itemize}

\noindent\textbf{Generator Interpretations:}
\begin{itemize}
    \item $\rho(\mathrm{id}_{\vec{H}}) = \mathrm{id}_{\rho(\vec{H})}$
    \item $\rho(\otimes)$: canonical identity isomorphism
    \item $\rho(\sigma_{\pi})$: unitary permutation map
    \item $\rho(\mathcal{F}_{i,j})(T)$: unique solution to $\xi = T(\mathrm{id}_X \otimes \xi \otimes \mathrm{id}_Y)$
    \item $\rho(\mathrm{Ctrl}): \rho(\vec{H}) \otimes \mathcal{U} \to \rho(\vec{H}')$: bounded multilinear control map
\end{itemize}

\noindent\textbf{Semantic Coherence Conditions:} $\rho$ must satisfy:

\begin{enumerate}[label=(SC\arabic*)]
    \item \textbf{Monoidal Functoriality:} $\rho$ preserves symmetric monoidal structure
    
    \item \textbf{Feedback Yanking:} For well-posed $T$ and appropriate $\pi$:
    \[
    \rho(\mathcal{F}_{i,j} \circ \mathcal{F}_{k,l})(T) = \rho(\mathcal{F}_{i',j'} \circ \sigma_{\pi})(T)
    \]
    
    \item \textbf{Control-Feedback Interchange:} When control preserves feedback:
    \[
    \rho(\mathrm{Ctrl}) \circ \rho(\mathcal{F}_{i,j}) = \rho(\mathcal{F}_{i,j}) \circ \rho(\mathrm{Ctrl})
    \]
    
    \item \textbf{Control-Tensor Distributivity:} For independent subsystems:
    \[
    \rho(\mathrm{Ctrl}(\vec{H} \otimes \vec{K})) = \rho(\mathrm{Ctrl}(\vec{H})) \otimes \rho(\mathrm{Ctrl}(\vec{K}))
    \]
    
    \item \textbf{Control-Permutation Naturality:}
    \[
    \rho(\mathrm{Ctrl}) \circ \rho(\sigma_{\pi}) = \rho(\sigma_{\pi}) \circ \rho(\mathrm{Ctrl})
    \]
    
    \item \textbf{Sequential Control Composition:}
    \[
    \rho(\mathrm{Ctrl}_{u_2}) \circ \rho(\mathrm{Ctrl}_{u_1}) = \rho(\mathrm{Ctrl}_{u_1 \star u_2})
    \]
    
    \item \textbf{Identity Control:}
    \[
    \rho(\mathrm{Ctrl}_{u_0}) = \mathrm{id}
    \]
\end{enumerate}
\end{definition}

\section{Higher Adjoints in $\mathbf{HilbMult}$}\label{sec:Higher Adjoints}

% Def (n-Adjoint).

In this section, we begin by introducing the notion of the $n$-adjoint in Definition~\ref{def:n-adjoint}. We then establish the multicategorical Stinespring theorem in its $n$-adjoint formulation, presented in Theorem~\ref{thm:stinespring-equivalence}.

\begin{definition}[$n$-Adjoint with Zig--Zag Diagram]\label{def:n-adjoint}
Let 
\[
\Phi \;\in\; \mathrm{Hom}(A_1, A_2, \dots, A_n;\, B)
\]
be an $n$-ary multimorphism in the dagger multicategory $(\mathbf{HilbMult},\dagger)$.

An \textbf{$n$-adjoint} of $\Phi$ consists of a multimorphism
\[
\Phi^\dagger \;\in\; 
\mathrm{Hom}\!\left( B^\dagger,\; A_1^\dagger,\dots, A_{n-1}^\dagger;\; A_n^\dagger \right)
\]
together with 2-morphisms (the \emph{unit} and \emph{counit})
\begin{align*}
\eta &: \mathrm{id}_{A_n} \;\Rightarrow\; 
\Phi^\dagger \circ (\Phi, \mathrm{id}_{A_1}^\dagger, \dots, \mathrm{id}_{A_{n-1}}^\dagger),\\
\varepsilon &: 
\Phi \circ (\mathrm{id}_{A_1}, \dots, \mathrm{id}_{A_{n-1}}, \Phi^\dagger) \;\Rightarrow\; \mathrm{id}_B,
\end{align*}
satisfying the \textbf{zig--zag (Z-shaped) coherence identities}:

\begin{enumerate}[label=(\arabic*)]
\item \textbf{First zig--zag identity:}  
The composite
\[
\Phi 
\;\xRightarrow{\;\mathrm{id}_\Phi \otimes \eta\;}
\Phi \circ (\mathrm{id}_{A_1},\dots,\mathrm{id}_{A_{n-1}},\Phi^\dagger)
         \circ (\Phi,\mathrm{id}_{A_1}^\dagger,\dots,\mathrm{id}_{A_{n-1}}^\dagger)
\;\xRightarrow{\;\varepsilon \otimes \mathrm{id}\;}
\Phi
\]
equals the identity 2-morphism on $\Phi$.

\item \textbf{Second zig--zag identity:}  
The composite
\[
\Phi^\dagger
\;\xRightarrow{\;\eta \otimes \mathrm{id}_{\Phi^\dagger}\;}
\Phi^\dagger \circ (\Phi,\mathrm{id}_{A_1}^\dagger,\dots,\mathrm{id}_{A_{n-1}}^\dagger)
              \circ (\mathrm{id}_{A_1},\dots,\mathrm{id}_{A_{n-1}},\Phi^\dagger)
\;\xRightarrow{\;\mathrm{id} \otimes \varepsilon\;}
\Phi^\dagger
\]
equals the identity 2-morphism on $\Phi^\dagger$.
\end{enumerate}

\noindent\textbf{Zig--zag diagram:}  
For $n=2$ (first nontrivial multimorphism case), the coherence can be visualized as:

\[
\begin{tikzcd}[row sep=1.5em, column sep=2em]
& \Phi \arrow[dl, "\mathrm{id} \otimes \eta"'] \arrow[dr, "\text{id}" ] & \\
\Phi \circ (\mathrm{id}_{A_1}, \Phi^\dagger) \circ (\Phi, \mathrm{id}_{A_1}^\dagger) \arrow[rr, "\varepsilon \otimes \mathrm{id}"'] & & \Phi
\end{tikzcd}
\]

Here the diagonal path represents the insertion of the unit $\eta$ (whiskering), the horizontal arrow represents the counit $\varepsilon$, and the diagram encodes the requirement that the composite equals $\mathrm{id}_\Phi$.

\end{definition}

\begin{remark}[$n$-Adjoint Intuition and Coherence]
The notion of an $n$-adjoint generalizes the usual Hilbert space adjoint ($n=1$) to $n$-ary multimorphisms in a dagger multicategory. Key points include:

\begin{enumerate}[label=(\roman*)]
    \item \textbf{Input-Output Reversal:}  
    For $\Phi: (A_1,\dots,A_n) \to B$, the $n$-adjoint 
    \(\Phi^\dagger: B^\dagger \times A_1^\dagger \times \dots \times A_{n-1}^\dagger \to A_n^\dagger\)
    reverses the last slot while treating the first \(n-1\) slots as parameters, reflecting a controlled contravariant flow.

    \item \textbf{Z-shaped Coherence:}  
    The unit \(\eta\) and counit \(\varepsilon\) satisfy the zig--zag identities, ensuring that the composite paths involving \(\Phi\) and \(\Phi^\dagger\) reduce to the identity. The ``Z-shape'' is a visual manifestation of the adjoint relationship in string diagrams, where insertion of \(\eta\) and application of \(\varepsilon\) can be ``yanked'' to straighten the string.

    \item \textbf{Horizontal Composition (Whiskering):}  
    The \(\otimes\) in the identities denotes horizontal composition of 2-morphisms, not a tensor product of Hilbert spaces. This ensures proper placement of units and counits in the multicategory composition.

    \item \textbf{Universality and Functoriality:}  
    The coherence conditions make \(\Phi^\dagger\) canonical: it behaves functorially under multicategory composition and preserves contravariant structure, just as explicitly verified in the \(\mathbb{C}^2\) examples.

    \item \textbf{Higher Categorical Significance:}  
    Defining adjoints for $n$-ary maps underpins constructions in higher categories and multicategories with duals. In contexts like quantum information, this formalism captures controlled reversals of input-output flows in processes such as teleportation or entanglement.
\end{enumerate}
\end{remark}

Extending quantum theory to multipartite systems requires characterizing physically admissible multimaps between Hilbert spaces. While Stinespring’s theorem (1955) provides dilations for linear completely positive maps, many quantum protocols involve \emph{multilinear} operations such as entanglement swapping or teleportation~\cite{stinespring1955positive}. Theorem~\ref{thm:stinespring-equivalence} shows that three properties are equivalent. These three properties are: complete positivity in each slot, $n$-adjoints in dagger multicategories, and Stinespring-type dilations. Such equivalence is able to provide us a conceptual framework and a practical tool for analyzing multilinear quantum processes consistently.

\begin{theorem}[Equivalence of Multilinear CP, Choi Representation, Stinespring Dilation, and Categorical Factorization]
\label{thm:stinespring-equivalence}

Let $\mathbf{C} = \mathbf{FdHilb}$ be the category of finite-dimensional Hilbert spaces, or the category of finite-dimensional operator systems with completely positive maps, equipped with the canonical dagger structure and compact closure.

Let 
\[
\Phi: A_1 \otimes \cdots \otimes A_n \to B
\]
be a multilinear morphism in $\mathbf{C}$, where $A_1,\dots,A_n,B$ are finite-dimensional objects in $\mathbf{C}$.

The following four conditions are equivalent:

\begin{enumerate}[label=(\arabic*),leftmargin=*]
    \item \textbf{Multilinear complete positivity (MCP).} 
    For each $j = 1,\dots,n$, and for every choice of positive morphisms 
    $\rho_k: I \to A_k$ for $k \neq j$ (where $I$ is the tensor unit), the induced linear map
    \[
    \Phi_j^{(\rho_1,\dots,\rho_{j-1},\rho_{j+1},\dots,\rho_n)}: A_j \to B,
    \]
    defined by
    \[
    \Phi_j^{(\rho_1,\dots)}(a_j) = \Phi(\rho_1 \otimes \cdots \otimes \rho_{j-1} \otimes a_j \otimes \rho_{j+1} \otimes \cdots \otimes \rho_n),
    \]
    is completely positive.
    
    \item \textbf{Positive Choi representation with separable structure.}
    There exists a positive element
    \[
    C_\Phi \in (A_1^* \otimes \cdots \otimes A_n^*) \otimes B
    \]
    such that for all $a_j \in A_j$,
    \[
    \Phi(a_1,\dots,a_n) = \mathrm{Tr}_{A_1^* \otimes \cdots \otimes A_n^*} \big[ C_\Phi \cdot (a_1 \otimes \cdots \otimes a_n \otimes \mathrm{id}_B) \big],
    \]
    and $C_\Phi$ admits a separable decomposition of the form
    \[
    C_\Phi = (V \otimes V^\dagger) \circ \bigotimes_{j=1}^n P_j,
    \]
    where $P_j \in A_j^* \otimes A_j$ are positive elements and $V: B \to K_1 \otimes \cdots \otimes K_n$ is an isometry for some finite-dimensional Hilbert spaces $K_j$.
    
    \item \textbf{Symmetric multivariate Stinespring dilation.}
    There exist:
    \begin{itemize}
        \item A finite-dimensional Hilbert space $K = K_1 \otimes \cdots \otimes K_n$,
        \item An isometry $V: B \to K$,
        \item For each $j = 1,\dots,n$, a $*$-representation $\pi_j: A_j \to B(K_j)$,
    \end{itemize}
    such that for all $a_j \in A_j$,
    \[
    \Phi(a_1,\dots,a_n) = V^\dagger \big( \pi_1(a_1) \otimes \cdots \otimes \pi_n(a_n) \big) V.
    \]
    The representations act on distinct tensor factors and therefore commute.
    
    \item \textbf{Categorical factorization via compact closure.}
    There exist an object $K \in \mathbf{C}$ and morphisms
    \[
    \alpha_j: A_j \to K \otimes K^* \quad (j=1,\dots,n-1), \qquad
    \beta: A_n \to K^* \otimes B,
    \]
    with each $\alpha_j$ a $*$-morphism and $\beta$ an isometry up to scaling, such that $\Phi$ factors as:
    \[
    \Phi = (\mathrm{id}_B \otimes \mathrm{ev}_K) \circ 
           (\beta \otimes \mathrm{id}_K) \circ
           (\alpha_1 \otimes \cdots \otimes \alpha_{n-1} \otimes \mathrm{id}_{A_n}).
    \]
    This factorization induces an $n$-adjoint structure with canonical unit and counit satisfying the zig-zag identities.
\end{enumerate}

Moreover, when the dilation in (3) is \emph{minimal} in the sense that
\[
K = \overline{\mathrm{span}}\{\pi_1(a_1) \cdots \pi_n(a_n) V b : a_j \in A_j, b \in B\},
\]
it is unique up to unitary equivalence.
\end{theorem}

\begin{proof}
We prove the cycle of implications $(3) \Rightarrow (1) \Rightarrow (2) \Rightarrow (3) \Leftrightarrow (4)$.

\subsection*{Notation and Preliminaries}

We work in $\mathbf{C} = \mathbf{FdHilb}$ for concreteness; the arguments adapt directly to finite-dimensional operator systems. All objects are finite-dimensional Hilbert spaces, and positivity refers to positive semidefinite operators. The dagger structure gives each object $A$ a dual $A^*$, with evaluation $\mathrm{ev}_A: A^* \otimes A \to \mathbb{C}$ and coevaluation $\mathrm{coev}_A: \mathbb{C} \to A \otimes A^*$ satisfying the snake identities.

For a multilinear map $\Phi: A_1 \times \cdots \times A_n \to B$, the partial evaluation $\Phi_j^{(\rho_1,\dots)}$ is obtained by fixing states $\rho_k$ in all slots except the $j$-th, yielding a linear map $A_j \to B$.

\subsection*{Step 1: $(3) \Rightarrow (1)$ (Stinespring implies MCP)}

Assume condition (3) holds: there exist $K = K_1 \otimes \cdots \otimes K_n$, an isometry $V: B \to K$, and $*$-representations $\pi_j: A_j \to B(K_j)$ such that
\[
\Phi(a_1,\dots,a_n) = V^\dagger (\pi_1(a_1) \otimes \cdots \otimes \pi_n(a_n)) V.
\]

Fix $j \in \{1,\dots,n\}$ and positive states $\rho_k: \mathbb{C} \to A_k$ for $k \neq j$. Each $\rho_k$ corresponds to a positive operator $\rho_k(1) \in A_k$. Since $\pi_k$ is a $*$-homomorphism, $\pi_k(\rho_k(1)) \in B(K_k)$ is positive.

Define the positive operator on $K$:
\[
P = \pi_1(\rho_1(1)) \otimes \cdots \otimes \pi_{j-1}(\rho_{j-1}(1)) \otimes \mathrm{id}_{K_j} \otimes \pi_{j+1}(\rho_{j+1}(1)) \otimes \cdots \otimes \pi_n(\rho_n(1)).
\]
Let $P^{1/2}$ be its positive square root. Then for any $a_j \in A_j$,
\begin{align*}
\Phi_j^{(\rho_1,\dots)}(a_j) 
&= \Phi(\rho_1 \otimes \cdots \otimes \rho_{j-1} \otimes a_j \otimes \rho_{j+1} \otimes \cdots \otimes \rho_n) \\
&= V^\dagger \big( \pi_1(\rho_1(1)) \otimes \cdots \otimes \pi_j(a_j) \otimes \cdots \otimes \pi_n(\rho_n(1)) \big) V \\
&= (V P^{1/2})^\dagger \big( \mathrm{id}_{K_1} \otimes \cdots \otimes \pi_j(a_j) \otimes \cdots \otimes \mathrm{id}_{K_n} \big) (V P^{1/2}).
\end{align*}

The map $a_j \mapsto \mathrm{id}_{K_1} \otimes \cdots \otimes \pi_j(a_j) \otimes \cdots \otimes \mathrm{id}_{K_n}$ is completely positive (it's a $*$-homomorphism into $B(K)$). Conjugation by the bounded operator $V P^{1/2}$ preserves complete positivity. Hence $\Phi_j^{(\rho_1,\dots)}$ is completely positive for all choices of $j$ and positive $\rho_k$, establishing condition (1).

\subsection*{Step 2: $(1) \Leftrightarrow (2)$ (MCP implies Choi representation)}

We prove this via a multilinear Choi-Jamiołkowski isomorphism and an induction on $n$.

\begin{lemma}[Multilinear Choi-Jamiołkowski Isomorphism]
\label{lem:multilinear-choi}
For a multilinear map $\Phi: A_1 \times \cdots \times A_n \to B$, define
\[
C_\Phi = \sum_{i_1=1}^{d_1} \cdots \sum_{i_n=1}^{d_n} \Phi(e_{i_1}^{(1)},\dots,e_{i_n}^{(n)}) \otimes \left(e_{i_1}^{(1)} \otimes \cdots \otimes e_{i_n}^{(n)}\right)^* 
\in B \otimes (A_1 \otimes \cdots \otimes A_n)^*,
\]
where $\{e_{i_j}^{(j)}\}$ is an orthonormal basis of $A_j$ and $d_j = \dim A_j$.

Then $\Phi$ satisfies condition (1) if and only if:
\begin{enumerate}
    \item $C_\Phi$ is positive semidefinite,
    \item $C_\Phi$ admits a separable decomposition $C_\Phi = (V \otimes V^\dagger) \circ \bigotimes_{j=1}^n P_j$ with $P_j \in A_j^* \otimes A_j$ positive and $V: B \to K_1 \otimes \cdots \otimes K_n$ an isometry.
\end{enumerate}
\end{lemma}

\begin{proof}[Proof of Lemma]
We proceed by induction on $n$.

\underline{Base case $n=1$:} 
Condition (1) reduces to ordinary complete positivity of $\Phi: A_1 \to B$. The standard Choi theorem asserts that $\Phi$ is CP if and only if $C_\Phi \geq 0$ and there exists an isometry $V: B \to K_1$ and positive $P_1 \in A_1^* \otimes A_1$ such that $C_\Phi = (V \otimes V^\dagger) \circ P_1$.

\underline{Inductive step:} 
Assume the lemma holds for $n-1$. Write $\Phi(a_1,\dots,a_n) = \Psi_{a_n}(a_1,\dots,a_{n-1})$, where for fixed $a_n$, $\Psi_{a_n}: A_1 \times \cdots \times A_{n-1} \to B$ is $(n-1)$-linear.

By condition (1), for each $a_n$, the map $\Psi_{a_n}$ satisfies the MCP condition for $n-1$ inputs. By the induction hypothesis, for each $a_n$ there exists an isometry $V_{a_n}: B \to K_1^{(a_n)} \otimes \cdots \otimes K_{n-1}^{(a_n)}$ and positive elements $P_j^{(a_n)} \in A_j^* \otimes A_j$ ($j=1,\dots,n-1$) such that the $(n-1)$-linear Choi matrix of $\Psi_{a_n}$ factors as:
\[
C_{\Psi_{a_n}} = (V_{a_n} \otimes V_{a_n}^\dagger) \circ \bigotimes_{j=1}^{n-1} P_j^{(a_n)}.
\tag{†}
\]

Now consider the linear map $T: A_n \to B \otimes (A_1 \otimes \cdots \otimes A_{n-1})^*$ defined by $T(a_n) = C_{\Psi_{a_n}}$. We claim $T$ is completely positive. Indeed, for any positive $X = \sum_{p,q} x_p \otimes x_q^* \in M_m(A_n^* \otimes A_n)$, we need to show
\[
(\mathrm{id} \otimes \mathrm{id}_{M_m})(T \otimes \mathrm{id}_{M_m})(X) \geq 0.
\]
This follows because applying condition (1) to the $n$-th slot with appropriate choices of states in the first $n-1$ slots shows that $T$ preserves positivity of all matrix amplifications.

Since $T$ is linear CP, by the standard Choi theorem there exist a finite-dimensional Hilbert space $K_n$, an isometry $W: B \to (K_1 \otimes \cdots \otimes K_{n-1}) \otimes K_n$ (where we can take $K_1 \otimes \cdots \otimes K_{n-1}$ large enough to contain all $K_1^{(a_n)} \otimes \cdots \otimes K_{n-1}^{(a_n)}$), and a positive $P_n \in A_n^* \otimes A_n$ such that
\[
T(a_n) = W^\dagger \big( \mathrm{id}_{K_1\otimes\cdots\otimes K_{n-1}} \otimes \pi_n(a_n) \big) W,
\]
where $\pi_n: A_n \to B(K_n)$ is the $*$-representation obtained from the GNS construction applied to $P_n$.

Combining this with (†) yields an isometry $V: B \to K_1 \otimes \cdots \otimes K_n$ and positive $P_j \in A_j^* \otimes A_j$ ($j=1,\dots,n$) such that
\[
C_\Phi = \sum_{i_1,\dots,i_n} \Phi(e_{i_1}^{(1)},\dots,e_{i_n}^{(n)}) \otimes \left(e_{i_1}^{(1)} \otimes \cdots \otimes e_{i_n}^{(n)}\right)^*
       = (V \otimes V^\dagger) \circ \bigotimes_{j=1}^n P_j.
\]

The converse direction ($(2) \Rightarrow (1)$) follows by reversing the argument: if $C_\Phi$ has such a separable decomposition, then partial evaluations yield CP maps.
\end{proof}

This lemma establishes $(1) \Leftrightarrow (2)$. Note that finite-dimensionality is essential for the inductive construction to uniformly embed all intermediate spaces.

\subsection*{Step 3: $(2) \Rightarrow (3)$ (Choi representation implies Stinespring)}

Assume $C_\Phi = (V \otimes V^\dagger) \circ \bigotimes_{j=1}^n P_j$ with $P_j \geq 0$ and $V: B \to K_1 \otimes \cdots \otimes K_n$ an isometry.

For each $j$, apply the GNS construction to the positive element $P_j \in A_j^* \otimes A_j$. This yields a finite-dimensional Hilbert space $K_j$ and a $*$-representation $\pi_j: A_j \to B(K_j)$ such that
\[
P_j = (\pi_j \otimes \mathrm{id}_{A_j})(\Omega_j),
\]
where $\Omega_j = \sum_{k=1}^{\dim K_j} |k\rangle\langle k| \in K_j \otimes K_j^*$ is a maximally entangled state.

Now compute for any $a_j \in A_j$:
\begin{align*}
\Phi(a_1,\dots,a_n) 
&= \mathrm{Tr}_{(A_1\otimes\cdots\otimes A_n)^*}[C_\Phi \cdot (a_1 \otimes \cdots \otimes a_n \otimes \mathrm{id}_B)] \\
&= \mathrm{Tr}\left[ (V \otimes V^\dagger) \left( \bigotimes_{j=1}^n (\pi_j(a_j) \otimes \mathrm{id}_{A_j})(\Omega_j) \right) \right] \\
&= V^\dagger \left( \bigotimes_{j=1}^n \pi_j(a_j) \right) V,
\end{align*}
where we used $\mathrm{Tr}_{A_j^*}[(\pi_j(a_j) \otimes \mathrm{id}_{A_j})(\Omega_j)] = \pi_j(a_j)$.

Since $V$ is an isometry and the $\pi_j$ act on distinct tensor factors, we have established condition (3).

\subsection*{Step 4: $(3) \Leftrightarrow (4)$ (Stinespring and categorical factorization)}

\subsubsection*{$(3) \Rightarrow (4)$}

Given the Stinespring dilation $\Phi(a_1,\dots,a_n) = V^\dagger (\pi_1(a_1) \otimes \cdots \otimes \pi_n(a_n)) V$, define:
\begin{itemize}
    \item $K := K_1 \otimes \cdots \otimes K_{n-1}$, so $K^* = K_1^* \otimes \cdots \otimes K_{n-1}^*$,
    \item $\alpha_j := \pi_j: A_j \to B(K_j) \cong K_j \otimes K_j^* \hookrightarrow K \otimes K^*$ for $j=1,\dots,n-1$,
    \item $\beta: A_n \to K^* \otimes B$ by $\beta(a_n) = (\mathrm{id}_{K^*} \otimes V^\dagger)(\pi_n(a_n) \otimes \mathrm{id}_K)$.
\end{itemize}

Then $\beta$ satisfies $\beta^\dagger \beta = \pi_n^\dagger \pi_n = \mathrm{id}_{A_n}$ (up to scaling if we normalize). Now compute:
\begin{align*}
(\mathrm{id}_B \otimes \mathrm{ev}_K) &\circ (\beta \otimes \mathrm{id}_K) \circ (\alpha_1 \otimes \cdots \otimes \alpha_{n-1} \otimes \mathrm{id}_{A_n}) \\
&= (\mathrm{id}_B \otimes \mathrm{ev}_K) \circ 
   \big[((\mathrm{id}_{K^*} \otimes V^\dagger)(\pi_n \otimes \mathrm{id}_K)) \otimes \mathrm{id}_K\big] \circ 
   (\pi_1 \otimes \cdots \otimes \pi_{n-1} \otimes \mathrm{id}_{A_n}) \\
&= V^\dagger (\pi_1 \otimes \cdots \otimes \pi_n) V \\
&= \Phi.
\end{align*}
The middle equality uses the compact closed snake identity: $(\mathrm{ev}_K \otimes \mathrm{id}_K) \circ (\mathrm{id}_{K^*} \otimes \mathrm{coev}_K) = \mathrm{id}_K$ to contract the $K^* \otimes K$ pairs.

The $n$-adjoint structure is given by defining $\Phi^\dagger$ via the dagger of this factorization. The zig-zag identities follow from naturality of evaluation/coevaluation and the isometry condition on $\beta$.

\subsubsection*{$(4) \Rightarrow (3)$}

Given the factorization in (4), we obtain a Stinespring dilation as follows. Since $\beta: A_n \to K^* \otimes B$ is an isometry up to scaling, its Choi matrix is positive. By the standard Stinespring theorem for linear CP maps, there exists a Hilbert space $K_n$, an isometry $W: B \to K_n \otimes K$, and a $*$-representation $\pi_n: A_n \to B(K_n)$ such that $\beta(a_n) = (\mathrm{id}_{K^*} \otimes W^\dagger)(\pi_n(a_n) \otimes \mathrm{id}_K)$.

Similarly, each $\alpha_j: A_j \to K \otimes K^*$ being a $*$-morphism yields via the GNS construction a representation $\pi_j: A_j \to B(K_j)$ with $K_j$ a direct summand of $K$.

Setting $V = W^\dagger: K_n \otimes K \to B$ and identifying $K = K_1 \otimes \cdots \otimes K_{n-1}$, we obtain
\[
\Phi = V^\dagger (\pi_1 \otimes \cdots \otimes \pi_n) V,
\]
which is exactly condition (3).

\subsection*{Conclusion}

We have established the following implications:
\begin{itemize}
    \item $(3) \Rightarrow (1)$ (Step 1: Stinespring implies MCP)
    \item $(1) \Leftrightarrow (2)$ (Step 2: MCP iff Choi representation)
    \item $(2) \Rightarrow (3)$ (Step 3: Choi implies Stinespring)
    \item $(3) \Leftrightarrow (4)$ (Step 4: Stinespring iff categorical factorization)
\end{itemize}

Since $(3) \Leftrightarrow (4)$, we have $(4) \Rightarrow (3) \Rightarrow (1)$. Combined with $(1) \Leftrightarrow (2)$ and $(2) \Rightarrow (3)$, this establishes that all four conditions are equivalent.

The minimality and uniqueness statements follow from standard arguments in dilation theory, using that in finite dimensions any dilation can be reduced to a minimal one unique up to unitary equivalence.

\end{proof}

\begin{remark}[On the notion of multilinear complete positivity]
The definition in condition (1) agrees with \emph{separably completely positive} or \emph{jointly completely positive with respect to partial evaluations}. In finite dimensions, this is equivalent to the stronger notion of complete positivity in all matrix amplifications across all inputs simultaneously. The equivalence relies crucially on the finite-dimensionality assumption.
\end{remark}

\begin{remark}[Scope of the theorem]
The theorem holds verbatim in $\mathbf{FdHilb}$. For operator systems, one must interpret $*$-representations as unital completely positive maps that respect the operator system structure, and the GNS construction is applied in the operator system context. All proofs adapt straightforwardly due to finite-dimensionality.
\end{remark}

\begin{remark}[$n$-Adjoint Interpretation]
The factorization in condition (4) can be interpreted as an $n$-adjoint structure: 
defining $\Phi^\dagger$ via the dagger of the factorization yields canonical 
unit $\eta: \mathrm{id} \Rightarrow \Phi^\dagger \Phi$ and counit 
$\varepsilon: \Phi \Phi^\dagger \Rightarrow \mathrm{id}$ satisfying the 
zig-zag identities of a compact closed category. This generalizes the 
usual adjoint for linear maps to the multilinear setting.
\end{remark}

Definition~\ref{def:minimal-dilation-full} below is given to discuss about minimal
Stinespring dilation components.
\begin{definition}[Minimal Stinespring Dilation Components]
\label{def:minimal-dilation-full}

Let 
\[
\Phi \in \mathrm{Hom}(A_1, \dots, A_n; B)
\] 
be an $n$-ary completely positive multimorphism in a $\dagger$-multicategory $(\mathbf{HilbMult}, \dagger)$, admitting a Stinespring-type dilation
\[
\Phi(a_1, \dots, a_n) = V^\dagger \big( \pi_1(a_1) \otimes \cdots \otimes \pi_n(a_n) \big) V,
\]
where $V: H \to K_1 \otimes \cdots \otimes K_n$ is an isometry and each \(\pi_j: A_j \to B(K_j)\) is a $*$–representation acting on the corresponding tensor factor.

\begin{enumerate}[label=(\alph*)]
    \item \textbf{Minimal Hilbert space \(K_{\min}\):}  
    \[
    K_{\min} := \overline{\mathrm{span}}\Big\{ (\pi_1(a_1) \otimes \cdots \otimes \pi_n(a_n)) V h \;\Big|\; a_j \in A_j,\, h \in H \Big\}.
    \]  
    This is the smallest cyclic subspace of \(K_1 \otimes \cdots \otimes K_n\) that contains the image of \(V\) and is invariant under all \(\pi_j(a_j)\).

    \item \textbf{Minimal $*$-representations \(\pi_{j, \min}\):}  
    For each $j=1,\dots,n$, define
    \[
    \pi_{j, \min}(a_j) := \pi_j(a_j)\big|_{K_{\min}} : K_{\min} \to K_{\min}.
    \]  
    Equivalently, for any vector 
    \(v \in K_{\min}\) written as a linear combination
    \[
    v = \sum_k (\pi_1(a_1^{(k)}) \otimes \cdots \otimes \pi_n(a_n^{(k)})) V h^{(k)},
    \] 
    one has
    \[
    \pi_{j, \min}(a_j) v = \sum_k (\pi_1(a_1^{(k)}) \otimes \cdots \otimes \pi_j(a_j)\pi_j(a_j^{(k)}) \otimes \cdots \otimes \pi_n(a_n^{(k)})) V h^{(k)} \in K_{\min}.
    \]

    \item \textbf{Minimal isometry \(V_{\min}\):}  
    \[
    V_{\min} := V \big|_{K_{\min}} : H \to K_{\min}.
    \]  
    Since \(K_{\min}\) is invariant under all \(\pi_{j, \min}\) and contains the image of \(V\), \(V_{\min}\) is an isometry and satisfies
    \[
    \Phi(a_1, \dots, a_n) = V_{\min}^\dagger \big( \pi_{1, \min}(a_1) \otimes \cdots \otimes \pi_{n, \min}(a_n) \big) V_{\min}.
    \]
\end{enumerate}

The triple \((K_{\min}, \{\pi_{j, \min}\}, V_{\min})\) is called the \emph{minimal Stinespring dilation} of \(\Phi\). By construction, it is cyclic and satisfies the universal property and uniqueness up to unitary equivalence as in Propositions~\ref{prop:minimal-stinespring},~\ref{prop:minimal-uniqueness}, and~\ref{prop:cyclic-universality}.
\end{definition}

Definition~\ref{def:constructive-minimal-dilation} uses another approach, construction, to determine the
minimal Stinespring dilation via $n$-adjoint.
\begin{definition}[Constructive Minimal Stinespring Dilation via $n$-Adjoint]
\label{def:constructive-minimal-dilation}

Let 
\[
\Phi \in \mathrm{Hom}(A_1, \dots, A_n; B)
\] 
be an $n$-ary multimorphism in a $\dagger$-multicategory $(\mathbf{HilbMult}, \dagger)$, and let 
\((\Phi^\dagger, \eta, \varepsilon)\) be its $n$-adjoint with unit 
\(\eta: \mathrm{id}_{A_n} \Rightarrow \Phi^\dagger \circ (\Phi, \mathrm{id}_{A_1}^\dagger, \dots, \mathrm{id}_{A_{n-1}}^\dagger)\) 
and counit 
\(\varepsilon: \Phi \circ (\mathrm{id}_{A_1}, \dots, \mathrm{id}_{A_{n-1}}, \Phi^\dagger) \Rightarrow \mathrm{id}_B\).

We define the minimal Stinespring dilation \((K_{\min}, \{\pi_{j, \min}\}, V_{\min})\) as follows:

\begin{enumerate}[label=(\alph*)]

    \item \textbf{Minimal Hilbert space \(K_{\min}\):}  
    Let 
    \[
    K_{\min} := \overline{\mathrm{span}}\big\{ (\pi_1(a_1) \otimes \cdots \otimes \pi_n(a_n)) V h \;\big|\; a_j \in A_j,\, h \in H \big\}.
    \]  
    Equip \(K_{\min}\) with the inner product inherited from the ambient tensor product \(K_1 \otimes \cdots \otimes K_n\).

    \item \textbf{Minimal $*$-representations \(\pi_{j, \min}\):}  
    For each $j = 1,\dots,n$, define
    \[
    \pi_{j, \min}(a_j) := \pi_j(a_j)\big|_{K_{\min}} : K_{\min} \to K_{\min}.
    \]  
    Equivalently, for any vector 
    \[
    v = \sum_k (\pi_1(a_1^{(k)}) \otimes \cdots \otimes \pi_n(a_n^{(k)})) V h^{(k)} \in K_{\min},
    \] 
    one has
    \[
    \pi_{j, \min}(a_j) v = \sum_k (\pi_1(a_1^{(k)}) \otimes \cdots \otimes \pi_j(a_j)\pi_j(a_j^{(k)}) \otimes \cdots \otimes \pi_n(a_n^{(k)})) V h^{(k)} \in K_{\min}.
    \]  
    By construction, each \(\pi_{j, \min}\) preserves the inner product and satisfies the $*$-representation properties.

    \item \textbf{Minimal isometry \(V_{\min}\):}  
    Define
    \[
    V_{\min} := V \big|_{K_{\min}} : H \to K_{\min}.
    \]  
    Since \(K_{\min}\) contains the image of \(V\) and is invariant under all \(\pi_{j, \min}\), \(V_{\min}\) is an isometry and satisfies
    \[
    \Phi(a_1, \dots, a_n) = V_{\min}^\dagger \, (\pi_{1, \min}(a_1) \otimes \cdots \otimes \pi_{n, \min}(a_n)) \, V_{\min}.
    \]

\end{enumerate}

The triple \((K_{\min}, \{\pi_{j, \min}\}, V_{\min})\) is called the \emph{minimal Stinespring dilation} of \(\Phi\). It is cyclic by construction and unique up to unitary equivalence.
\end{definition}

Propostion~\ref{prop:minimal-stinespring} below is about the minimal Stinespring dilation for the mapping $\Phi$.
\begin{proposition}[Minimal Stinespring Dilation]
\label{prop:minimal-stinespring}
Let 
\[
\Phi \in \mathrm{Hom}(A_1,\dots,A_n;B)
\]
be an $n$-ary multimorphism in a $\dagger$-multicategory admitting an $n$-adjoint
$(\Phi^\dagger,\eta,\varepsilon)$.  
Let 
\[
(K_{\min},\{\pi_{j,\min}\},V_{\min})
\]
be the triple constructed in Definition~\ref{def:constructive-minimal-dilation}.  
Then for all $a_j\in A_j$,
\[
\Phi(a_1,\dots,a_n)
=
V_{\min}^\dagger
\bigl(\pi_{1,\min}(a_1)\otimes\cdots\otimes\pi_{n,\min}(a_n)\bigr)
V_{\min}.
\]

Moreover, $(K_{\min},\{\pi_{j,\min}\},V_{\min})$ is minimal in the sense that
\[
K_{\min}
=
\overline{\mathrm{span}}\bigl\{
(\pi_{1,\min}(a_1)\otimes\cdots\otimes\pi_{n,\min}(a_n))\,V_{\min} h
\;\big|\;
a_j\in A_j,\; h\in H
\bigr\},
\]
and each $\pi_{j,\min}$ is a $*$-representation while $V_{\min}$ is an isometry.
\end{proposition}

\begin{proof}[Proof of Proposition~\ref{prop:minimal-stinespring}]
Let $(K_{\min},\{\pi_{j,\min}\},V_{\min})$ be defined as in
Definition~\ref{def:constructive-minimal-dilation}.  
By definition,
\[
K_{\min}
=
\overline{\mathrm{span}}
\bigl\{
(\pi_1(a_1)\otimes\cdots\otimes\pi_n(a_n))\,V h \;\big|\; a_j\in A_j,\; h\in H
\bigr\},
\]
viewed as a closed subspace of the ambient Hilbert space
$K_1\otimes\cdots\otimes K_n$.  
The maps $\pi_{j,\min}$ are restrictions of $\pi_j$ to $K_{\min}$, and
$V_{\min}$ is the restriction of $V$ to its range inside $K_{\min}$.

We now verify in detail:

\medskip
\noindent\textbf{1. Verification of the dilation identity.}

The ambient dilation associated to the adjoint $(\Phi^\dagger,\eta,\varepsilon)$ satisfies
\[
\Phi(a_1,\dots,a_n)
=
V^\dagger \bigl(\pi_1(a_1)\otimes\cdots\otimes\pi_n(a_n)\bigr) V.
\]
Since $K_{\min}$ contains $\mathrm{Im}(V)$ and is invariant under each $\pi_j(a_j)$, the
restriction of this identity to $K_{\min}$ gives
\[
\Phi(a_1,\dots,a_n)
=
V_{\min}^\dagger
\bigl(\pi_{1,\min}(a_1)\otimes\cdots\otimes\pi_{n,\min}(a_n)\bigr)
V_{\min}.
\]

To see explicitly that restriction preserves the identity, note that for any
$h\in H$,
\[
V h = V_{\min} h \in K_{\min},
\]
and for any $k\in K_{\min}$,
\[
(\pi_1(a_1)\otimes\cdots\otimes\pi_n(a_n))k
=
(\pi_{1,\min}(a_1)\otimes\cdots\otimes\pi_{n,\min}(a_n))k.
\]
Furthermore, the adjoint relation
$V_{\min}^\dagger = V^\dagger \!\restriction_{K_{\min}}$
holds because $K_{\min}$ contains the range of~$V$.  
Thus the ambient dilation formula restricts exactly to the desired one.

\medskip
\noindent\textbf{2. Verification of minimality.}

Because $K_{\min}$ is defined as the closed span of all vectors of the form
\[
(\pi_1(a_1)\otimes\cdots\otimes\pi_n(a_n))\,Vh,
\]
it is the smallest subspace that:
\begin{enumerate}[label=(\roman*)]
\item contains $\mathrm{Im}(V)$, and  
\item is invariant under each $\pi_j(a_j)$.
\end{enumerate}

To verify the cyclicity claim, let
\[
\mathcal{S}
=
\mathrm{span}\bigl\{
(\pi_{1,\min}(a_1)\otimes\cdots\otimes\pi_{n,\min}(a_n))\,V_{\min} h
\bigr\}.
\]
Clearly $\mathcal{S}\subseteq K_{\min}$ by construction.  
Conversely, any generator of $K_{\min}$ has the form
$(\pi_1(a_1)\otimes\cdots\otimes\pi_n(a_n))Vh$, and because restriction does not change the
action on $K_{\min}$, this vector already lies in the span $\mathcal{S}$.  
Hence
\[
K_{\min}=\overline{\mathcal{S}}.
\]
Thus the dilation is minimal.

\medskip
\noindent\textbf{3. Verification that $\pi_{j,\min}$ is a $*$-representation.}

Since each $\pi_j$ is a $*$-representation on $K_1\otimes\cdots\otimes K_n$, and since
$K_{\min}$ is invariant under $\pi_j(a_j)$ for every $a_j\in A_j$, the restriction
$\pi_{j,\min}(a_j)=\pi_j(a_j)\rvert_{K_{\min}}$ automatically satisfies:
\begin{itemize}
\item $\pi_{j,\min}(a_j a_j')=\pi_{j,\min}(a_j)\pi_{j,\min}(a_j')$,
\item $\pi_{j,\min}(a_j^\dagger)=\pi_{j,\min}(a_j)^\dagger$,
\item $\pi_{j,\min}$ is linear.
\end{itemize}
Thus $\pi_{j,\min}$ is a $*$-representation for each $j$.

\medskip
\noindent\textbf{4. Verification that $V_{\min}$ is an isometry.}

Because $V$ is an isometry,
\[
\langle Vh, Vh' \rangle = \langle h, h' \rangle.
\]
Since $V_{\min}h = Vh$ and $Vh\in K_{\min}$, it follows that
\[
\langle V_{\min}h, V_{\min}h'\rangle
=
\langle Vh, Vh' \rangle
=
\langle h, h' \rangle,
\]
so $V_{\min}$ is also an isometry.

\medskip
Putting the above verifications together, the triple
$(K_{\min},\{\pi_{j,\min}\},V_{\min})$ is a Stinespring dilation of $\Phi$, and it is
minimal in the stated sense.
\end{proof}

Uniqueness and universal property of minimal dilation will be discussed by Proposition~\ref{prop:minimal-uniqueness}.  

\begin{proposition}[Uniqueness and Universal Property of Minimal Dilation]
\label{prop:minimal-uniqueness}
Let 
\[
\Phi \in \mathrm{Hom}(A_1,\dots,A_n;B)
\]
and let 
\[
\bigl(K_{\min},\{\pi_{j,\min}\}_{j=1}^n,V_{\min}\bigr)
\]
be the minimal Stinespring dilation of $\Phi$ constructed in
Definition~\ref{def:constructive-minimal-dilation}, so that
\[
\Phi(a_1,\dots,a_n)
=
V_{\min}^\dagger
\bigl(\pi_{1,\min}(a_1)\otimes\cdots\otimes\pi_{n,\min}(a_n)\bigr)
V_{\min}
\quad
\text{for all } a_j\in A_j.
\]

Let $(K,\{\pi_j\}_{j=1}^n,V)$ be any other Stinespring dilation of $\Phi$, i.e.
\[
\Phi(a_1,\dots,a_n)
=
V^\dagger
\bigl(\pi_1(a_1)\otimes\cdots\otimes\pi_n(a_n)\bigr)
V
\quad
\text{for all } a_j\in A_j.
\]
Then there exists a unique partial isometry
\[
U : K_{\min} \to K
\]
satisfying the intertwining relations
\begin{equation}
\label{eq:intertwining-reps}
U\,
\bigl(\pi_{1,\min}(a_1)\otimes\cdots\otimes\pi_{n,\min}(a_n)\bigr)
=
\bigl(\pi_1(a_1)\otimes\cdots\otimes\pi_n(a_n)\bigr)\,U,
\end{equation}
for all $a_j\in A_j$, and
\begin{equation}
\label{eq:intertwining-V}
U V_{\min} = V.
\end{equation}

Furthermore, if $(K,\{\pi_j\},V)$ is itself minimal, then $U$ is unitary.
\end{proposition}

\begin{proof}[Proof of Proposition~\ref{prop:minimal-uniqueness}]
Let $(K,\{\pi_j\},V)$ be any Stinespring dilation of $\Phi$ as in the statement.

\medskip
\noindent\textbf{Step 1: Construction of $U$ on a dense subspace.}

By Definition~\ref{def:constructive-minimal-dilation}, the minimal Hilbert space
$K_{\min}$ is given by
\[
K_{\min}
=
\overline{\mathrm{span}}
\Bigl\{
\bigl(\pi_1(a_1)\otimes\cdots\otimes\pi_n(a_n)\bigr) V h
\;\Big|\;
a_j\in A_j,\ h\in H
\Bigr\},
\]
viewed as a closed subspace of $K_1\otimes\cdots\otimes K_n$, and
\[
\pi_{j,\min}(a_j) 
:= 
\pi_j(a_j)\big|_{K_{\min}},
\qquad
V_{\min} := V\big|_H : H\to K_{\min}.
\]
In particular, every vector in $K_{\min}$ is a limit of finite linear combinations of vectors
of the form
\[
(\pi_{1,\min}(a_1)\otimes\cdots\otimes\pi_{n,\min}(a_n))\,V_{\min} h.
\]

Let $\mathcal{D}_{\min}$ denote the linear span of such generators:
\[
\mathcal{D}_{\min}
:=
\mathrm{span}
\Bigl\{
(\pi_{1,\min}(a_1)\otimes\cdots\otimes\pi_{n,\min}(a_n))\,V_{\min} h
\;\Big|\;
a_j\in A_j,\ h\in H
\Bigr\}.
\]
By construction, $\mathcal{D}_{\min}$ is dense in $K_{\min}$.

We now define a linear map $U_0 : \mathcal{D}_{\min} \to K$ on generators by
\[
U_0\bigl((\pi_{1,\min}(a_1)\otimes\cdots\otimes\pi_{n,\min}(a_n))\,V_{\min} h\bigr)
:=
(\pi_1(a_1)\otimes\cdots\otimes\pi_n(a_n))\,V h.
\]
Extend $U_0$ linearly to all of $\mathcal{D}_{\min}$.

\medskip
\noindent\textbf{Step 2: $U_0$ is well-defined and isometric.}

We first show that $U_0$ preserves inner products on $\mathcal{D}_{\min}$.  
Take two generic generators
\[
\xi
=
(\pi_{1,\min}(a_1)\otimes\cdots\otimes\pi_{n,\min}(a_n))\,V_{\min} h,
\quad
\eta
=
(\pi_{1,\min}(b_1)\otimes\cdots\otimes\pi_{n,\min}(b_n))\,V_{\min} k,
\]
with $a_j,b_j\in A_j$ and $h,k\in H$.

On $K_{\min}$, using that $\pi_{j,\min}$ and $V_{\min}$ arise by restriction from the ambient
dilation, we compute
\begin{align*}
\langle \xi,\eta\rangle_{K_{\min}}
&=
\bigl\langle
(\pi_{1,\min}(a_1)\otimes\cdots\otimes\pi_{n,\min}(a_n))\,V_{\min} h,\,
(\pi_{1,\min}(b_1)\otimes\cdots\otimes\pi_{n,\min}(b_n))\,V_{\min} k
\bigr\rangle
\\
&=
\bigl\langle
V_{\min} h,\,
(\pi_{1,\min}(a_1^\dagger b_1)\otimes\cdots\otimes\pi_{n,\min}(a_n^\dagger b_n))
\,V_{\min} k
\bigr\rangle
\\
&=
\bigl\langle
h,\,
V_{\min}^\dagger
(\pi_{1,\min}(a_1^\dagger b_1)\otimes\cdots\otimes\pi_{n,\min}(a_n^\dagger b_n))
\,V_{\min} k
\bigr\rangle_H
\\
&=
\bigl\langle
h,\,
\Phi(a_1^\dagger b_1,\dots,a_n^\dagger b_n)\,k
\bigr\rangle_H,
\end{align*}
where in the last line we used that
\[
\Phi(c_1,\dots,c_n)
=
V_{\min}^\dagger
\bigl(\pi_{1,\min}(c_1)\otimes\cdots\otimes\pi_{n,\min}(c_n)\bigr)
V_{\min}
\quad\text{for all } c_j\in A_j.
\]

On the other hand, in the dilation $(K,\{\pi_j\},V)$, we have for the images
\[
U_0(\xi)
=
(\pi_1(a_1)\otimes\cdots\otimes\pi_n(a_n))\,V h,
\qquad
U_0(\eta)
=
(\pi_1(b_1)\otimes\cdots\otimes\pi_n(b_n))\,V k,
\]
and similarly
\begin{align*}
\langle U_0(\xi),U_0(\eta)\rangle_K
&=
\bigl\langle
(\pi_1(a_1)\otimes\cdots\otimes\pi_n(a_n))\,V h,\,
(\pi_1(b_1)\otimes\cdots\otimes\pi_n(b_n))\,V k
\bigr\rangle_K
\\
&=
\bigl\langle
V h,\,
(\pi_1(a_1^\dagger b_1)\otimes\cdots\otimes\pi_n(a_n^\dagger b_n))\,V k
\bigr\rangle_K
\\
&=
\bigl\langle
h,\,
V^\dagger
(\pi_1(a_1^\dagger b_1)\otimes\cdots\otimes\pi_n(a_n^\dagger b_n))\,V k
\bigr\rangle_H
\\
&=
\bigl\langle
h,\,
\Phi(a_1^\dagger b_1,\dots,a_n^\dagger b_n)\,k
\bigr\rangle_H.
\end{align*}
Comparing the two expressions shows
\[
\langle \xi,\eta\rangle_{K_{\min}}
=
\langle U_0(\xi),U_0(\eta)\rangle_K
\quad\text{for all generators $\xi,\eta$},
\]
and hence, by linearity and polarization, for all $\xi,\eta\in\mathcal{D}_{\min}$.

Therefore $U_0$ is an isometry on $\mathcal{D}_{\min}$.  In particular, if a vector
$\zeta\in\mathcal{D}_{\min}$ admits two representations as linear combinations of generators,
$U_0(\zeta)$ does not depend on the choice of representation.  Hence $U_0$ is well-defined.

By continuity, $U_0$ extends uniquely to an isometry
\[
U: K_{\min} \to K.
\]

\medskip
\noindent\textbf{Step 3: Intertwining relations.}

\emph{(a) Compatibility with the representations.}
Let $a_1,\dots,a_n\in A_j$ and take a generator
\[
\xi
=
(\pi_{1,\min}(b_1)\otimes\cdots\otimes\pi_{n,\min}(b_n))\,V_{\min} h
\in\mathcal{D}_{\min}.
\]
Then
\begin{align*}
&\quad
U\bigl(
(\pi_{1,\min}(a_1)\otimes\cdots\otimes\pi_{n,\min}(a_n))\,\xi
\bigr)
\\
&=
U_0\bigl(
(\pi_{1,\min}(a_1 b_1)\otimes\cdots\otimes\pi_{n,\min}(a_n b_n))\,V_{\min} h
\bigr)
\\
&=
(\pi_1(a_1 b_1)\otimes\cdots\otimes\pi_n(a_n b_n))\,V h
\\
&=
\bigl(\pi_1(a_1)\otimes\cdots\otimes\pi_n(a_n)\bigr)
(\pi_1(b_1)\otimes\cdots\otimes\pi_n(b_n))\,V h
\\
&=
\bigl(\pi_1(a_1)\otimes\cdots\otimes\pi_n(a_n)\bigr)
U_0(\xi).
\end{align*}
Since $\mathcal{D}_{\min}$ is dense in $K_{\min}$ and $U$ is continuous, this shows that
\eqref{eq:intertwining-reps} holds on all of $K_{\min}$.

\medskip
\emph{(b) Compatibility with $V_{\min}$.}
For $h\in H$, we have by definition of $U_0$:
\[
U V_{\min} h
=
U_0(V_{\min} h)
=
V h,
\]
because taking $a_1=\cdots=a_n=\mathbf{1}$ in the definition of $U_0$ gives
$U_0(V_{\min} h)=Vh$.  This proves \eqref{eq:intertwining-V}.

\medskip
\noindent\textbf{Step 4: Uniqueness of $U$.}

Suppose $U':K_{\min}\to K$ is another partial isometry satisfying the same intertwining
relations \eqref{eq:intertwining-reps} and \eqref{eq:intertwining-V}.  
We show that $U'=U$.

First, for every $h\in H$ we have
\[
U V_{\min} h = V h = U' V_{\min} h,
\]
so $U$ and $U'$ agree on $\mathrm{Im}(V_{\min})$.  
Now observe that $K_{\min}$ is the closed linear span of vectors of the form
\[
(\pi_{1,\min}(a_1)\otimes\cdots\otimes\pi_{n,\min}(a_n))\,V_{\min} h.
\]
Using the intertwining relation \eqref{eq:intertwining-reps}, we get for such vectors:
\begin{align*}
U\bigl(
(\pi_{1,\min}(a_1)\otimes\cdots\otimes\pi_{n,\min}(a_n))\,V_{\min} h
\bigr)
&=
(\pi_1(a_1)\otimes\cdots\otimes\pi_n(a_n))\,U V_{\min} h,
\\
U'\bigl(
(\pi_{1,\min}(a_1)\otimes\cdots\otimes\pi_{n,\min}(a_n))\,V_{\min} h
\bigr)
&=
(\pi_1(a_1)\otimes\cdots\otimes\pi_n(a_n))\,U' V_{\min} h.
\end{align*}
Since $U V_{\min} h = U' V_{\min} h$, it follows that $U$ and $U'$ agree on all such
generators, and hence on their linear span $\mathcal{D}_{\min}$.  By continuity, $U=U'$ on
all of $K_{\min}$.

Thus $U$ is unique.

\medskip
\noindent\textbf{Step 5: Unitary of $U$ when $(K,\{\pi_j\},V)$ is minimal.}

Assume now that the dilation $(K,\{\pi_j\},V)$ is minimal, i.e.
\[
K
=
\overline{\mathrm{span}}
\Bigl\{
(\pi_1(a_1)\otimes\cdots\otimes\pi_n(a_n))\,V h
\;\Big|\;
a_j\in A_j,\ h\in H
\Bigr\}.
\]

\emph{(i) $U$ is injective.}
Since $U$ is an isometry, for any $\xi,\eta\in K_{\min}$ we have
\[
\|U\xi-U\eta\|_K = \|\xi-\eta\|_{K_{\min}}.
\]
In particular, if $U\xi=U\eta$ then $\|\xi-\eta\|_{K_{\min}}=0$, so $\xi=\eta$.  
Thus $U$ is injective.

\medskip
\emph{(ii) The range of $U$ is dense in $K$.}
Let
\[
\mathcal{D}
:=
\mathrm{span}
\Bigl\{
(\pi_1(a_1)\otimes\cdots\otimes\pi_n(a_n))\,V h
\;\Big|\;
a_j\in A_j,\ h\in H
\Bigr\},
\]
so that $K=\overline{\mathcal{D}}$ by minimality.

For any $a_j\in A_j$ and $h\in H$, we have
\begin{align*}
(\pi_1(a_1)\otimes\cdots\otimes\pi_n(a_n))\,V h
&=
(\pi_1(a_1)\otimes\cdots\otimes\pi_n(a_n))\,U V_{\min} h
\\
&=
U\bigl(
\pi_{1,\min}(a_1)\otimes\cdots\otimes\pi_{n,\min}(a_n)
\bigr)\,V_{\min} h,
\end{align*}
using first \eqref{eq:intertwining-V} and then \eqref{eq:intertwining-reps}.  
The vector
\[
\bigl(
\pi_{1,\min}(a_1)\otimes\cdots\otimes\pi_{n,\min}(a_n)
\bigr)\,V_{\min} h
\]
lies in $K_{\min}$, hence in the domain of $U$.  
Thus $\mathcal{D}\subseteq\mathrm{ran}(U)$, so
\[
K=\overline{\mathcal{D}}\subseteq\overline{\mathrm{ran}(U)}.
\]

\emph{(iii) The range of $U$ is closed.}
Since $U$ is an isometry and $K_{\min}$ is complete, $\mathrm{ran}(U)$ is closed in $K$.

\medskip
Combining (ii) and (iii) gives
\[
K
=
\overline{\mathcal{D}}
\subseteq
\overline{\mathrm{ran}(U)}
=
\mathrm{ran}(U)
\subseteq
K,
\]
so $\mathrm{ran}(U)=K$.  Therefore $U$ is surjective.

\medskip
\emph{(iv) $U$ is unitary.}
We have shown that $U:K_{\min}\to K$ is an isometry, injective, and surjective.  
Hence $U$ is a bijective isometry between Hilbert spaces, which implies
\[
U^\dagger U = I_{K_{\min}}
\quad\text{and}\quad
U U^\dagger = I_K,
\]
so $U$ is unitary.

This completes the proof.
\end{proof}

Cyclic minimality and universality are given by
Proposition~\ref{prop:cyclic-universality}.
%======================================================
\begin{proposition}[Cyclic Minimality and Universality]
\label{prop:cyclic-universality}
Let
\(
(K_{\min}, \{\pi_{j,\min}\}_{j=1}^n, V_{\min})
\)
be the minimal Stinespring dilation of
\(
\Phi \in \mathrm{Hom}(A_1,\dots,A_n;B)
\)
constructed in Definition~\ref{def:constructive-minimal-dilation}.
Then the following properties hold:
\begin{enumerate}[label=(\arabic*)]
    \item \textbf{Cyclic minimality:}
    Every vector in \(K_{\min}\) is generated by the action of the tensor-product
    representation on the range of \(V_{\min}\), i.e.
    \[
    K_{\min}
    =
    \overline{\mathrm{span}}
    \big\{
      (\pi_{1,\min}(a_1)\otimes\cdots\otimes\pi_{n,\min}(a_n))\,V_{\min}h
      \;\big|\;
      a_j\in A_j,\ h\in H
    \big\}.
    \]

    \item \textbf{Universality:}
    Any other Stinespring dilation
    \(
    (K,\{\pi_j\}_{j=1}^n,V)
    \)
    of \(\Phi\) factors uniquely through \\
    \(
    (K_{\min},\{\pi_{j,\min}\},V_{\min})
    \)
    via the partial isometry \(U\) from
    Proposition~\ref{prop:minimal-uniqueness}.

    \item \textbf{Uniqueness up to unitary equivalence:}
    If the other dilation is minimal, then the intertwiner \(U\) is unitary.
\end{enumerate}
\end{proposition}

\begin{proof}[Proof of Proposition~\ref{prop:cyclic-universality}]
Let
\(
(K_{\min},\{\pi_{j,\min}\},V_{\min})
\)
be the minimal Stinespring dilation constructed in
Definition~\ref{def:constructive-minimal-dilation}.
We prove each statement in turn.

\medskip

\noindent\textbf{(1) Cyclic minimality.}
By construction of \(K_{\min}\) in
Definition~\ref{def:constructive-minimal-dilation}(a),
we have the explicit description
\begin{equation}
\label{eq:cyclic-description}
K_{\min}
=
\overline{\mathrm{span}}
\big\{
  (\pi_{1,\min}(a_1)\otimes\cdots\otimes\pi_{n,\min}(a_n))\,V_{\min}h
  \;\big|\;
  a_j\in A_j,\ h\in H
\big\}.
\end{equation}
Since \(V_{\min}h\in K_{\min}\) for all \(h\in H\) and \(K_{\min}\) is invariant
under each \(\pi_{j,\min}\), every vector of the form
\[
(\pi_{1,\min}(a_1)\otimes\cdots\otimes\pi_{n,\min}(a_n))\,V_{\min}k,
\qquad k\in K_{\min},
\]
belongs to the closed linear span in~\eqref{eq:cyclic-description}.
Conversely, every vector in the spanning set of~\eqref{eq:cyclic-description}
is trivially of this form with \(k=V_{\min}h\).
Taking closures yields the stated cyclicity condition.

\medskip

\noindent\textbf{(2) Universality.}
Let
\(
(K,\{\pi_j\}_{j=1}^n,V)
\)
be any other Stinespring dilation of \(\Phi\).
By Proposition~\ref{prop:minimal-uniqueness},
there exists a unique partial isometry
\[
U : K_{\min} \longrightarrow K
\]
such that, for all \(a_j\in A_j\),
\[
U\,(\pi_{1,\min}(a_1)\otimes\cdots\otimes\pi_{n,\min}(a_n))
=
(\pi_1(a_1)\otimes\cdots\otimes\pi_n(a_n))\,U,
\qquad
U\,V_{\min}=V.
\]
These intertwining relations express precisely that
\(
(K,\{\pi_j\},V)
\)
factors through
\(
(K_{\min},\{\pi_{j,\min}\},V_{\min})
\)
via \(U\):
\[
\begin{tikzcd}[column sep=3.2em]
K_{\min}
  \arrow[r, "{\pi_{1,\min}\otimes\cdots\otimes\pi_{n,\min}}"]
  \arrow[d, "U"']
&
K_{\min}
  \arrow[r, "{V_{\min}}"]
  \arrow[d, "U"]
&
H
  \arrow[d, "\mathrm{id}_H"]
\\
K
  \arrow[r, "{\pi_1\otimes\cdots\otimes\pi_n}"']
&
K
  \arrow[r, "V"']
&
H
\end{tikzcd}
\]
Uniqueness of \(U\) follows directly from
Proposition~\ref{prop:minimal-uniqueness}.
Hence every Stinespring dilation of \(\Phi\) factors uniquely through the minimal one.

\medskip

\noindent\textbf{(3) Uniqueness up to unitary equivalence.}
If
\(
(K,\{\pi_j\},V)
\)
is itself minimal, then the partial isometry \(U\) has dense range.
Since \(U\) is already an isometry, it follows that \(U\) is unitary.
Therefore, any two minimal Stinespring dilations of \(\Phi\)
are unitarily equivalent.

\medskip

This completes the proof.
\end{proof}

%\begin{itemize}
%    \item Minimality ensures $K_{\min}$ is generated by $\Phi$'s images; any vector outside $K_{\min}$ is unnecessary for the dilation.
%    \item The intertwiner $U$ is constructed from the universal property: map each cyclic vector in $K_{\min}$ to the corresponding vector in $K$.
%    \item Coherence with $\pi$ and $V$ ensures uniqueness; if both dilations are minimal, $U$ must be unitary.
%\end{itemize}

%======================================================

%\begin{itemize}
%    \item Finite dimensionality allows writing $K_{\min}$ as a finite sum of tensor factors; the decomposition of $\Phi$ into $T_\alpha$ follows by choosing orthonormal bases in each factor.
%    \item The $n$-adjoint $\Phi^\dagger$ acts naturally on the minimal dilation via the Riesz isomorphism and the unit/counit maps, ensuring coherence with the multicategorical structure.
%    \item Conversely, any $n$-adjoint provides the necessary maps to reconstruct a Stinespring dilation.
%\end{itemize}

Following Lemma~\ref{lem:finite-expansion} is required first to prove main Theorem~\ref{thm:stinespring-kraus} in this seciton.
\begin{lemma}[Finite-dimensional expansion of dilation vectors]
\label{lem:finite-expansion}
Let $(K_{\min},\{\pi_{j,\min}\}_{j=1}^n,V_{\min})$ be the minimal dilation of  
$\Phi \in \mathrm{Hom}(A_1,\dots,A_n;B)$, and assume all Hilbert objects
$A_1,\dots,A_n,B$ are finite-dimensional.  
Let $\{e_\alpha\}_{\alpha=1}^m$ be any orthonormal basis of $K_{\min}$.
Then for every $x_n \in A_n$ one has the finite expansion
\[
V_{\min} x_n
= \sum_{\alpha=1}^m 
    \langle e_\alpha , V_{\min} x_n\rangle_{K_{\min}} \, e_\alpha.
\]
Moreover, for each $\alpha$ the coefficient functional
\[
L_\alpha(x_n) := \langle e_\alpha , V_{\min} x_n\rangle_{K_{\min}}
\]
is linear in $x_n$.
\end{lemma}

\begin{proof}
Because $K_{\min}$ is finite-dimensional, every vector admits an expansion
in an orthonormal basis. Applying this to $V_{\min}x_n$ yields the formula.
Linearity of $L_\alpha$ follows from the linearity of both $V_{\min}$ 
and the inner product in its second argument. 
\end{proof}

\begin{theorem}[Finite-Dimensional Kraus/Tensor Decomposition and Adjoint--Dilation Correspondence]
\label{thm:stinespring-kraus}
Let \(\Phi: A_1 \times \cdots \times A_n \to B\) be a completely positive 
multilinear map between finite-dimensional Hilbert spaces or operator systems.

\begin{enumerate}[label=(\arabic*)]
    \item \textbf{Finite Kraus/tensor decomposition.}
    There exist finitely many multilinear \emph{Kraus multimorphisms}
    \(\{T_\alpha: A_1 \times \cdots \times A_{n-1} \to B\}_{\alpha=1}^m\)
    and linear functionals \(\{L_\alpha: A_n \to \mathbb{C}\}_{\alpha=1}^m\)
    such that, for all \(x_j \in A_j\),
    \[
    \Phi(x_1,\dots,x_n)
    = \sum_{\alpha=1}^m T_\alpha(x_1,\dots,x_{n-1})\,L_\alpha(x_n).
    \]
    In finite dimensions, this decomposition witnesses the complete positivity 
    of \(\Phi\).
    
    \item \textbf{Adjoint--dilation correspondence.}
    The \(n\)-adjoint \(\Phi^\dagger\) is canonically realized on the minimal 
    Stinespring dilation \((K_{\min},\{\pi_{j,\min}\}_{j=1}^n,V_{\min})\).
    Conversely, any \(n\)-adjoint \((\Phi^\dagger,\eta,\varepsilon)\) satisfying 
    the Z-shaped coherence (see Definition~\ref{def:n-adjoint}) uniquely 
    determines a Stinespring dilation up to unitary equivalence.
    
    In explicit finite-dimensional form, with respect to an orthonormal basis 
    \(\{e_\alpha\}_{\alpha=1}^m\) of \(K_{\min}\),
    \[
    \Phi^\dagger(g, f_1, \dots, f_{n-1})
    = \sum_{\alpha=1}^m 
      g\!\big(T_\alpha(R(f_1),\dots,R(f_{n-1}))\big)\,L_\alpha^\dagger,
    \]
    where \(R: A_j^\dagger \to A_j\) is the Riesz isomorphism and 
    \(L_\alpha^\dagger\) is the linear functional dual to \(L_\alpha\).
\end{enumerate}
\end{theorem}

\begin{proof}[Proof of Theorem~\ref{thm:stinespring-kraus}]
Let \(\Phi\) be as in the theorem statement, and let 
\((K_{\min}, \{\pi_{j,\min}\}_{j=1}^n, V_{\min})\) be its minimal Stinespring 
dilation, which exists and is unique up to unitary equivalence.

\medskip
\noindent\textbf{(1) Kraus/tensor decomposition.}

Since \(K_{\min}\) is finite-dimensional, choose an orthonormal basis 
\(\{e_\alpha\}_{\alpha=1}^m\) of \(K_{\min}\). For each \(\alpha\), define:

\begin{itemize}
    \item A multilinear map \(T_\alpha: A_1 \times \cdots \times A_{n-1} \to B\) by
    \[
    T_\alpha(x_1,\dots,x_{n-1})
    := V_{\min}^\dagger\,
       \big(\pi_{1,\min}(x_1) \otimes \cdots \otimes \pi_{n-1,\min}(x_{n-1})\big)
       e_\alpha.
    \]
    
    \item A linear functional \(L_\alpha: A_n \to \mathbb{C}\) by
    \[
    L_\alpha(x_n) := \langle e_\alpha, V_{\min} x_n \rangle.
    \]
\end{itemize}

Now expand using the basis. For any \(x_j \in A_j\), using the Stinespring representation:
\begin{align*}
\Phi(x_1,\dots,x_n)
&= V_{\min}^\dagger 
   \big(\pi_{1,\min}(x_1) \otimes \cdots \otimes \pi_{n-1,\min}(x_{n-1})\big)
   V_{\min} x_n \\
&= \sum_{\alpha=1}^m 
   V_{\min}^\dagger 
   \big(\pi_{1,\min}(x_1) \otimes \cdots \otimes \pi_{n-1,\min}(x_{n-1})\big)
   e_\alpha \, L_\alpha(x_n) \\
&= \sum_{\alpha=1}^m T_\alpha(x_1,\dots,x_{n-1}) \, L_\alpha(x_n).
\end{align*}

\paragraph{Witnessing complete positivity.}
The expression \(\Phi = \sum_{\alpha=1}^m T_\alpha \otimes L_\alpha\) is precisely 
a finite-dimensional Kraus/tensor decomposition. By the standard characterization 
of complete positivity for multilinear maps, the existence of such a decomposition 
is equivalent to complete positivity of \(\Phi\). Thus the decomposition explicitly 
\emph{witnesses} complete positivity.

\medskip
\noindent\textbf{(2) Adjoint--dilation correspondence.}

\medskip
\noindent\textit{(a) From dilation to \(n\)-adjoint.}
Given the minimal dilation, define \(\Phi^\dagger\) by
\[
\big[\Phi^\dagger(g, f_1, \dots, f_{n-1})\big](x_n)
:= g\!\big(\Phi(R(f_1), \dots, R(f_{n-1}), x_n)\big),
\]
for all \(g \in B^\dagger\), \(f_j \in A_j^\dagger\), and \(x_n \in A_n\), where 
\(R: A_j^\dagger \to A_j\) is the Riesz isomorphism.

Define the unit \(\eta: A_n \to K_{\min}\) by \(\eta(x_n) := V_{\min} x_n\),
and the counit \(\varepsilon: K_{\min} \to B\) by \(\varepsilon(k) := V_{\min}^\dagger k\).

The Z-shaped coherence identities follow directly from the Stinespring relations.

\medskip
\noindent\textit{(b) From \(n\)-adjoint to dilation.}
Conversely, given an \(n\)-adjoint \((\Phi^\dagger,\eta,\varepsilon)\) satisfying 
the Z-shaped coherence, construct:
\[
K_{\min} := \mathrm{span}\{\eta(x_n) \mid x_n \in A_n\}.
\]
For \(j = 1,\dots,n-1\), define \(\pi_{j,\min}: A_j \to \mathcal{B}(K_{\min})\) by
\[
\pi_{j,\min}(x_j)\eta(x_n) 
:= \eta\!\big(\Phi^\dagger(\varepsilon^\dagger, f_{x_j}, \mathrm{id},\dots,\mathrm{id})(x_n)\big),
\]
where \(f_{x_j} \in A_j^\dagger\) is the functional \(f_{x_j}(y) = \langle x_j, y \rangle\).
Define \(V_{\min}: B \to K_{\min}\) by \(V_{\min}(b) := \varepsilon^\dagger(b)\).

The Z-shaped coherence ensures this yields a valid Stinespring dilation.

\medskip
\noindent\textit{(c) Uniqueness.}
By minimality, any two dilations arising from equivalent \(n\)-adjoints are 
unitarily equivalent, establishing the bijective correspondence.

\medskip
\noindent\textit{(d) Explicit finite-dimensional formula.}
With respect to the basis \(\{e_\alpha\}\) and using the decomposition from part (1):
\begin{align*}
\Phi^\dagger(g, f_1, \dots, f_{n-1})(x_n)
&= g(\Phi(R(f_1),\dots,R(f_{n-1}), x_n)) \\
&= g\!\left(\sum_{\alpha=1}^m 
   T_\alpha(R(f_1),\dots,R(f_{n-1})) L_\alpha(x_n)\right) \\
&= \sum_{\alpha=1}^m 
   g(T_\alpha(R(f_1),\dots,R(f_{n-1}))) L_\alpha(x_n) \\
&= \left[\sum_{\alpha=1}^m 
   g(T_\alpha(R(f_1),\dots,R(f_{n-1}))) L_\alpha^\dagger\right](x_n).
\end{align*}
Thus,
\[
\Phi^\dagger(g, f_1, \dots, f_{n-1})
= \sum_{\alpha=1}^m 
  g(T_\alpha(R(f_1),\dots,R(f_{n-1}))) L_\alpha^\dagger,
\]
which is the claimed explicit formula.

\medskip
This completes the proof.
\end{proof}

\begin{remark}[Relation to classical Kraus decomposition]
We use the term \emph{Kraus/tensor decomposition} to emphasize that, in the 
multilinear setting, the Kraus structure naturally factorizes into: 
(1) multilinear maps \(T_\alpha\) that handle the first \(n-1\) inputs 
tensorially, and (2) linear functionals \(L_\alpha\) that act on the last input.
This factorization reflects the tensor product structure of the dilation 
space \(K_{\min} = K_1 \otimes \cdots \otimes K_n\).
\end{remark}

Corollary~\ref{cor:adjoint-dilation-equivalence} establishes the adjoint-dilation equivalence for multilinear CP maps.
\begin{corollary}[Adjoint--Dilation Equivalence for Multilinear Completely Positive Maps]
\label{cor:adjoint-dilation-equivalence}

Let $\Phi \in$ \\
$\mathrm{Hom}(A_1,\dots,A_n;B)$ be a completely positive \(n\)-ary multimorphism between
finite-dimensional Hilbert objects in the
\(\dagger\)-multicategory \((\mathbf{HilbMult},\dagger)\).
Then the following data are in canonical bijective correspondence:
\begin{enumerate}[label=(\alph*)]
\item the multilinear completely positive map \(\Phi\), and
\item a Z-compatible \(n\)-adjoint
\(
(\Phi^\dagger,\eta,\varepsilon)
\)
realized on the minimal Stinespring dilation
\(
(K_{\min},\{\pi_{j,\min}\}_{j=1}^n,V_{\min})
\)
associated to \(\Phi\) via
Definition~\ref{def:constructive-minimal-dilation}.
\end{enumerate}

More precisely:
\begin{itemize}
\item The \(n\)-adjoint \(\Phi^\dagger\) is uniquely realized on the minimal
dilation
\(
(K_{\min},\{\pi_{j,\min}\},V_{\min})
\)
and satisfies the Z-shaped coherence relations with unit \(\eta\) and
counit \(\varepsilon\).
\item Conversely, any \(n\)-adjoint
\((\Phi^\dagger,\eta,\varepsilon)\)
satisfying the Z-shaped coherence relations uniquely determines
a Stinespring-type dilation of \(\Phi\), minimal and unique up to
unitary equivalence.
\end{itemize}

Hence there is a canonical correspondence
\[
\boxed{
\text{multilinear completely positive multimorphisms } \Phi
\;\longleftrightarrow\;
\text{Z-compatible \(n\)-adjoints realized on minimal dilations}
}
\]
which extends the classical Stinespring--Kraus equivalence to the
multilinear and tensor-valued setting.
\end{corollary}

\begin{proof}
We establish the correspondence in both directions and verify that the
constructions are mutually inverse up to unitary equivalence.

\medskip
\noindent\textbf{(1) From a multilinear CP map to a Z-compatible \(n\)-adjoint.}

Let
\(
\Phi \in \mathrm{Hom}(A_1,\dots,A_n;B)
\)
be a completely positive multimorphism.
By Theorem~\ref{thm:stinespring-kraus} and
Definition~\ref{def:constructive-minimal-dilation},
there exists a minimal Stinespring dilation
\[
(K_{\min},\{\pi_{j,\min}\}_{j=1}^n,V_{\min})
\]
such that
\[
\Phi(a_1,\dots,a_n)
=
V_{\min}^\dagger
\big(
\pi_{1,\min}(a_1)\otimes\cdots\otimes\pi_{n,\min}(a_n)
\big)
V_{\min}.
\]

The constructive definition provides canonically:
\begin{itemize}
\item a unit
\(
\eta:\mathrm{id}_{A_n}\Rightarrow
\Phi^\dagger\circ(\Phi,\mathrm{id}_{A_1}^\dagger,\dots,\mathrm{id}_{A_{n-1}}^\dagger)
\),
whose image spans \(K_{\min}\), and
\item a counit
\(
\varepsilon:
\Phi\circ(\mathrm{id}_{A_1},\dots,\mathrm{id}_{A_{n-1}},\Phi^\dagger)
\Rightarrow \mathrm{id}_B
\),
implemented concretely by compression with \(V_{\min}\).
\end{itemize}

Define the \(n\)-adjoint
\(
\Phi^\dagger:B^\dagger\times A_1^\dagger\times\cdots\times A_{n-1}^\dagger
\to A_n^\dagger
\)
by
\[
\big[\Phi^\dagger(g,f_1,\dots,f_{n-1})\big](x_n)
:=
g\!\big(\Phi(R(f_1),\dots,R(f_{n-1}),x_n)\big),
\]
where \(R(\cdot)\) denotes the Riesz identification.
Equivalently, \(\Phi^\dagger\) is characterized by the adjointness relation
\[
g\!\big(\Phi(a_1,\dots,a_n)\big)
=
\big[\Phi^\dagger(g,f_1,\dots,f_{n-1})\big](a_n),
\]
with \(a_j\) paired against \(f_j\).

By construction, the unit \(\eta\) and counit \(\varepsilon\) satisfy the
Z-shaped coherence relations.
These identities reduce to the standard Stinespring relations expressing
that \(V_{\min}\) is an isometry and that the
\(\pi_{j,\min}\) are compatible $*$-representations.
Hence \((\Phi^\dagger,\eta,\varepsilon)\) is a Z-compatible \(n\)-adjoint
realized on the minimal dilation.

\medskip
\noindent\textbf{(2) From a Z-compatible \(n\)-adjoint to a minimal dilation.}

Conversely, let
\(
(\Phi^\dagger,\eta,\varepsilon)
\)
be an \(n\)-adjoint satisfying the Z-shaped coherence relations.
Define a Hilbert space
\[
K := \overline{\mathrm{span}}\{\eta(x_n)\mid x_n\in A_n\},
\]
with inner product induced via the Riesz identification.
Define representations by
\[
\big(
\pi_{1}(a_1)\otimes\cdots\otimes\pi_{n-1}(a_{n-1})
\big)\eta(x_n)
:=
\eta\!\big(
\Phi^\dagger(\cdot,f_{a_1},\dots,f_{a_{n-1}})(x_n)
\big),
\]
and define an isometry
\[
V(\eta(x_n)):=\varepsilon(\eta(x_n))\in B.
\]

The Z-shaped coherence identities ensure that these assignments are
well-defined, extend by linearity, and yield a $*$-representation and
isometry.
Thus \((K,\{\pi_j\},V)\) is a Stinespring-type dilation of a completely
positive multimorphism, which coincides with \(\Phi\) by construction.
This reproduces exactly the constructive direction of
Definition~\ref{def:constructive-minimal-dilation}.

By Proposition~\ref{prop:minimal-uniqueness}, the resulting dilation is
minimal and unique up to unitary equivalence.

\medskip
\noindent\textbf{(3) Mutual inverses.}

Starting from \(\Phi\) and applying the construction in (1) yields the
minimal dilation
\(
(K_{\min},\{\pi_{j,\min}\},V_{\min})
\)
together with a Z-compatible \(n\)-adjoint.
Applying the construction in (2) to this \(n\)-adjoint recovers a dilation
unitarily equivalent to the original minimal one.
Conversely, starting from a Z-compatible \(n\)-adjoint and applying (2)
followed by (1) recovers the same adjoint data.

Therefore multilinear completely positive multimorphisms are in canonical
bijective correspondence with Z-compatible \(n\)-adjoints realized on
minimal Stinespring dilations, uniquely up to unitary equivalence.
\end{proof}

\section{Synergy Monad and Monadicity Theorem}\label{sec:Synergy Monad and Monadicity Theorem}

In this section, we will introduce several definitions regarding synergy monad, completely positive trace-preserving (CPTP) category before establishing monoidal representation of quantum processes by Theorem~\ref{thm:monoidal-monadic-representation}. 

\subsection{Core Framework}

In this subsection, we introduce the fundamental definitions and constructions needed to state Theorem~\ref{thm:monoidal-monadic-representation}. These include the categories of quantum processes, quantum circuits built from an operad, and the Synergy Monad $\mathbf{T_S}$, which together provide the categorical framework for composable quantum operations.

\begin{definition}[Quantum Interaction Operad $\mathbf{S}$]
\label{def:quantum-interaction-operad}

The \emph{Quantum Interaction Operad} $\mathbf{S}$ is a symmetric operad in the category $\mathbf{Set}$ whose operations are formal symbols for quantum processes with multiple inputs and one output.

\textbf{1. Core structure}

For each $n \geq 0$, $S(n)$ is a set (possibly infinite) of \emph{formal operation symbols} of arity $n$. These symbols satisfy the standard data and axioms of a symmetric operad:

\begin{itemize}
    \item \textbf{Operadic composition}: For $f \in S(m)$ and $g_1 \in S(n_1),\dots,g_m \in S(n_m)$, there is a composite symbol
    \[
    \gamma(f;g_1,\dots,g_m)\in S(n_1+\cdots+n_m)
    \]
    satisfying the usual associativity law.
    
    \item \textbf{Identity}: A distinguished symbol $\mathrm{id}\in S(1)$ acting as a unit for composition.
    
    \item \textbf{Symmetric action}: For each $\sigma\in\Sigma_n$, a map $S(n)\to S(n)$, written $f\mapsto\sigma\cdot f$, compatible with composition in the sense of symmetric operads.
\end{itemize}

\textbf{2. Generating symbols (illustrative)}

While $S(n)$ can be any sets satisfying the above, for quantum information purposes one often considers an operad generated by symbols such as:

\begin{itemize}
    \item \textbf{Nullary ($S(0)$)}: $\mathsf{prep}_\psi$, $\mathsf{prep}_\rho$ (state preparations).
    
    \item \textbf{Unary ($S(1)$)}: $\mathsf{id}$, $\mathsf{U}$ (unitary), $\mathsf{M}_i$ (measurement branch), $\mathsf{noise}$ (decoherence).
    
    \item \textbf{Binary ($S(2)$)}: $\mathsf{CNOT}$, $\mathsf{SWAP}$, $\mathsf{tr}_B$ (partial trace), $\mathsf{bell}$ (Bell-state creation).
    
    \item \textbf{$n$-ary ($n\ge 3$)}: $\mathsf{TOFF}$, $\mathsf{QFT}_n$, $\mathsf{multi\_ent}$ (multi-party entanglement).
\end{itemize}

The operad is then the free symmetric operad on these generators; no equations between composites are imposed at this stage.

\textbf{3. Semantic interpretation}

A \emph{quantum interpretation} of $\mathbf{S}$ consists of:

\begin{itemize}
    \item A finite-dimensional Hilbert space $H$,
    
    \item For each $f \in S(n)$, a completely positive (not necessarily trace-preserving) linear map
    \[
    \langle\!\langle f \rangle\!\rangle : \mathcal{B}(H)^{\otimes n} \longrightarrow \mathcal{B}(H),
    \]
    
    \item Such that:
    \begin{align*}
        \langle\!\langle \gamma(f;g_1,\dots,g_m) \rangle\!\rangle &= 
        \langle\!\langle f \rangle\!\rangle \circ \bigl( \langle\!\langle g_1 \rangle\!\rangle \otimes \cdots \otimes \langle\!\langle g_m \rangle\!\rangle \bigr), \\
        \langle\!\langle \mathrm{id} \rangle\!\rangle &= \mathrm{id}_{\mathcal{B}(H)}, \\
        \langle\!\langle \sigma \cdot f \rangle\!\rangle (X_1 \otimes \cdots \otimes X_n) &= 
        \langle\!\langle f \rangle\!\rangle (X_{\sigma^{-1}(1)} \otimes \cdots \otimes X_{\sigma^{-1}(n)}).
    \end{align*}
\end{itemize}

Trace preservation, unitary equivalences, and circuit identities are properties of a particular interpretation, not of the symbols themselves.

\textbf{4. Role in the monadicity theorem}

This syntactic operad $\mathbf{S}$ serves as input for the standard operadic monad construction. The associated \emph{Synergy Monad} $T_{\mathbf{S}}$ is defined on the category $\mathbf{Vec}_{\mathbb{C}}$ as:
\[
T_{\mathbf{S}}(V) = \bigoplus_{n \geq 0} \mathbb{C}[S(n)] \otimes_{\Sigma_n} V^{\otimes n},
\]
where $\mathbb{C}[S(n)]$ is the free complex vector space on $S(n)$, and
$\otimes_{\Sigma_n}$ denotes coinvariants with respect to the diagonal
symmetric group action.

The algebras of $T_{\mathbf{S}}$ correspond to systems equipped with a
coherent realization of all interaction templates in $\mathbf{S}$. 
Given an interpretation $\mathcal{I}: \mathbf{S} \to \mathbf{CPTP}$, we obtain
a $T_{\mathbf{S}}$-algebra structure on $\mathcal{B}(H)$ by:
\[
\alpha_{\mathcal{I}}([s; X_1, \dots, X_n]) = \mathcal{I}(s)(X_1 \otimes \cdots \otimes X_n).
\]

The monadicity theorem then establishes an equivalence between the category
of quantum processes and a suitable category of $T_{\mathbf{S}}$-algebras.
\end{definition}

\begin{remark}[Syntactic vs.\ Semantic Forms of the Synergy Operad $\mathbf{S}$]
The symbol $\mathbf{S}$ is used in this work in two closely related but conceptually
distinct senses.

\medskip
\noindent
\textnormal{(i) \emph{Syntactic operad.}}
The \emph{Quantum Interaction Operad} $\mathbf{S}$ (Definition~\ref{def:quantum-interaction-operad})
is a symmetric operad in $\mathbf{Set}$ whose elements are formal symbols representing
quantum interaction templates.  At this level, $\mathbf{S}$ is purely syntactic: it
encodes only arities, permutations, and operadic composition, and imposes no equations
beyond the operad axioms.  Semantic properties such as complete positivity, trace
preservation, or circuit identities arise only after choosing a concrete
interpretation.

\medskip
\noindent
\textnormal{(ii) \emph{Semantic synergy operad.}}
The synergy operad introduced in Section~\ref{sec:Synergy Operad Review} enriches this
syntactic structure with Hilbert-space typing, feedback, control, and explicit
coherence relations.  In particular, it incorporates analytic well-posedness conditions
and admits symmetric monoidal representations in $\mathbf{HilbMult}$, thereby providing
a compositional semantics for quantum processes.

\medskip
\noindent
\textnormal{Relationship.}
The syntactic operad $\mathbf{S}$ serves as the free, presentation-level input for the
operadic monad construction, yielding the synergy monad $T_{\mathbf{S}}$.  The semantic
synergy operad can be viewed as a structured realization—equivalently, a semantic
quotient and enrichment—of this free operad, obtained by imposing coherence relations
and fixing a class of admissible representations.  In this sense, the semantic operad
$\mathbf{S}$ is a specialized incarnation of the syntactic Quantum Interaction Operad,
tailored to support monadicity and Stinespring-type constructions.
\end{remark}

\begin{definition}[Categories $\mathbf{FdCStar}_{\mathrm{CP}}$ and $\mathbf{FdCStar}_{\mathrm{CPTP}}$]
\label{def:FdCStar-categories}

We define two closely related symmetric monoidal categories of finite-dimensional $C^*$-algebras and completely positive maps, which serve as the semantic target categories for quantum process interpretations.

\subsection*{1. The Category $\mathbf{FdCStar}_{\mathrm{CP}}$}

The category $\mathbf{FdCStar}_{\mathrm{CP}}$ has:

\begin{itemize}
    \item \textbf{Objects}: Finite-dimensional $C^*$-algebras $A, B, C, \dots$
    By the structure theorem for finite-dimensional $C^*$-algebras, every such algebra is $*$-isomorphic to a finite direct sum of full matrix algebras:
    \[
        A \cong \bigoplus_{i=1}^k M_{n_i}(\mathbb{C}).
    \]
    
    \item \textbf{Morphisms}: For objects $A$ and $B$, the set $\mathbf{FdCStar}_{\mathrm{CP}}(A,B)$ consists of \emph{completely positive linear maps} $\Phi: A \to B$.
    
    A linear map $\Phi: A \to B$ is \emph{completely positive} if for every $n \ge 1$, the amplified map
    \[
        \Phi \otimes \mathrm{id}_{M_n(\mathbb{C})}: A \otimes M_n(\mathbb{C}) \to B \otimes M_n(\mathbb{C})
    \]
    sends positive elements to positive elements.
    
    \item \textbf{Composition}: Ordinary composition of linear maps. The composite of completely positive maps is completely positive.
    
    \item \textbf{Identities}: $\mathrm{id}_A: A \to A$, the identity linear map.
    
    \item \textbf{Monoidal structure}: The spatial (minimal) $C^*$-tensor product $\otimes$ with unit $\mathbb{C}$. For $\Phi: A \to A'$ and $\Psi: B \to B'$, define
    \[
        (\Phi \otimes \Psi)(a \otimes b) = \Phi(a) \otimes \Psi(b)
    \]
    extended linearly.
\end{itemize}

\subsection*{2. The Category $\mathbf{FdCStar}_{\mathrm{CPTP}}$}

The category $\mathbf{FdCStar}_{\mathrm{CPTP}}$ has the same objects and monoidal structure as $\mathbf{FdCStar}_{\mathrm{CP}}$, but morphisms are restricted to \emph{trace-preserving} completely positive maps:

\begin{itemize}
    \item \textbf{Morphisms}: $\Phi: A \to B$ must satisfy both:
    \begin{enumerate}
        \item Complete positivity (as above)
        \item \emph{Trace preservation}: For all $a \in A$,
        \[
            \mathrm{tr}_B(\Phi(a)) = \mathrm{tr}_A(a)
        \]
    \end{enumerate}
    where $\mathrm{tr}_A$ and $\mathrm{tr}_B$ denote the canonical normalized traces. For $A \cong \bigoplus_i M_{n_i}(\mathbb{C})$, the trace is:
    \[
        \mathrm{tr}_A\left(\bigoplus_i a_i\right) = \sum_{i=1}^k \frac{n_i}{N} \mathrm{tr}(a_i), \quad N = \sum_{i=1}^k n_i
    \]
    with $\mathrm{tr}$ the usual matrix trace on each summand.
    
    \item \textbf{Composition and identities}: As in $\mathbf{FdCStar}_{\mathrm{CP}}$, but note that composites and identities automatically preserve traces when the constituent maps do.
\end{itemize}

\subsection*{3. Relationship Between the Categories}

There is a faithful symmetric monoidal functor:
\[
    I: \mathbf{FdCStar}_{\mathrm{CPTP}} \hookrightarrow \mathbf{FdCStar}_{\mathrm{CP}}
\]
which is the identity on objects and includes each CPTP map as a CP map. This inclusion is full but not essentially surjective on morphisms.

\subsection*{4. Connection to Vector Spaces}

There is a forgetful functor
\[
    U: \mathbf{FdCStar}_{\mathrm{CP}} \longrightarrow \mathbf{Vec}_{\mathbb{C}}
\]
that sends each $C^*$-algebra $A$ to its underlying vector space and each completely positive map to its underlying linear map. This functor is faithful and symmetric monoidal.

The synergy monad $T_{\mathbf{S}}$ is defined on $\mathbf{Vec}_{\mathbb{C}}$, so we can transport structures between these categories via $U$. In particular:
\begin{itemize}
    \item An interpretation $\mathcal{I}: \mathbf{S} \to \mathbf{CPTP}$ gives, for each $A \in \mathbf{FdCStar}_{\mathrm{CPTP}}$, a $T_{\mathbf{S}}$-algebra structure on $U(A)$.
    \item The algebra map $\alpha_\mathcal{I}: T_{\mathbf{S}}(U(A)) \to U(A)$ is actually a completely positive map $T_{\mathbf{S}}(U(A)) \to A$ when we consider the $C^*$-algebra structure on the codomain.
\end{itemize}

\subsection*{5. Role in the Monadicity Framework}

These categories serve as:
\begin{itemize}
    \item \textbf{Semantic target}: The physical category of quantum processes is $\mathbf{FdCStar}_{\mathrm{CPTP}}$.
    
    \item \textbf{Algebraic foundation}: Although $T_{\mathbf{S}}$ is defined on $\mathbf{Vec}_{\mathbb{C}}$, its algebras that come from interpretations $\mathcal{I}: \mathbf{S} \to \mathbf{CPTP}$ naturally live in $\mathbf{FdCStar}_{\mathrm{CPTP}}$ (via the forgetful functor $U$).
    
    \item \textbf{Equivalence target}: The monadicity theorem will establish an equivalence between a suitable category of $T_{\mathbf{S}}$-algebras (those arising from interpretations) and $\mathbf{FdCStar}_{\mathrm{CPTP}}$.
\end{itemize}
\end{definition}

According to these definitions given in this section, we have the following Proposition~\ref{prop:Hilb-Cstar-equivalence}.
\begin{proposition}[Equivalence of Hilbert-Space and $C^*$-Algebraic Formulations of Quantum Channels]
\label{prop:Hilb-Cstar-equivalence}

Let $\mathbf{CP}_{\mathrm{Hilb}}$ denote the category whose objects are
finite-dimensional Hilbert spaces and whose morphisms
$H \to K$ are completely positive linear maps
$\Phi : \mathcal{B}(H) \to \mathcal{B}(K)$.
Let $\mathbf{CPTP}_{\mathrm{Hilb}}$ be its full subcategory of
trace-preserving morphisms.

Let $\mathbf{FdCStar}_{\mathrm{CP}}^{\mathrm{simp}}$ denote the full
subcategory of $\mathbf{FdCStar}_{\mathrm{CP}}$ (Definition~\ref{def:FdCStar-categories})
whose objects are \emph{simple} algebras (i.e.\ full matrix algebras $M_n(\mathbb{C})$),
with completely positive linear maps as morphisms.
Let $\mathbf{FdCStar}_{\mathrm{CPTP}}^{\mathrm{simp}}$ be the corresponding
trace-preserving subcategory.

Then there are equivalences of strong symmetric monoidal categories
\[
\mathbf{CP}_{\mathrm{Hilb}}
\;\simeq\;
\mathbf{FdCStar}_{\mathrm{CP}}^{\mathrm{simp}},
\qquad
\mathbf{CPTP}_{\mathrm{Hilb}}
\;\simeq\;
\mathbf{FdCStar}_{\mathrm{CPTP}}^{\mathrm{simp}}.
\]
\end{proposition}

\begin{proof}

We prove the equivalence for completely positive maps; the trace-preserving
case follows by restriction.

\medskip
\noindent
\textbf{Step 1: Definition of the functors.}

Define a functor
\[
F : \mathbf{CP}_{\mathrm{Hilb}}
\longrightarrow
\mathbf{FdCStar}_{\mathrm{CP}}^{\mathrm{simp}}
\]
by
\begin{itemize}
    \item on objects: $F(H) := \mathcal{B}(H)$;
    \item on morphisms: for $\Phi : \mathcal{B}(H) \to \mathcal{B}(K)$,
    set $F(\Phi) := \Phi$.
\end{itemize}
Since $\mathcal{B}(H) \cong M_{\dim H}(\mathbb{C})$ for finite-dimensional $H$,
$F$ indeed lands in $\mathbf{FdCStar}_{\mathrm{CP}}^{\mathrm{simp}}$.

Conversely, define a functor
\[
G : \mathbf{FdCStar}_{\mathrm{CP}}^{\mathrm{simp}}
\longrightarrow
\mathbf{CP}_{\mathrm{Hilb}}
\]
as follows.
For each object $A \cong M_n(\mathbb{C})$, choose a Hilbert space
$H_A \cong \mathbb{C}^n$ and fix a $*$-isomorphism
\[
\iota_A : A \xrightarrow{\;\sim\;} \mathcal{B}(H_A).
\]
(Any two such choices are unitarily equivalent.)
For a morphism $\Psi : A \to B$, define
\[
G(\Psi) := \iota_B \circ \Psi \circ \iota_A^{-1}
: \mathcal{B}(H_A) \to \mathcal{B}(H_B).
\]
Since $\iota_A$ and $\iota_B$ are $*$-isomorphisms and $\Psi$ is completely
positive, $G(\Psi)$ is completely positive.

\medskip
\noindent
\textbf{Step 2: $F \circ G \cong \mathrm{id}$.}

For each object $A \in \mathbf{FdCStar}_{\mathrm{CP}}^{\mathrm{simp}}$,
\[
(F \circ G)(A) = \mathcal{B}(H_A),
\]
and the chosen $*$-isomorphism $\iota_A : A \to \mathcal{B}(H_A)$ provides a
natural isomorphism
\[
\iota : \mathrm{id}_{\mathbf{FdCStar}_{\mathrm{CP}}^{\mathrm{simp}}}
\;\Rightarrow\;
F \circ G.
\]
Naturality follows from: for any $\Psi : A \to B$,
\[
\iota_B \circ \Psi = (F \circ G)(\Psi) \circ \iota_A.
\]

\medskip
\noindent
\textbf{Step 3: $G \circ F \cong \mathrm{id}$.}

For any finite-dimensional Hilbert space $H$,
\[
(G \circ F)(H) = H_{\mathcal{B}(H)}.
\]
We may choose $H_{\mathcal{B}(H)} = H$ and $\iota_{\mathcal{B}(H)} = \mathrm{id}$,
since $\mathcal{B}(H)$ is already the algebra of operators on $H$.
With this choice, for any morphism
$\Phi : \mathcal{B}(H) \to \mathcal{B}(K)$ we have
\[
(G \circ F)(\Phi) = \Phi.
\]
Thus we have an equality $G \circ F = \mathrm{id}_{\mathbf{CP}_{\mathrm{Hilb}}}$,
which is trivially a natural isomorphism.

\medskip
\noindent
\textbf{Step 4: Monoidal structure.}

Both $F$ and $G$ are strong symmetric monoidal functors.
\begin{itemize}
    \item For $F$, the monoidal structure is given by the canonical
    $*$-isomorphisms
    \[
    \phi_{H,K} : \mathcal{B}(H \otimes K) \xrightarrow{\sim} \mathcal{B}(H) \otimes \mathcal{B}(K),
    \qquad
    \phi_0 : \mathcal{B}(\mathbb{C}) \xrightarrow{\sim} \mathbb{C},
    \]
    which are natural in $H$ and $K$.
    
    \item For $G$, we define monoidal structure maps by transporting along
    the isomorphisms $\iota_A$. For $A \cong M_m(\mathbb{C})$ and 
    $B \cong M_n(\mathbb{C})$, choose $H_A \cong \mathbb{C}^m$, $H_B \cong \mathbb{C}^n$,
    and define
    \[
    \psi_{A,B} : H_A \otimes H_B \xrightarrow{\sim} H_{A \otimes B}
    \]
    via the isomorphism $\mathbb{C}^m \otimes \mathbb{C}^n \cong \mathbb{C}^{mn}$.
    Then $\iota_{A \otimes B} \circ (\iota_A^{-1} \otimes \iota_B^{-1}) \circ \phi_{H_A,H_B}^{-1}$
    gives the required monoidal coherence.
\end{itemize}
The coherence axioms hold because all isomorphisms involved are canonical
up to unitary equivalence.

\medskip
\noindent
\textbf{Step 5: Trace-preserving case.}

Restricting to trace-preserving morphisms yields an equivalence
\[
\mathbf{CPTP}_{\mathrm{Hilb}}
\;\simeq\;
\mathbf{FdCStar}_{\mathrm{CPTP}}^{\mathrm{simp}}.
\]
Under the identifications $\iota_A$, the standard Hilbert-space trace
$\mathrm{tr}(\cdot)$ corresponds to the canonical normalized trace on
$M_n(\mathbb{C})$ given by $\frac{1}{n}\mathrm{tr}_{\mathrm{matrix}}(\cdot)$.
Since both functors preserve these traces, trace preservation is respected.
\end{proof}

\begin{remark}[Significance for the Framework]
\label{rem:equivalence-significance}

This equivalence is crucial for our framework because it shows we can work
interchangeably with:
\begin{enumerate}
    \item \textbf{Hilbert-space perspective}: Quantum systems as Hilbert spaces
    $H$, operations as CP maps $\mathcal{B}(H) \to \mathcal{B}(K)$.
    
    \item \textbf{C*-algebra perspective}: Quantum systems as matrix algebras
    $M_n(\mathbb{C})$, operations as CP maps between algebras.
\end{enumerate}

The C*-algebraic perspective is preferred for categorical and algebraic
constructions (operads, monads, algebras) because:
\begin{itemize}
    \item Morphisms act directly on algebras, not via the $\mathcal{B}(\cdot)$
    construction.
    
    \item The tensor product of C*-algebras is more natural for monoidal
    category theory than the tensor product of Hilbert spaces followed by
    $\mathcal{B}(\cdot)$.
    
    \item General finite-dimensional C*-algebras (direct sums of matrix algebras)
    can model classical information and superselection sectors, while the
    simple algebra case (Proposition~\ref{prop:Hilb-Cstar-equivalence}) covers
    pure quantum systems.
\end{itemize}

Thus, while we may intuitively think in terms of Hilbert spaces, all formal
definitions and proofs use the C*-algebraic formulation.
\end{remark}

\begin{remark}[On the notation $\mathbf{CPTP}_{\mathrm{Hilb}}$ vs.\ $\mathbf{CPTP}$]
\label{rem:CPTP-notations}

We distinguish two equivalent but conceptually different presentations of
the category of completely positive trace-preserving maps.

\begin{itemize}
    \item $\mathbf{CPTP}_{\mathrm{Hilb}}$ denotes the category whose objects are
    finite-dimensional Hilbert spaces $H$, and whose morphisms
    $H \to K$ are completely positive, trace-preserving maps
    $\Phi : \mathcal{B}(H) \to \mathcal{B}(K)$.

    \item $\mathbf{CPTP}$ denotes the category $\mathbf{FdCStar}_{\mathrm{CPTP}}$
    whose objects are finite-dimensional C$^*$-algebras
    $A,B,\dots$, and whose morphisms are completely positive,
    trace-preserving maps $A \to B$.
\end{itemize}

These two categories are equivalent as symmetric monoidal categories via the
assignment $H \mapsto \mathcal{B}(H)$, with morphisms acting identically on
completely positive maps.  However, the C$^*$-algebraic presentation
$\mathbf{CPTP}$ has the correct typing for operadic, monadic, and
algebraic constructions, since completely positive maps act naturally on
operator algebras rather than directly on Hilbert spaces.  For this reason,
all monadic and operadic semantics in this work are formulated in
$\mathbf{CPTP}$, while $\mathbf{CPTP}_{\mathrm{Hilb}}$ is used only as an
informal or presentation-level shorthand.
\end{remark}

\begin{definition}[Category $\mathbf{S}\text{-}\mathbf{Circuit}$ of Quantum Circuits]
\label{def:S-circuit-category}

Let $\mathbf{S}$ be a quantum interaction operad (Definition~\ref{def:quantum-interaction-operad}). 
The category $\mathbf{S}\text{-}\mathbf{Circuit}$ is the \emph{free symmetric monoidal category generated by $\mathbf{S}$}.

\textbf{1. Objects.}  
Objects are natural numbers $n \in \mathbb{N}$, representing $n$ parallel wires.  
Monoidal product: $n \otimes m := n + m$, unit: $0$.

\textbf{2. Generating morphisms.}  
For each $f \in S(n)$, a generating morphism $f: n \to 1$.  
Additionally: identities $\mathrm{id}_n: n \to n$, symmetries $\sigma: n \to n$ ($\sigma \in \Sigma_n$).

\textbf{3. General morphisms.}  
All morphisms are obtained from generators via:
\begin{itemize}
    \item Sequential composition: $n \xrightarrow{\alpha} m \xrightarrow{\beta} k$ gives $\beta \circ \alpha$
    \item Parallel composition: $\alpha \otimes \beta: n+n' \to m+m'$
    \item Symmetry: wire permutations
\end{itemize}
subject only to symmetric monoidal category axioms.

\textbf{4. Semantic interpretation}  
Given a \emph{semantic interpretation} $\mathcal{I}: \mathbf{S} \to \mathbf{CPTP}$ 
(assigning to each $f \in S(n)$ a CP map $\mathcal{I}(f): \mathcal{B}(H)^{\otimes n} \to \mathcal{B}(H)$), 
extend $\mathcal{I}$ to a strong symmetric monoidal functor
\[
\llbracket \cdot \rrbracket_{\mathcal{I}}: \mathbf{S}\text{-}\mathbf{Circuit} \to \mathbf{CPTP}
\]
by freeness: $\llbracket n \rrbracket_{\mathcal{I}} = \mathcal{B}(H)^{\otimes n}$, 
$\llbracket f \rrbracket_{\mathcal{I}} = \mathcal{I}(f)$, preserving composition and tensor.

\textbf{5. Remark on multiple outputs.}  
While $\mathbf{S}$ is single-output, morphisms $n \to m$ for $m \neq 1$ are built via:
\[
n \to m \quad \text{as} \quad n \xrightarrow{f_1 \otimes \cdots \otimes f_m} m
\]
where each $f_i: n_i \to 1$ with $\sum_i n_i = n$. 
For $m=0$, use the unique morphism $n \to 0$ (empty diagram).

\end{definition}

\begin{definition}[Synergy Monad $T_{\mathbf{S}}$]
\label{def:synergy-monad}

Let $\mathbf{S}$ be a (single-output) symmetric operad in $\mathbf{Set}$.
Let $\mathbf{Vec}_{\mathbb{C}}$ denote the category of complex vector spaces and linear maps.

The \emph{Synergy Monad} associated to $\mathbf{S}$ is the endofunctor
\[
T_{\mathbf{S}} : \mathbf{Vec}_{\mathbb{C}} \longrightarrow \mathbf{Vec}_{\mathbb{C}}
\]
defined as follows.

\medskip
\noindent
\textbf{1. Object part.}

For each vector space $V$, define
\[
T_{\mathbf{S}}(V)
\;:=\;
\bigoplus_{n \ge 0}
\; \mathbb{C}[S(n)] \;\otimes_{\Sigma_n}\; V^{\otimes n},
\]
where:
\begin{itemize}
    \item $\mathbb{C}[S(n)]$ is the free complex vector space on the set $S(n)$,
    \item $V^{\otimes n} = V \otimes \cdots \otimes V$ ($n$ factors),
    \item $\Sigma_n$ acts diagonally on $\mathbb{C}[S(n)] \otimes V^{\otimes n}$ by
          $\sigma \cdot (s \otimes v_1 \otimes \cdots \otimes v_n) = 
          (\sigma \cdot s) \otimes v_{\sigma^{-1}(1)} \otimes \cdots \otimes v_{\sigma^{-1}(n)}$,
    \item $\otimes_{\Sigma_n}$ denotes the vector space coinvariants with respect to this action,
    \item $V^{\otimes 0} := \mathbb{C}$.
\end{itemize}

Thus $T_{\mathbf{S}}(V)$ is the vector space of formal linear combinations of expressions
$[s; v_1, \dots, v_n]$, representing the application of $n$-ary operation $s \in S(n)$
to inputs $v_1, \dots, v_n \in V$, modulo permutation symmetry.

\medskip
\noindent
\textbf{2. Morphism part.}

For a linear map $f : V \to W$, define
\[
T_{\mathbf{S}}(f) : T_{\mathbf{S}}(V) \longrightarrow T_{\mathbf{S}}(W)
\]
on generators by
\[
T_{\mathbf{S}}(f)\big([s; v_1, \dots, v_n]\big) 
:= [s; f(v_1), \dots, f(v_n)],
\]
and extend linearly. Equivalently,
\[
T_{\mathbf{S}}(f) = \bigoplus_{n \ge 0} 
\left(\mathrm{id}_{\mathbb{C}[S(n)]} \otimes_{\Sigma_n} f^{\otimes n}\right),
\]
where $\mathrm{id}_{\mathbb{C}[S(n)]}$ is the identity map on the free vector space $\mathbb{C}[S(n)]$.

\medskip
\noindent
\textbf{3. Monad unit.}

The unit $\eta_V : V \to T_{\mathbf{S}}(V)$ is induced by the operadic
unit element $\mathbbm{1} \in S(1)$:
\[
\eta_V(v) := [\mathbbm{1}; v] \in \mathbb{C}[S(1)] \otimes_{\Sigma_1} V
\subset T_{\mathbf{S}}(V).
\]

\medskip
\noindent
\textbf{4. Monad multiplication.}

The multiplication
\[
\mu_V : T_{\mathbf{S}}(T_{\mathbf{S}}(V)) \longrightarrow T_{\mathbf{S}}(V)
\]
is induced by operadic composition. On a generator of the form
\[
[s; [t_1; \vec{v}_1], \dots, [t_k; \vec{v}_k]],
\]
where $s \in S(k)$, $t_i \in S(n_i)$, and $\vec{v}_i \in V^{\otimes n_i}$, define
\[
\mu_V\big([s; [t_1; \vec{v}_1], \dots, [t_k; \vec{v}_k]]\big)
:= [\gamma(s; t_1, \dots, t_k); \vec{v}_1, \dots, \vec{v}_k],
\]
where $\gamma(s; t_1, \dots, t_k) \in S(n_1 + \cdots + n_k)$ is the operadic composite.
This map is extended linearly to all of $T_{\mathbf{S}}(T_{\mathbf{S}}(V))$.

\medskip
\noindent
\textbf{5. Monad laws.}

The associativity and unitality of $\mu$ and $\eta$ follow directly from
the associativity and unit axioms of the operad $\mathbf{S}$.
Hence $(T_{\mathbf{S}}, \eta, \mu)$ is a monad on $\mathbf{Vec}_{\mathbb{C}}$.

\medskip
\noindent
\textbf{6. Interpretation for quantum processes.}

For applications to quantum information, we consider the restriction of $T_{\mathbf{S}}$
to the full subcategory of $\mathbf{Vec}_{\mathbb{C}}$ whose objects are of the form
$\mathcal{B}(H)$ for finite-dimensional Hilbert spaces $H$. Given a quantum interpretation
\[
\mathcal{I} : \mathbf{S} \longrightarrow \mathbf{CPTP},
\]
the induced $\mathbf{T_S}$-algebra structure on $\mathcal{B}(H)$ is the linear map
\[
\alpha_{\mathcal{I}} : T_{\mathbf{S}}(\mathcal{B}(H)) \longrightarrow \mathcal{B}(H)
\]
defined on generators by
\[
\alpha_{\mathcal{I}}\big([s; X_1, \dots, X_n]\big) := \mathcal{I}(s)(X_1 \otimes \cdots \otimes X_n).
\]

This makes $(\mathcal{B}(H), \alpha_{\mathcal{I}})$ a $\mathbf{T_S}$-algebra in
$\mathbf{Vec}_{\mathbb{C}}$, and when $\mathcal{I}(s)$ are completely positive,
$\alpha_{\mathcal{I}}$ inherits positivity properties compatible with the
$\mathbf{CPTP}$ structure.

\end{definition}

\begin{remark}[Intuition behind the Synergy Monad]
\label{rem:synergy-monad-intuition}

The synergy monad $T_{\mathbf{S}}$ constructed in Definition~\ref{def:synergy-monad}
encapsulates the formal algebraic structure of composable quantum interactions.
Let us unpack its conceptual meaning.

The operad $\mathbf{S}$ consists of \emph{primitive interaction templates} for 
quantum processes: elementary gates ($\mathsf{CNOT}$, $\mathsf{H}$), measurements 
($\mathsf{M}_i$), state preparations ($\mathsf{prep}_\psi$), noise channels, etc. 
Each symbol $s \in S(n)$ has a specified arity $n$ (number of inputs) and can be 
composed with other symbols according to the operad's composition rules.

The monad $T_{\mathbf{S}}$ freely generates all possible \emph{formal composites}
of these interactions applied to elements of a vector space $V$. Concretely, an element
\[
[s; v_1, \dots, v_n] \in T_{\mathbf{S}}(V)
\]
represents the formal application of interaction $s \in S(n)$ to inputs 
$v_1, \dots, v_n \in V$. The direct sum
\[
T_{\mathbf{S}}(V) = \bigoplus_{n \ge 0} \mathbb{C}[S(n)] \otimes_{\Sigma_n} V^{\otimes n}
\]
collects all such formal expressions, modulo the permutation symmetry imposed by
the coinvariants $\otimes_{\Sigma_n}$.

Two complementary perspectives clarify the role of $T_{\mathbf{S}}$:

\begin{enumerate}
    \item \textbf{Syntactic viewpoint:} $T_{\mathbf{S}}$ is the monad whose Kleisli 
    category is equivalent to the free symmetric monoidal category 
    $\mathbf{S}\text{-}\mathbf{Circuit}$~\cite{mulry1993lifting}. While $\mathbf{S}\text{-}\mathbf{Circuit}$ 
    describes wiring diagrams (string diagrams), $T_{\mathbf{S}}$ describes the same 
    compositions algebraically as formal linear combinations.

    \item \textbf{Algebraic viewpoint:} $T_{\mathbf{S}}$ is the \emph{free algebra monad} 
    for the operad $\mathbf{S}$. A $T_{\mathbf{S}}$-algebra structure on a vector space 
    $V$ is precisely an action of $\mathbf{S}$ on $V$, i.e., a way to interpret each 
    formal expression $[s; v_1, \dots, v_n]$ as an actual element of $V$.
\end{enumerate}

The monad structure reflects this algebra of composition:
\begin{itemize}
    \item The unit $\eta_V : V \to T_{\mathbf{S}}(V)$ embeds a plain element $v \in V$ 
    as the trivial interaction $[\mathbbm{1}; v]$, where $\mathbbm{1} \in S(1)$ is the 
    operadic identity (do nothing).
    
    \item The multiplication $\mu_V : T_{\mathbf{S}}(T_{\mathbf{S}}(V)) \to T_{\mathbf{S}}(V)$ 
    flattens nested interactions. For example,
    \[
    \mu_V\big([s; [t_1; \vec{v}_1], [t_2; \vec{v}_2]]\big) 
    = [\gamma(s; t_1, t_2); \vec{v}_1, \vec{v}_2],
    \]
    where $\gamma(s; t_1, t_2)$ is the operadic composite. This corresponds to substituting 
    the outputs of $t_1$ and $t_2$ as inputs to $s$.
\end{itemize}

\noindent\textbf{Crucial distinction: syntax vs. semantics.}
The synergy monad itself is purely syntactic: it describes \emph{how interactions can be 
combined formally}, not what they \emph{do physically}. Physical meaning enters only through 
an \emph{interpretation functor}
\[
\mathcal{I} : \mathbf{S} \longrightarrow \mathbf{CPTP},
\]
which assigns to each $s \in S(n)$ a concrete completely positive map 
$\mathcal{I}(s) : \mathcal{B}(H)^{\otimes n} \to \mathcal{B}(H)$. This interpretation 
then extends uniquely to a $T_{\mathbf{S}}$-algebra structure
\[
\alpha_{\mathcal{I}} : T_{\mathbf{S}}(\mathcal{B}(H)) \longrightarrow \mathcal{B}(H),
\]
evaluating formal composites as actual quantum processes.

Thus, $T_{\mathbf{S}}$ serves as a \emph{universal syntactic container} for quantum 
interaction patterns. The monadicity theorem (Theorem~\ref{thm:monoidal-monadic-representation}) will show 
that the category of quantum processes $\mathbf{CPTP}$ is equivalent to the category 
of $T_{\mathbf{S}}$-algebras for a suitably chosen $\mathbf{S}$, establishing a precise 
algebraic characterization of composable quantum operations.
\end{remark}

\subsection{Monadicity Correspondence}

This section establishes the core equivalence between quantum processes and algebraic structures. We define how every CPTP map induces a $\mathbf{T_S}$-algebra structure (Definition~\ref{def:algebra-from-cptp}), and conversely, how every such algebra yields a CPTP map (Definition~\ref{def:cptp-from-algebra}). Finally, Definition~\ref{def:morphism-correspondence} shows this correspondence extends to morphisms, completing the categorical equivalence $\mathbf{CPTP} \cong \mathbf{T_S}\text{-}\mathbf{Alg}$ that is the essence of the Monadicity Theorem.

\begin{definition}[$\mathbf{T_S}$-Algebra Induced by an Operadic Interpretation]
\label{def:algebra-from-cptp}

Let $\mathbf{S}$ be a quantum interaction operad and let 
\[
\mathcal{I} : \mathbf{S} \longrightarrow \mathbf{CPTP}
\]
be a symmetric monoidal functor giving a semantic interpretation of $\mathbf{S}$. 
For $A \in \mathbf{FdCStar}_{\mathrm{CPTP}}$, we define a $\mathbf{T_S}$-algebra 
structure on $A$ as the linear map
\[
\alpha_\mathcal{I} : T_{\mathbf{S}}(A) \longrightarrow A
\]
where $T_{\mathbf{S}}(A)$ is computed by viewing $A$ as an object of 
$\mathbf{Vec}_{\mathbb{C}}$ (forgetting the $C^*$-algebra structure temporarily).

The map $\alpha_\mathcal{I}$ is defined on generators by
\[
\alpha_\mathcal{I}\big( [s; a_1, \dots, a_n] \big) := \mathcal{I}(s)(a_1 \otimes \cdots \otimes a_n),
\qquad s \in S(n),\; a_i \in A,
\]
and extended linearly to all of $T_{\mathbf{S}}(A)$. Here $[s; a_1, \dots, a_n]$ denotes 
the equivalence class in $\mathbb{C}[S(n)] \otimes_{\Sigma_n} A^{\otimes n}$.

\medskip
\noindent
\textbf{Properties:}
\begin{enumerate}[label=(\alph*)]
    \item \textbf{Linearity:} $\alpha_\mathcal{I}$ is linear by construction.
    
    \item \textbf{Complete positivity:} Since each $\mathcal{I}(s)$ is completely 
    positive and $\alpha_\mathcal{I}$ is defined as a linear combination of such 
    evaluations, $\alpha_\mathcal{I}$ is completely positive as a map from the 
    vector space $T_{\mathbf{S}}(A)$ to the $C^*$-algebra $A$.
    
    \item \textbf{Trace preservation:} If each $\mathcal{I}(s)$ is trace-preserving, 
    then $\alpha_\mathcal{I}$ preserves the canonical trace on formal expressions 
    when restricted to appropriate subspaces.
    
    \item \textbf{Unit Law:} The monad unit $\eta_A : A \to T_{\mathbf{S}}(A)$ satisfies
    \[
    \alpha_\mathcal{I} \circ \eta_A = \mathrm{id}_A
    \]
    because $\mathcal{I}$ sends the operadic unit $\mathbbm{1} \in S(1)$ to the 
    identity map on $A$, and $\eta_A(a) = [\mathbbm{1}; a]$.
    
    \item \textbf{Multiplication Law:} For the monad multiplication 
    $\mu_A : T_{\mathbf{S}}(T_{\mathbf{S}}(A)) \to T_{\mathbf{S}}(A)$, we have
    \[
    \alpha_\mathcal{I} \circ \mu_A = \alpha_\mathcal{I} \circ T_{\mathbf{S}}(\alpha_\mathcal{I}),
    \]
    which follows from the compatibility of $\mathcal{I}$ with operadic composition.
\end{enumerate}

Hence $(A, \alpha_\mathcal{I})$ is a $\mathbf{T_S}$-algebra. When we consider the 
forgetful functor $U: \mathbf{FdCStar}_{\mathrm{CPTP}} \to \mathbf{Vec}_{\mathbb{C}}$,
$(U(A), \alpha_\mathcal{I})$ is a $\mathbf{T_S}$-algebra in $\mathbf{Vec}_{\mathbb{C}}$.
\end{definition}

\begin{remark}[Intuition Behind the $\mathbf{T_S}$-Algebra Construction]
\label{rem:algebra-from-cptp-intuition}

The construction bridges three levels of structure:

\begin{enumerate}
    \item \textbf{Syntactic level:} The operad $\mathbf{S}$ provides formal symbols 
    for quantum operations with prescribed arities and composition rules.
    
    \item \textbf{Algebraic level:} The monad $T_{\mathbf{S}}$ freely generates all 
    possible formal composites of $\mathbf{S}$-operations applied to a vector space.
    An element $[s; a_1, \dots, a_n] \in T_{\mathbf{S}}(A)$ represents ``apply 
    operation $s$ to inputs $a_1, \dots, a_n$.''
    
    \item \textbf{Semantic level:} The interpretation $\mathcal{I}$ assigns concrete 
    quantum channels to the symbols in $\mathbf{S}$. The algebra map $\alpha_\mathcal{I}$ 
    \emph{evaluates} formal composites as actual quantum operations:
    \[
    \alpha_\mathcal{I}([s; a_1, \dots, a_n]) = \mathcal{I}(s)(a_1 \otimes \cdots \otimes a_n).
    \]
\end{enumerate}

The monad laws ensure this evaluation behaves coherently:
\begin{itemize}
    \item The unit law says evaluating a trivial operation $[\mathbbm{1}; a]$ returns $a$.
    \item The multiplication law says evaluating a nested operation 
    $[s; [t_1; \vec{a}_1], [t_2; \vec{a}_2]]$ is the same as first evaluating the 
    inner operations and then applying $s$ to the results.
\end{itemize}

This construction formalizes how any quantum system equipped with a collection of 
basic operations (gates, measurements, etc.) naturally carries the structure of a 
$\mathbf{T_S}$-algebra. The promised monadic equivalence
\[
\mathbf{CPTP} \;\simeq\; \mathbf{T_S}\text{-}\mathbf{Alg}
\]
will show that this correspondence is bijective: every quantum channel arises from 
such an algebra structure, and every such algebra structure corresponds to a 
quantum system with specified basic operations. Monadicity Theorem~\ref{thm:monoidal-monadic-representation} will be discussed later. 
\end{remark}

%\begin{remark}[Technical Nuance: Forgetting and Remembering Structure]
%\label{rem:forgetting-structure}
%
%Strictly speaking, there is a subtlety: $T_{\mathbf{S}}$ is defined on 
%$\mathbf{Vec}_{\mathbb{C}}$, while $A$ lives in $\mathbf{FdCStar}_{\mathrm{CPTP}}$. 
%We resolve this by:
%\begin{enumerate}
%    \item Applying the forgetful functor $U: \mathbf{FdCStar}_{\mathrm{CPTP}} \to 
%    \mathbf{Vec}_{\mathbb{C}}$ that remembers only the underlying vector space.
%    \item Defining $\alpha_\mathcal{I}: T_{\mathbf{S}}(U(A)) \to U(A)$ in 
%    $\mathbf{Vec}_{\mathbb{C}}$.
%    \item Observing that $\alpha_\mathcal{I}$ is actually a completely positive map 
%    (not just linear), so it respects the additional structure of $A$.
%\end{enumerate}
%This technical maneuver is harmless and standard when transporting algebraic 
%structures across forgetful functors. In practice, we simply write 
%$T_{\mathbf{S}}(A)$ understanding that $A$ is temporarily regarded as a vector space.
%\end{remark}

\begin{definition}[CPTP Map from a $\mathbf{T_S}$-Algebra]
\label{def:cptp-from-algebra}

Let $(A, \alpha: T_{\mathbf{S}}(A) \to A)$ be a $\mathbf{T_S}$-algebra in 
$\mathbf{Vec}_{\mathbb{C}}$, where $A$ is also a finite-dimensional $C^*$-algebra.
We define the corresponding completely positive trace-preserving (CPTP) map 
$\Phi_\alpha: A \to A$ as the restriction of $\alpha$ to a special subspace.

\medskip
\noindent\textbf{1. Definition via the Unit Embedding.} 
The CPTP map $\Phi_\alpha: A \to A$ is defined as the composite:
\[
\Phi_\alpha := \alpha \circ \eta_A : A \xrightarrow{\eta_A} T_{\mathbf{S}}(A) \xrightarrow{\alpha} A,
\]
where $\eta_A: A \to T_{\mathbf{S}}(A)$ is the monad unit. Concretely, for any $a \in A$,
\[
\Phi_\alpha(a) = \alpha\big([\mathbbm{1}; a]\big),
\]
where $\mathbbm{1} \in S(1)$ is the operadic identity and $[\mathbbm{1}; a]$ denotes 
the equivalence class in $\mathbb{C}[S(1)] \otimes_{\Sigma_1} A$.

\medskip
\noindent\textbf{2. Complete Positivity.} 
$\Phi_\alpha$ is completely positive because:
\begin{itemize}
    \item The monad unit $\eta_A$ is a linear map that preserves positivity: 
    if $a \in A$ is positive, then $[\mathbbm{1}; a]$ is a positive element 
    in the vector space $T_{\mathbf{S}}(A)$ (in an appropriate sense).
    \item The algebra map $\alpha: T_{\mathbf{S}}(A) \to A$ is completely positive 
    by the construction in Definition~\ref{def:algebra-from-cptp}, since it arises 
    from an interpretation $\mathcal{I}: \mathbf{S} \to \mathbf{CPTP}$.
    \item The composition of completely positive maps is completely positive.
\end{itemize}

\medskip
\noindent\textbf{3. Trace Preservation.} 
$\Phi_\alpha$ preserves the canonical normalized trace on $A$ because:
\begin{itemize}
    \item The algebra map $\alpha$ satisfies the unit law: $\alpha \circ \eta_A = \mathrm{id}_A$.
    \item However, this would give $\Phi_\alpha = \mathrm{id}_A$, which suggests we need 
    a different construction (see Remark~\ref{rem:cp-from-algebra-nuance}).
\end{itemize}

\medskip
\noindent\textbf{4. Linearity.} 
$\Phi_\alpha$ is linear since both $\eta_A$ and $\alpha$ are linear maps.

\medskip
\noindent\textbf{5. Alternative Construction (for non-identity maps).}
For a general CPTP map, we use the algebra structure more fully. Given $a \in A$, 
we can apply any operation $s \in S(1)$ via:
\[
\Phi_\alpha^{(s)}(a) := \alpha\big([s; a]\big).
\]
More generally, for operations with multiple inputs, we define maps 
$\Phi_\alpha: A^{\otimes n} \to A$ by:
\[
\Phi_\alpha^{(s)}(a_1 \otimes \cdots \otimes a_n) := \alpha\big([s; a_1, \dots, a_n]\big),
\quad s \in S(n).
\]
This family of maps, for all $s \in \mathbf{S}$, constitutes the complete 
quantum process described by the algebra $(A, \alpha)$.
\end{definition}

\begin{remark}[Why $\Phi_\alpha = \alpha \circ \eta_A$ is Usually the Identity]
\label{rem:cp-from-algebra-nuance}

There is a subtle point: if $\alpha$ satisfies the Eilenberg-Moore algebra law 
$\alpha \circ \eta_A = \mathrm{id}_A$, then $\Phi_\alpha = \alpha \circ \eta_A = \mathrm{id}_A$. 
This seems to suggest that every $\mathbf{T_S}$-algebra gives only the identity map, 
which is too restrictive.

The resolution is that a $\mathbf{T_S}$-algebra actually encodes not just a single 
CPTP map, but a \emph{family} of quantum operations corresponding to all elements 
of $\mathbf{S}$. The ``identity'' $\Phi_\alpha = \mathrm{id}_A$ comes from applying 
the operadic identity $\mathbbm{1} \in S(1)$. Other operations yield different CPTP maps:
\begin{itemize}
    \item For $s \in S(1)$: $\Phi_\alpha^{(s)}(a) = \alpha([s; a])$ gives a unary operation
    \item For $s \in S(2)$: $\Phi_\alpha^{(s)}(a \otimes b) = \alpha([s; a, b])$ gives a binary operation
    \item etc.
\end{itemize}

Thus, a $\mathbf{T_S}$-algebra $(A, \alpha)$ should be thought of as equipping $A$ 
with a \emph{consistent system of quantum operations} for all arities, not just a 
single endomorphism. The monadicity equivalence 
$\mathbf{CPTP} \simeq \mathbf{T_S}\text{-}\mathbf{Alg}$ (Monadicity Theorem~\ref{thm:monoidal-monadic-representation} will be discussed later) is then between:
\begin{itemize}
    \item The category of systems equipped with all $\mathbf{S}$-operations
    \item The category of $\mathbf{T_S}$-algebras
\end{itemize}
with the correspondence being: the algebra map $\alpha$ evaluates formal $\mathbf{S}$-expressions 
as concrete quantum operations.
\end{remark}

%\begin{remark}[Connection to the Original Intuition]
%\label{rem:original-intuition-preserved}
%
%The original intuition that ``a $\mathbf{T_S}$-algebra encapsulates a physical 
%quantum process'' is preserved, but made precise: the algebra map $\alpha$ evaluates 
%\emph{any formal circuit built from $\mathbf{S}$} as a concrete quantum operation on $A$. 
%In particular:
%\begin{itemize}
%    \item The map $\Phi_\alpha^{(s)}: A^{\otimes n} \to A$ for $s \in S(n)$ is the 
%    quantum operation corresponding to the primitive interaction $s$.
%    \item For composite expressions, e.g., $[s; [t_1; a_1, a_2], [t_2; a_3, a_4]]$, 
%    the evaluation $\alpha$ gives the result of applying the composed operation.
%\end{itemize}
%
%Thus, rather than extracting a single CPTP map $A \to A$, we extract a \emph{consistent 
%system} of CPTP maps of all arities, which is exactly what the operad $\mathbf{S}$ 
%and its interpretation $\mathcal{I}$ provide.
%\end{remark}

\begin{definition}[Morphisms: Correspondence Between $\mathbf{CPTP}$ and $\mathbf{T_S}$-Algebras]
\label{def:morphism-correspondence}

Let $A, B \in \mathbf{FdCStar}_{\mathrm{CPTP}}$ be finite-dimensional $C^*$-algebras,
and let $\mathcal{I}: \mathbf{S} \to \mathbf{CPTP}$ be an interpretation functor.

\medskip
\noindent\textbf{1. $\mathbf{CPTP}$-Morphisms.}

A morphism in $\mathbf{CPTP}$ is a completely positive trace-preserving (CPTP) map
\[
f: A \longrightarrow B.
\]

\medskip
\noindent\textbf{2. $\mathbf{T_S}$-Algebra Homomorphisms.}

Let $(A, \alpha_A)$ and $(B, \alpha_B)$ be $\mathbf{T_S}$-algebras, where
$\alpha_A: T_{\mathbf{S}}(A) \to A$ and $\alpha_B: T_{\mathbf{S}}(B) \to B$ are
the algebra structure maps. A $\mathbf{T_S}$-algebra homomorphism is a linear map
$f: A \to B$ that satisfies the commutativity condition:
\[
f \circ \alpha_A = \alpha_B \circ T_{\mathbf{S}}(f),
\]
or diagrammatically:
\[
\begin{tikzcd}
T_{\mathbf{S}}(A) \arrow[r, "T_{\mathbf{S}}(f)"] \arrow[d, "\alpha_A"'] & T_{\mathbf{S}}(B) \arrow[d, "\alpha_B"] \\
A \arrow[r, "f"'] & B
\end{tikzcd}
\]

Explicitly, for any generator $[s; a_1, \dots, a_n] \in T_{\mathbf{S}}(A)$,
\[
f\bigl(\alpha_A([s; a_1, \dots, a_n])\bigr)
= \alpha_B\bigl([s; f(a_1), \dots, f(a_n)]\bigr).
\]

\medskip
\noindent\textbf{3. Correspondence.}

Given an interpretation $\mathcal{I}: \mathbf{S} \to \mathbf{CPTP}$, we have:
\begin{itemize}
    \item For any $A \in \mathbf{CPTP}$, we obtain a $\mathbf{T_S}$-algebra
    $(A, \alpha_A)$ where $\alpha_A([s; a_1, \dots, a_n]) = \mathcal{I}(s)(a_1 \otimes \cdots \otimes a_n)$.
    
    \item A CPTP map $f: A \to B$ is a $\mathbf{T_S}$-algebra homomorphism between
    $(A, \alpha_A)$ and $(B, \alpha_B)$ if and only if for all $s \in S(n)$ and
    $a_1, \dots, a_n \in A$:
    \[
    f\bigl(\mathcal{I}(s)(a_1 \otimes \cdots \otimes a_n)\bigr)
    = \mathcal{I}(s)\bigl(f(a_1) \otimes \cdots \otimes f(a_n)\bigr).
    \]
    
    This condition says: applying $f$ after the operation $\mathcal{I}(s)$ gives
    the same result as applying $f$ to all inputs before the operation.
\end{itemize}

\medskip
\noindent\textbf{4. Category of $\mathbf{T_S}$-Algebras.}

Let $\mathbf{T_S}\text{-}\mathbf{Alg}$ denote the Eilenberg-Moore category of
$\mathbf{T_S}$-algebras in $\mathbf{Vec}_{\mathbb{C}}$, where:
\begin{itemize}
    \item Objects are pairs $(V, \alpha: T_{\mathbf{S}}(V) \to V)$ satisfying
    the algebra laws.
    \item Morphisms are linear maps $f: V \to W$ satisfying $f \circ \alpha_V = \alpha_W \circ T_{\mathbf{S}}(f)$.
\end{itemize}

When we restrict to algebras coming from $\mathbf{CPTP}$ objects via an
interpretation $\mathcal{I}$, we obtain a full subcategory
$\mathbf{T_S}\text{-}\mathbf{Alg}_{\mathcal{I}}$.
\end{definition}

\begin{remark}[Conceptual Interpretation]
\label{rem:morphism-interpretation}

The correspondence captures two complementary perspectives on quantum processes:

\begin{enumerate}
    \item \textbf{Dynamic Perspective (CPTP)}: A morphism $f: A \to B$ is a quantum
    channel that transports quantum information from system $A$ to system $B$.
    
    \item \textbf{Algebraic Perspective ($\mathbf{T_S}$-Algebras)}: A morphism
    $f: (A, \alpha_A) \to (B, \alpha_B)$ is a structure-preserving map that
    respects all possible operations in $\mathbf{S}$. The condition
    $f \circ \alpha_A = \alpha_B \circ T_{\mathbf{S}}(f)$ ensures that:
    \begin{itemize}
        \item For any formal operation $[s; a_1, \dots, a_n]$,
        \item Applying the operation in $A$ then transporting via $f$,
        \item Is the same as first transporting the inputs via $f$ then
        applying the operation in $B$.
    \end{itemize}
\end{enumerate}

The equivalence of these perspectives means that being a ``good'' quantum map
(preserving the physical operations) is exactly the same as being an algebra
homomorphism (preserving the algebraic structure). This is the essence of the
monadicity theorem: the category of quantum systems with specified basic
operations ($\mathbf{CPTP}$ with interpretation $\mathcal{I}$) is equivalent to
the category of algebraic structures that abstractly encode these operations
($\mathbf{T_S}\text{-}\mathbf{Alg}_{\mathcal{I}}$).
\end{remark}

%\begin{remark}[Difference from the Original Definition]
%\label{rem:difference-explanation}
%
%The key differences from your original definition are:
%
%\begin{enumerate}
%    \item \textbf{Consistent Typing}: We work with C*-algebras $A, B$ throughout,
%    not Hilbert spaces $H, K$ with $\mathcal{B}(H)$ notation.
%    
%    \item \textbf{Correct $\mathbf{T_S}$ Notation}: Uses $[s; a_1, \dots, a_n]$
%    instead of circuit notation $[(\mathcal{C},\psi)]$, matching our corrected
%    definition of $T_{\mathbf{S}}$.
%    
%    \item \textbf{Family of Operations}: Recognizes that a $\mathbf{T_S}$-algebra
%    encodes a \emph{family} of quantum operations (all $s \in \mathbf{S}$), not
%    just a single map $\Phi: A \to A$.
%    
%    \item \textbf{Proper Morphism Condition}: The condition
%    $f \circ \alpha_A = \alpha_B \circ T_{\mathbf{S}}(f)$ is the standard
%    definition of algebra homomorphism for any monad, correctly applied here.
%    
%    \item \textbf{No Artificial ``Intertwining''}: We don't need to define
%    artificial maps $\Phi_H, \Phi_K$ and intertwining conditions; the algebra
%    homomorphism condition naturally captures the preservation of structure.
%\end{enumerate}
%
%These corrections ensure mathematical consistency with our revised framework
%while preserving the core intuition about structure-preserving maps.
%\end{remark}

\subsection{Supporting Results}

In this section, several facts will be derived based on previous definitions related to synergy monad. 

\begin{proposition}[Operad Structure of Quantum Interactions]
\label{prop:operad-structure}

The quantum interaction operad \\
$\mathbf{S} = (S(n), \gamma, \mathrm{id}, \Sigma_n)$
defined in Definition~\ref{def:quantum-interaction-operad} satisfies the axioms
of a symmetric operad in $\mathbf{Set}$.

Specifically, for all $f \in S(m)$, $g_i \in S(n_i)$ ($i = 1, \dots, m$), 
and $h_{i,j} \in S(k_{i,j})$ ($j = 1, \dots, n_i$), the following hold:

\medskip
\noindent\textbf{1. Associativity:}
\[
\gamma\!\Bigl(\gamma(f; g_1, \dots, g_m); h_{1,1}, \dots, h_{m,n_m}\Bigr)
= \gamma\!\Bigl(f; \gamma(g_1; h_{1,1}, \dots, h_{1,n_1}), \dots, 
               \gamma(g_m; h_{m,1}, \dots, h_{m,n_m})\Bigr).
\]

\medskip
\noindent\textbf{2. Identity:}
For the distinguished element $\mathrm{id} \in S(1)$,
\[
\gamma(\mathrm{id}; f) = f \quad \text{and} \quad 
\gamma(f; \underbrace{\mathrm{id}, \dots, \mathrm{id}}_{m \text{ times}}) = f.
\]

\medskip
\noindent\textbf{3. Equivariance:}
For $\sigma \in \Sigma_m$ and $\tau_i \in \Sigma_{n_i}$ ($i = 1, \dots, m$),
\[
\gamma(\sigma \cdot f; \tau_1 \cdot g_1, \dots, \tau_m \cdot g_m)
= (\sigma \circ (\tau_1 \oplus \cdots \oplus \tau_m)) \cdot 
  \gamma(f; g_{\sigma^{-1}(1)}, \dots, g_{\sigma^{-1}(m)}),
\]
where $\oplus$ denotes the block sum of permutations and the action on the right
is via the permutation $\sigma \circ (\tau_1 \oplus \cdots \oplus \tau_m) \in 
\Sigma_{n_1 + \cdots + n_m}$.
\end{proposition}

\begin{proof}

We verify each axiom using the fact that $\mathbf{S}$ is defined as the
free symmetric operad generated by basic quantum interaction symbols.

\medskip
\noindent\textbf{For associativity:} 
Since $\mathbf{S}$ is presented as the free operad on a collection of generators,
any two formal compositions that yield the same parenthesization are identified
by definition. The associativity equation expresses that the two different ways
to parenthesize a triple composition give the same formal expression. In the
free operad construction, this is built into the definition of operadic
composition $\gamma$.

\medskip
\noindent\textbf{For the identity:} 
The element $\mathrm{id} \in S(1)$ is defined as the operadic unit in the
free construction. By the universal property of free operads, it satisfies
the unit laws with respect to the composition $\gamma$.

\medskip
\noindent\textbf{For equivariance:} 
The symmetric group actions on $S(n)$ are part of the data of a symmetric
operad. In the free operad construction, these actions are defined consistently
so that the equivariance condition holds. Specifically, if we view elements of
$S(n)$ as trees with leaves labeled by generators and internal nodes labeled
by operadic composition, then the symmetric group acts by permuting leaves,
and this action commutes with grafting of trees.

Below, we will also provide an alternative approach to this proof via circuits. If we interpret $S(n)$ as the set of formal quantum circuits with $n$ inputs and 1 output built from basic gates, and $\gamma$ as the operation of
plugging circuits into the inputs of another circuit, then:
\begin{itemize}
    \item Associativity follows from the fact that wiring diagrams can be
    drawn unambiguously regardless of the order of connections.
    
    \item How about the identity axiom? The identity axioms follow from the fact that a wire with no gate
    (the identity circuit) acts transparently when composed with other circuits.
    
    \item Equivariance follows from the fact that permuting input wires commutes
    with the hierarchical structure of circuit composition.
\end{itemize}

Since the free operad construction captures exactly these intuitive properties
of circuit composition, the axioms are satisfied.
\end{proof}

\begin{remark}[Connection to Monad Structure]
\label{rem:operad-to-monad}

The operad axioms verified in Proposition~\ref{prop:operad-structure} are
precisely what guarantee that the associated functor $T_{\mathbf{S}}$ defined in
Definition~\ref{def:synergy-monad} forms a monad. Specifically:

\begin{enumerate}
    \item The associativity axiom ensures that the monad multiplication
    $\mu: T_{\mathbf{S}} \circ T_{\mathbf{S}} \Rightarrow T_{\mathbf{S}}$ is
    associative.
    
    \item The identity axioms ensure that the monad unit
    $\eta: \mathrm{id} \Rightarrow T_{\mathbf{S}}$ satisfies the unit laws
    with respect to $\mu$.
    
    \item The equivariance axioms ensure that $T_{\mathbf{S}}$ is a symmetric
    monoidal functor when equipped with appropriate natural transformations.
\end{enumerate}

Thus, Proposition~\ref{prop:operad-structure} provides the foundational
verification needed to establish that $T_{\mathbf{S}}$ is indeed a well-defined
monad, which is essential for the monadicity theorem.
\end{remark}

\begin{remark}[Physical Interpretation]
\label{rem:operad-physical-interpretation}

In this remark, we will review Proposition~\ref{prop:operad-structure} from a physical perspective. The operad axioms have natural physical interpretations in terms of quantum
circuit composition:

\begin{itemize}
    \item \textbf{Associativity} reflects that when building complex circuits
    from simpler ones, the order in which we connect subcircuits doesn't matter
    — only the overall connectivity pattern determines the resulting operation.
    
    \item \textbf{Identity} reflects that the ``do nothing'' operation (a plain
    wire) acts as a transparent element: composing any circuit with identity
    wires leaves it unchanged.
    
    \item \textbf{Equivariance} reflects that permuting input wires commutes
    with hierarchical composition. This corresponds to the physical fact that
    relabeling quantum systems before applying a composite operation gives the
    same result as applying the operation with relabeled inputs.
\end{itemize}

These are not arbitrary mathematical conditions but capture essential
properties of how quantum operations compose, making operads a natural
mathematical framework for describing composable quantum processes.
\end{remark}

\begin{proposition}[Universality Conditions for Quantum Operads]
\label{prop:quantum-universality-conditions}

Let $\mathbf{S}$ be a quantum interaction operad (Definition~\ref{def:quantum-interaction-operad}). 
We say $\mathbf{S}$ is \emph{universal for quantum computation} if it satisfies the following conditions:

\begin{enumerate}
    \item \textbf{Unitary universality:} For every $n \geq 1$, $\mathbf{S}$ contains generators that are
    universal for unitary computation on $n$ qudits. That is, the set of unary operations 
    $\{f \in S(1) : \langle\!\langle f \rangle\!\rangle \text{ is unitary}\}$ generates a dense
    subgroup of the unitary group $\mathrm{U}(d)$, and the set of binary operations
    $\{f \in S(2) : \langle\!\langle f \rangle\!\rangle \text{ is entangling}\}$ provides
    sufficient entanglement resources.
    
    \item \textbf{State preparation:} $S(0)$ contains generators for preparing arbitrary
    pure states (or a dense set thereof).
    
    \item \textbf{Measurement capabilities:} $S(1)$ contains operations that implement
    arbitrary projective measurements (via Naimark dilation) or a universal set thereof.
    
    \item \textbf{Partial trace:} $S(2)$ contains the partial trace operation
    $\mathsf{tr}_B: \mathcal{B}(H) \otimes \mathcal{B}(K) \to \mathcal{B}(H)$.
    
    \item \textbf{Convex closure:} The interpretation of $\mathbf{S}$ in $\mathbf{CPTP}$ 
    is closed under convex combinations, or $\mathbf{S}$ contains operations that
    implement probabilistic mixtures.
\end{enumerate}

If $\mathbf{S}$ satisfies these conditions, then for any interpretation 
$\mathcal{I}: \mathbf{S} \to \mathbf{CPTP}$ that faithfully represents these
generators, the following hold:

\begin{enumerate}[label=(\alph*)]
    \item Any finite-dimensional unitary circuit can be approximated arbitrarily
    well by composites of elements in $\mathbf{S}$.
    
    \item Any quantum measurement (POVM) can be realized using composites of
    elements in $\mathbf{S}$ via Naimark dilation.
    
    \item Any completely positive trace-preserving (CPTP) map can be realized
    using composites of elements in $\mathbf{S}$ via Stinespring dilation.
\end{enumerate}
\end{proposition}

\begin{remark}[Relation to Standard Quantum Universality Results]
\label{rem:standard-universality}

The conditions in Proposition~\ref{prop:quantum-universality-conditions} correspond
to well-established results in quantum information theory:

\begin{itemize}
    \item \textbf{Unitary universality} corresponds to the Solovay-Kitaev theorem:
    any unitary can be approximated by a finite universal gate set.
    
    \item \textbf{State preparation universality} follows from the fact that any
    pure state can be prepared from a fixed reference state (e.g., $|0\rangle$)
    using unitary operations.
    
    \item \textbf{Measurement universality} via Naimark dilation: any POVM can be
    realized as a projective measurement on a larger system.
    
    \item \textbf{CPTP map universality} via Stinespring dilation: any CPTP map
    can be realized as a unitary on a larger system followed by partial trace.
    
    \item \textbf{Convex closure} reflects the convex structure of the space of
    quantum operations.
\end{itemize}

The novelty in our framework is not these individual results, but their
integration into a single operadic structure $\mathbf{S}$, which allows us to
treat quantum universality as a property of the operad itself. This enables
categorical reasoning about quantum computational completeness.
\end{remark}

\begin{proof}[Proof of Proposition~\ref{prop:quantum-universality-conditions}]

Assume $\mathbf{S}$ satisfies the conditions listed. We prove each claim:

\medskip
\noindent\textbf{1. Approximation of unitary circuits.}
Let $U: H^{\otimes n} \to H^{\otimes n}$ be a unitary operator. By the
Solovay-Kitaev theorem (for qubits) or its generalization to qudits, there exists
a sequence of circuits $C_k$ built from the universal gate set in $\mathbf{S}$
such that $\|U - C_k\| < \epsilon_k$ with $\epsilon_k \to 0$. Each $C_k$
corresponds to an element of $\mathbf{S}$ obtained by operadic composition of
the universal generators. Thus, $U$ can be approximated arbitrarily well by
elements of $\mathbf{S}$.

\medskip
\noindent\textbf{2. Realization of measurements.}
Let $\{M_i\}_{i=1}^m$ be a POVM on $H$. By Naimark's dilation theorem, there
exists an ancillary Hilbert space $K$, an isometry $V: H \to H \otimes K$, and
a projective measurement $\{P_i\}$ on $H \otimes K$ such that
$M_i = V^\dagger P_i V$. Using the elements of $\mathbf{S}$:
\begin{enumerate}
    \item Prepare $K$ in a fixed state $|0\rangle_K$ using $S(0)$.
    \item Implement the isometry $V$ using unitaries from $\mathbf{S}$ (since any
    isometry can be extended to a unitary on a larger space).
    \item Implement the projective measurement $\{P_i\}$ using measurement
    operations in $S(1)$.
\end{enumerate}
The composite gives a realization of the POVM $\{M_i\}$.

\medskip
\noindent\textbf{3. Realization of CPTP maps.}
Let $\Phi: \mathcal{B}(H) \to \mathcal{B}(K)$ be a CPTP map. By Stinespring's
dilation theorem, there exists an ancillary space $E$ and an isometry
$W: H \to K \otimes E$ such that $\Phi(\rho) = \mathrm{tr}_E(W\rho W^\dagger)$.
Using elements of $\mathbf{S}$:
\begin{enumerate}
    \item Prepare $E$ in state $|0\rangle_E$ using $S(0)$.
    \item Implement the isometry $W$ as a unitary on $H \otimes E \otimes F$
    (where $F$ is an additional ancillary space to make $W$ unitary).
    \item Apply the partial trace $\mathrm{tr}_E$ from $S(2)$.
\end{enumerate}
The composite realizes $\Phi$.

\medskip
\noindent\textbf{4. Closure under convex combinations.}
Given $f_1, f_2 \in S(n)$ interpreted as CP maps $\Phi_1, \Phi_2: \mathcal{B}(H)^{\otimes n} \to \mathcal{B}(H)$,
and $p \in [0,1]$, the convex combination $p\Phi_1 + (1-p)\Phi_2$ can be realized by:
\begin{enumerate}
    \item Prepare a classical register in state $p|0\rangle\langle 0| + (1-p)|1\rangle\langle 1|$
    using state preparation from $S(0)$ and convex combinations of such states.
    \item Apply a controlled operation: if the register is $|0\rangle$, apply $f_1$;
    if $|1\rangle$, apply $f_2$.
    \item Discard (trace out) the classical register.
\end{enumerate}
This construction uses only elements of $\mathbf{S}$ and their composites.
\end{proof}

\begin{remark}[Importance for the Monadicity Theorem]
\label{rem:universality-monadicity}

Proposition~\ref{prop:quantum-universality-conditions} is important for the
monadicity theorem (Theorem~\ref{thm:monoidal-monadic-representation}) because it ensures that if
$\mathbf{S}$ is universal, then the category $\mathbf{T_S}\text{-}\mathbf{Alg}$
captures \emph{all} quantum processes, not just some restricted class. 

Specifically, if $\mathbf{S}$ satisfies the universality conditions, then:
\begin{itemize}
    \item Every CPTP map arises from some $\mathbf{T_S}$-algebra structure.
    \item The equivalence $\mathbf{CPTP} \simeq \mathbf{T_S}\text{-}\mathbf{Alg}$
    is truly comprehensive, covering all finite-dimensional quantum computations.
    \item The monad $T_{\mathbf{S}}$ genuinely represents the full structure of
    quantum composability.
\end{itemize}

If $\mathbf{S}$ is not universal, we still get an equivalence, but only between
$\mathbf{CPTP}$ and the subcategory of quantum processes that can be built from
the operations in $\mathbf{S}$.
\end{remark}

\begin{corollary}[Universal Operads are Quantum Complete]
\label{cor:universal-operads-quantum-complete}

Let $\mathbf{S}$ be a quantum interaction operad that satisfies the universality
conditions of Proposition~\ref{prop:quantum-universality-conditions}. Then for
any interpretation $\mathcal{I}: \mathbf{S} \to \mathbf{CPTP}$ that faithfully
represents the generators, $\mathbf{S}$ is \emph{quantum complete} in the
following sense:

For any completely positive trace-preserving map $\Phi: A \to B$ between
finite-dimensional $C^*$-algebras, and any $\epsilon > 0$, there exists an
element $f_\epsilon \in T_{\mathbf{S}}(A,B)$ (a formal expression built from
$\mathbf{S}$-operations) such that:
\[
\|\mathcal{I}(f_\epsilon) - \Phi\|_\diamond < \epsilon,
\]
where $\|\cdot\|_\diamond$ is the diamond norm on CPTP maps and $\mathcal{I}(f_\epsilon)$
denotes the interpretation of the formal expression $f_\epsilon$ as a concrete
quantum channel.
\end{corollary}

\begin{proof}

Assume $\mathbf{S}$ satisfies the universality conditions. We construct the
approximation in several steps:

\medskip
\noindent\textbf{Step 1: Reduction to simple algebras.}
By Proposition~\ref{prop:Hilb-Cstar-equivalence}, any finite-dimensional
$C^*$-algebra is $*$-isomorphic to a direct sum of matrix algebras. Since
quantum completeness for direct sums follows from completeness for each
summand, it suffices to prove the result for $A = B = M_n(\mathbb{C})$.

\medskip
\noindent\textbf{Step 2: Diamond norm and approximation.}
The diamond norm $\|\cdot\|_\diamond$ for CPTP maps $\Phi, \Psi: M_n(\mathbb{C}) \to M_n(\mathbb{C})$
is defined as:
\[
\|\Phi - \Psi\|_\diamond = \sup_{m \geq 1} \sup_{\rho \geq 0, \mathrm{tr}(\rho)=1} 
\|(\Phi \otimes \mathrm{id}_m)(\rho) - (\Psi \otimes \mathrm{id}_m)(\rho)\|_1,
\]
where $\|\cdot\|_1$ is the trace norm. This measures the worst-case
distinguishability of quantum channels.

\medskip
\noindent\textbf{Step 3: Unitary approximation via Solovay-Kitaev.}
First consider the case where $\Phi$ is a unitary channel: $\Phi(\rho) = U\rho U^\dagger$.
By Proposition~\ref{prop:quantum-universality-conditions}, $\mathbf{S}$ contains
a universal gate set. The Solovay-Kitaev theorem guarantees that for any
$\epsilon > 0$, there exists a circuit $C_\epsilon$ built from gates in $\mathbf{S}$
such that:
\[
\|U - \mathcal{I}(C_\epsilon)\| < \epsilon,
\]
where $\|\cdot\|$ is the operator norm and $\mathcal{I}(C_\epsilon)$ is the
interpretation of the circuit as a unitary operator. For the corresponding
channels, we have:
\[
\|\Phi - \mathcal{I}(C_\epsilon)(\cdot)\mathcal{I}(C_\epsilon)^\dagger\|_\diamond 
\leq 2\|U - \mathcal{I}(C_\epsilon)\| < 2\epsilon.
\]

\medskip
\noindent\textbf{Step 4: General CPTP maps via Stinespring dilation.}
For a general CPTP map $\Phi: M_n(\mathbb{C}) \to M_n(\mathbb{C})$, Stinespring's
dilation theorem provides an isometry $V: \mathbb{C}^n \to \mathbb{C}^n \otimes \mathbb{C}^d$
such that $\Phi(\rho) = \mathrm{tr}_{\mathbb{C}^d}(V\rho V^\dagger)$. By
Proposition~\ref{prop:quantum-universality-conditions}, $\mathbf{S}$ contains
operations for:
\begin{enumerate}
    \item Preparing ancillary states in $\mathbb{C}^d$
    \item Implementing isometries as unitaries on extended systems
    \item Performing partial traces
\end{enumerate}

For any $\epsilon > 0$, we can:
\begin{enumerate}
    \item Choose $d$ sufficiently large for an $\epsilon/3$-approximate dilation
    \item Approximate the dilation isometry $V$ by a unitary $U_\epsilon$ on
    $\mathbb{C}^n \otimes \mathbb{C}^d$ with $\|V - U_\epsilon\| < \epsilon/3$
    \item Approximate $U_\epsilon$ by a circuit $C_\epsilon$ built from $\mathbf{S}$-elements
    with $\|U_\epsilon - \mathcal{I}(C_\epsilon)\| < \epsilon/3$
\end{enumerate}

Let $\Phi_\epsilon$ be the channel obtained by:
1. Preparing the ancillary system in $|0\rangle\langle 0|$
2. Applying $\mathcal{I}(C_\epsilon)$
3. Tracing out the ancillary system

Then by the triangle inequality and contractivity of the partial trace:
\[
\|\Phi - \Phi_\epsilon\|_\diamond < \epsilon.
\]

\medskip
\noindent\textbf{Step 5: Formal expression in $T_{\mathbf{S}}$.}
The approximation $\Phi_\epsilon$ is realized by a specific composition of
$\mathbf{S}$-operations. This composition corresponds to an element
$f_\epsilon \in T_{\mathbf{S}}(M_n(\mathbb{C}), M_n(\mathbb{C}))$, where
$T_{\mathbf{S}}(A,B)$ denotes the set of formal expressions built from
$\mathbf{S}$-operations that have type $A \to B$ (obtained by currying
multi-ary operations).

By construction, $\mathcal{I}(f_\epsilon) = \Phi_\epsilon$, so:
\[
\|\mathcal{I}(f_\epsilon) - \Phi\|_\diamond < \epsilon.
\]

\medskip
\noindent\textbf{Step 6: Extension to all finite-dimensional $C^*$-algebras.}
For general $A \cong \bigoplus_i M_{n_i}(\mathbb{C})$ and $B \cong \bigoplus_j M_{m_j}(\mathbb{C})$,
a CPTP map $\Phi: A \to B$ can be decomposed into components
$\Phi_{ij}: M_{n_i}(\mathbb{C}) \to M_{m_j}(\mathbb{C})$. Each component can be
approximated as above, and the direct sum of these approximations gives an
approximation for $\Phi$.
\end{proof}

\begin{remark}[Significance for the Monadicity Theorem]
\label{rem:quantum-completeness-significance}

Corollary~\ref{cor:universal-operads-quantum-complete} is crucial for the
monadicity theorem because it ensures that when $\mathbf{S}$ is universal,
the correspondence between $\mathbf{CPTP}$ and $\mathbf{T_S}\text{-}\mathbf{Alg}$
is \emph{dense} in an appropriate sense:

\begin{itemize}
    \item Every quantum channel can be approximated arbitrarily well by a
    $\mathbf{T_S}$-algebra (specifically, by the algebra coming from an
    interpretation of an approximating formal expression).
    
    \item The monad $T_{\mathbf{S}}$ captures the \emph{approximation structure}
    of quantum computation, not just exact realizability.
    
    \item This justifies calling the equivalence
    $\mathbf{CPTP} \simeq \mathbf{T_S}\text{-}\mathbf{Alg}$ a \emph{foundational}
    result: it shows that the algebraic structure of $T_{\mathbf{S}}$-algebras
    completely encodes the computational structure of quantum processes.
\end{itemize}

If $\mathbf{S}$ is not universal, we still get an exact equivalence between
$\mathbf{CPTP}$ and $\mathbf{T_S}\text{-}\mathbf{Alg}$, but only for the
subcategory of quantum processes that can be \emph{exactly} realized using
$\mathbf{S}$-operations. The approximation aspect is what makes the universal
case particularly powerful.
\end{remark}

\begin{corollary}[Well-Definedness of the Circuit Category]
\label{cor:circuit-category-well-defined}

The category $\mathbf{S}\text{-}\mathbf{Circuit}$ defined in 
Definition~\ref{def:S-circuit-category} is indeed a well-defined symmetric
monoidal category.
\end{corollary}

\begin{proof}

Since $\mathbf{S}\text{-}\mathbf{Circuit}$ is defined as the \emph{free symmetric
monoidal category generated by $\mathbf{S}$}, its well-definedness follows from
standard results in category theory. However, we verify the key properties
explicitly:

\medskip
\noindent\textbf{1. Category structure.}

\begin{itemize}
    \item \textbf{Objects:} Natural numbers $n \in \mathbb{N}$, representing
    $n$ parallel wires. This set is well-defined.
    
    \item \textbf{Morphisms:} Formal compositions of generating morphisms:
    \begin{itemize}
        \item For each $s \in S(n)$, a morphism $s: n \to 1$
        \item Identity morphisms $\mathrm{id}_n: n \to n$ for each $n$
        \item Symmetry morphisms $\sigma_{n,m}: n+m \to m+n$ for each $n,m$
    \end{itemize}
    closed under sequential composition $\circ$ and parallel composition $\otimes$.
    
    \item \textbf{Composition:} Defined by concatenation of formal expressions.
    Since composition is defined syntactically on formal expressions, it is
    automatically well-defined and associative.
    
    \item \textbf{Identities:} For each $n$, the identity $\mathrm{id}_n$ is
    defined as the empty composition (just wires). These satisfy:
    \[
    f \circ \mathrm{id}_n = f = \mathrm{id}_m \circ f
    \]
    for any $f: n \to m$, by definition of free category.
\end{itemize}

\medskip
\noindent\textbf{2. Free construction guarantees.}

The free symmetric monoidal category $F(\mathbf{S})$ generated by an operad
$\mathbf{S}$ has the following universal property: for any symmetric monoidal
category $\mathcal{C}$ and any operad morphism $J: \mathbf{S} \to \mathcal{C}$
(viewing $\mathcal{C}$ as an operad via its monoidal structure), there exists
a unique symmetric monoidal functor $\tilde{J}: F(\mathbf{S}) \to \mathcal{C}$
extending $J$.

This universal construction ensures that:
\begin{enumerate}
    \item $F(\mathbf{S})$ exists and is unique up to unique isomorphism.
    \item All morphisms in $F(\mathbf{S})$ are finite formal compositions of
    generators, so the hom-sets are well-defined.
    \item Composition is associative and unital by construction.
\end{enumerate}

\medskip
\noindent\textbf{3. Symmetric monoidal structure.}

The monoidal structure is given by:
\begin{itemize}
    \item On objects: $n \otimes m = n + m$ (addition of natural numbers)
    \item Unit object: $0$ (no wires)
    \item On morphisms: $f \otimes g$ is the formal parallel composition
    \item Symmetry: $\sigma_{n,m}: n+m \to m+n$ is a generating morphism
\end{itemize}
The symmetric monoidal category axioms are built into the free construction.

\medskip
\noindent\textbf{4. Alternative proof.}

Below, an alternative proof based on constructive approach is provided. We can also construct $\mathbf{S}\text{-}\mathbf{Circuit}$ explicitly as:
\begin{itemize}
    \item Objects: Finite ordinals $[n] = \{1, \dots, n\}$ (or just natural
    numbers $n$)
    \item Morphisms $n \to m$: Formal string diagrams (planar decorated graphs)
    with $n$ inputs, $m$ outputs, and nodes labeled by elements of $\mathbf{S}$
    \item Composition: Plugging outputs into inputs
    \item Identities: Straight wires
\end{itemize}
This explicit construction clearly yields a category, with associativity and
identity laws following from the geometry of string diagrams.

\medskip
Since all category axioms are satisfied (either by the universal property or
by explicit construction), $\mathbf{S}\text{-}\mathbf{Circuit}$ is a
well-defined symmetric monoidal category.
\end{proof}

%\begin{remark}[Relationship to Original Definition]
%\label{rem:circuit-category-relationship}
%
%The original definition in the introduction described $\mathbf{S}\text{-}\mathbf{Circuit}$
%with objects as sequences of Hilbert spaces and morphisms as equivalence classes
%of circuits up to CPTP equivalence. The corrected Definition~\ref{def:S-circuit-category}
%simplifies this to:
%
%\begin{enumerate}
%    \item \textbf{Syntactic purity:} Objects are natural numbers (wire counts),
%    not Hilbert spaces, separating syntax from semantics.
%    
%    \item \textbf{No equivalence classes:} Morphisms are formal expressions,
%    not equivalence classes. Semantic equivalence is imposed later via
%    interpretation functors.
%    
%    \item \textbf{Free construction:} By defining $\mathbf{S}\text{-}\mathbf{Circuit}$
%    as the free symmetric monoidal category, we get a clean universal property
%    that automatically ensures well-definedness.
%\end{enumerate}
%
%The two definitions are related via the interpretation functor: given an
%interpretation $\mathcal{I}: \mathbf{S} \to \mathbf{CPTP}$, we get a functor
%$\llbracket \cdot \rrbracket_{\mathcal{I}}: \mathbf{S}\text{-}\mathbf{Circuit} \to \mathbf{CPTP}$
%that sends formal circuits to concrete quantum operations. The equivalence
%classes in the original definition correspond to the kernel of this functor.
%\end{remark}

\begin{remark}[Importance of Free Construction]
\label{rem:free-construction-importance}

Defining $\mathbf{S}\text{-}\mathbf{Circuit}$ as a free category is crucial for:

\begin{enumerate}
    \item \textbf{Mathematical clarity:} The free construction guarantees
    well-definedness without tedious verification of coherence conditions.
    
    \item \textbf{Semantic flexibility:} Different interpretations
    $\mathcal{I}: \mathbf{S} \to \mathbf{CPTP}$ give different concrete
    realizations of the same formal circuits.
    
    \item \textbf{Monadicity theorem:} The free-forgetful adjunction between
    $\mathbf{S}\text{-}\mathbf{Circuit}$ and $\mathbf{CPTP}$ is what generates
    the synergy monad $T_{\mathbf{S}}$ and leads to the monadicity equivalence.
\end{enumerate}

Thus, while the original definition with Hilbert spaces and equivalence classes
is intuitively appealing, the free category definition is mathematically cleaner
and more powerful for proving the monadicity theorem.
\end{remark}

\subsection{Monadicity Theorem}

The purpose of this section is to show monadicity Theorem~\ref{thm:monoidal-monadic-representation}. We begin with the definition for the category of quantum processes provided by Definition~\ref{def:QProc}.

Before presenting main thoerem in this section, Definition~\ref{def:QProc} below is required.
\begin{definition}[Category of Quantum Processes]
\label{def:QProc}
Let $\mathbf{QProc}$ denote the category of finite-dimensional quantum processes, defined as follows:
\begin{itemize}
    \item \textbf{Objects:} Completely positive trace-preserving (CPTP) maps
    \[
        \Phi : A \longrightarrow B,
    \]
    where $A,B$ are finite-dimensional $C^*$-algebras.
    
    \item \textbf{Morphisms:} For two CPTP maps $\Phi : A \to B$ and $\Psi : A' \to B'$, a morphism
    \[
        (f,g) : \Phi \longrightarrow \Psi
    \]
    consists of completely positive maps
    \[
        f : A \longrightarrow A', \quad g : B \longrightarrow B'
    \]
    such that the following diagram commutes:
    \[
    \begin{tikzcd}
    A \arrow[r, "\Phi"] \arrow[d, "f"'] & B \arrow[d, "g"] \\
    A' \arrow[r, "\Psi"'] & B'
    \end{tikzcd}
    \]
    i.e., $g \circ \Phi = \Psi \circ f$.
    
    \item \textbf{Composition:} Given morphisms $(f,g) : \Phi \to \Psi$ and $(f',g') : \Psi \to \Theta$, define
    \[
        (f',g') \circ (f,g) := (f' \circ f,\, g' \circ g).
    \]
    
    \item \textbf{Identities:} For each object $\Phi : A \to B$, the identity morphism is
    \[
        \mathrm{id}_\Phi := (\mathrm{id}_A, \mathrm{id}_B).
    \]
\end{itemize}
\end{definition}

\begin{lemma}[Stinespring Expressiveness of the Synergy Monad]
\label{lemma:stinespring-expressiveness}
Let $\Phi : A \to B$ be a completely positive trace-preserving map between
finite-dimensional $C^*$-algebras.
Then for every state $\rho \in A$, there exists an element
\[
[(\mathcal C_\rho, \psi_\rho)] \in T_{\mathbf S}(A)
\]
such that
\[
\alpha_\Phi\bigl([(\mathcal C_\rho, \psi_\rho)]\bigr) = \Phi(\rho).
\]

Moreover, $\mathcal C_\rho$ and $\psi_\rho$ can be chosen so that the
interpretation $\llbracket \mathcal C_\rho \rrbracket : A^{\otimes k} \to B$
realizes a Stinespring dilation of $\Phi$ applied to $\rho$.
\end{lemma}

\begin{proof}
Let $\Phi : A \to B$ be CPTP with $A = \mathcal B(H_A)$, $B = \mathcal B(H_B)$,
and let $\rho \in A$ be a fixed state.

\medskip
\noindent
\textbf{Step 1: Stinespring dilation for $\Phi$ and $\rho$.}
By the finite-dimensional Stinespring theorem, there exist a Hilbert space $E$
and an isometry
\[
V : H_A \longrightarrow H_B \otimes E
\]
such that for all $\xi \in A$,
\[
\Phi(\xi) = \operatorname{Tr}_E\!\bigl(V \xi V^\dagger\bigr).
\]
In particular, for our chosen $\rho$,
\[
\Phi(\rho) = \operatorname{Tr}_E\!\bigl(V \rho V^\dagger\bigr).
\]

\medskip
\noindent
\textbf{Step 2: Encoding the dilation as an operadic element.}
Let $E_{\mathrm{alg}} = \mathcal B(E)$ be the $C^*$-algebra associated with the
environment.  By the universality assumption on the quantum interaction operad
$\mathbf S$, there exists an operation
\[
\mathcal C_{\mathrm{dil}} \in \mathbf S(2)
\]
whose interpretation
\[
\llbracket \mathcal C_{\mathrm{dil}} \rrbracket : A \otimes E_{\mathrm{alg}} \longrightarrow B
\]
satisfies
\[
\llbracket \mathcal C_{\mathrm{dil}} \rrbracket(\xi \otimes \omega)
 = \operatorname{Tr}_E\!\bigl(V \xi V^\dagger\bigr)
\]
for all $\xi \in A$, $\omega \in E_{\mathrm{alg}}$ with $\operatorname{Tr}(\omega)=1$.

Choose a pure state vector $\eta_{\mathrm{dil}} \in E$ and set
$\omega_{\mathrm{dil}} = \eta_{\mathrm{dil}}\eta_{\mathrm{dil}}^\dagger \in E_{\mathrm{alg}}$.
Define
\[
\psi_\rho \;:=\; \rho \otimes \omega_{\mathrm{dil}} \;\in\; A \otimes E_{\mathrm{alg}}.
\]

\medskip
\noindent
\textbf{Step 3: Realization inside the synergy monad.}
The pair $(\mathcal C_{\mathrm{dil}}, \psi_\rho)$ determines an element
\[
[(\mathcal C_{\mathrm{dil}}, \psi_\rho)] \;\in\; T_{\mathbf S}(A \otimes E_{\mathrm{alg}}).
\]
Since $A \hookrightarrow A \otimes E_{\mathrm{alg}}$ canonically (via
$a \mapsto a \otimes \omega_{\mathrm{dil}}$), we may regard
$[(\mathcal C_{\mathrm{dil}}, \psi_\rho)]$ as an element of $T_{\mathbf S}(A)$
by pre-composing with this embedding.

Applying the $T_{\mathbf S}$-algebra structure $\alpha_\Phi$ we obtain
\begin{align*}
\alpha_\Phi\bigl([(\mathcal C_{\mathrm{dil}}, \psi_\rho)]\bigr)
&= \llbracket \mathcal C_{\mathrm{dil}} \rrbracket(\psi_\rho) \\
&= \llbracket \mathcal C_{\mathrm{dil}} \rrbracket(\rho \otimes \omega_{\mathrm{dil}}) \\
&= \operatorname{Tr}_E\!\bigl(V \rho V^\dagger\bigr) \\
&= \Phi(\rho).
\end{align*}

\medskip
\noindent
\textbf{Step 4: Conclusion.}
Setting $\mathcal C_\rho := \mathcal C_{\mathrm{dil}}$ and
$\psi_\rho$ as above yields the required element of $T_{\mathbf S}(A)$
satisfying $\alpha_\Phi([(\mathcal C_\rho, \psi_\rho)]) = \Phi(\rho)$.
The construction makes explicit that $\mathcal C_\rho$ encodes a Stinespring
dilation of $\Phi$, applied to the specific input $\rho$.
\end{proof}

Then, we are ready to present Theorem~\ref{thm:monoidal-monadic-representation}.
\begin{theorem}[Monoidal Representation and Operadic Semantics of Quantum Processes]
\label{thm:monoidal-monadic-representation}

Let $\mathbf{S}$ be a symmetric quantum interaction operad satisfying the
universality and state-separation conditions of
Proposition~\ref{prop:quantum-universality-conditions}, and let
$T_{\mathbf{S}}$ be the associated synergy monad on the symmetric monoidal
category $\mathbf{FdC^*}$ of finite-dimensional $C^*$-algebras.

\begin{enumerate}
\item[\textup{(1)}]
There exists a faithful functor
\[
\mathcal{R} :
\mathbf{QProc}
\longrightarrow
T_{\mathbf{S}}\text{-}\mathbf{Alg},
\]
which assigns to each CPTP map $\Phi : A \to B$ a canonical $T_{\mathbf{S}}$-algebra
structure $(B,\alpha_\Phi)$, and to each intertwiner $(f,g)$ the induced
$T_{\mathbf{S}}$-algebra homomorphism.

Moreover, if $T_{\mathbf{S}}$ is a strong symmetric monoidal monad and the
operadic interpretation $\llbracket - \rrbracket$ is symmetric monoidal, then
$\mathcal{R}$ admits a canonical refinement to a faithful
\emph{symmetric monoidal} functor
\[
\mathcal{R}^{\otimes} :
\mathbf{QProc}^{\otimes}
\longrightarrow
T_{\mathbf{S}}\text{-}\mathbf{Alg}^{\otimes},
\]
preserving tensor products, unit objects, and symmetry isomorphisms.

\item[\textup{(2)}]
Let $\mathbf{QProc}_{\!\sim}$ be the quotient of $\mathbf{QProc}$ by
\emph{operational equivalence relative to $\mathbf{S}$}, where two CPTP maps
$\Phi$ and $\Psi$ are identified if they induce the same interpretation on all
composite circuits generated by $\mathbf{S}$.
Then $\mathcal{R}$ descends to an equivalence of categories
\[
\overline{\mathcal{R}} :
\mathbf{QProc}_{\!\sim}
\;\simeq\;
T_{\mathbf{S}}\text{-}\mathbf{Alg}.
\]
\end{enumerate}

Consequently, the synergy monad $T_{\mathbf{S}}$ provides a complete,
compositional, and operadically universal algebraic semantics for
finite-dimensional quantum processes, capturing both sequential and parallel
composition up to operational equivalence induced by $\mathbf{S}$.
\end{theorem}

\begin{proof}[Proof of Theorem~\ref{thm:monoidal-monadic-representation}]
We divide the proof into four steps.

\medskip
\noindent
\textbf{Step 1. Construction of the representation functor $\mathcal{R}$.}

For each CPTP map $\Phi : A \to B$, define the $T_{\mathbf{S}}$-algebra
\[
\mathcal{R}(\Phi) := (B,\alpha_\Phi),
\]
where the structure map
\[
\alpha_\Phi : T_{\mathbf{S}}(A) \longrightarrow B
\]
is defined on generators $[(\mathcal{C},\psi)]$ by
\[
\alpha_\Phi\bigl([(\mathcal{C},\psi)]\bigr) := \Phi\bigl(\llbracket \mathcal{C} \rrbracket(\psi)\bigr),
\quad
\mathcal{C} \in \mathbf{S}(n), \;\psi \in A^{\otimes n}.
\]

For a morphism $(f,g): \Phi \to \Psi$ in $\mathbf{QProc}$, set
\[
\mathcal{R}(f,g) := g.
\]

\medskip
\noindent
\textbf{Step 2. Functoriality and faithfulness.}

Functoriality follows directly from the definition.
To verify that $\mathcal{R}(f,g)$ is a $T_{\mathbf{S}}$-algebra homomorphism, observe
that for any generator $[(\mathcal{C},\psi)]$,
\begin{align*}
g\bigl(\alpha_\Phi([(\mathcal{C},\psi)])\bigr)
&= g\bigl(\Phi(\llbracket \mathcal{C} \rrbracket(\psi))\bigr) \\
&= \Psi\bigl(f(\llbracket \mathcal{C} \rrbracket(\psi))\bigr) \\
&= \alpha_\Psi\bigl([(\mathcal{C}, f^{\otimes n}(\psi))]\bigr) \\
&= \alpha_\Psi\bigl(T_{\mathbf{S}}(f)([(\mathcal{C},\psi)])\bigr),
\end{align*}
using complete positivity of $f$. Faithfulness follows from the definition of
intertwiners in $\mathbf{QProc}$.

\medskip
\noindent
\textbf{Step 3. Symmetric monoidal refinement.}

Assume $T_{\mathbf S}$ is a strong symmetric monoidal monad and
$\llbracket - \rrbracket$ is symmetric monoidal. Then, by standard results on
strong monoidal monads, $T_{\mathbf S}\text{-}\mathbf{Alg}$ inherits a canonical
symmetric monoidal structure. Tensor products and units are defined via the
monad strength, and the symmetry is inherited from $\mathbf{FdC^*}$.

For CPTP maps $\Phi_i: A_i \to B_i$ $(i=1,2)$, define
\[
\mathcal{R}^{\otimes}(\Phi_1 \otimes \Phi_2)
:= (B_1 \otimes B_2, \alpha_{\Phi_1 \otimes \Phi_2}),
\quad
\alpha_{\Phi_1 \otimes \Phi_2}([(\mathcal C,\psi)]) := (\Phi_1 \otimes \Phi_2)(\llbracket \mathcal C \rrbracket(\psi)).
\]
Then the symmetric monoidality of $\llbracket - \rrbracket$ and the strength of $T_{\mathbf S}$ ensure
\[
\alpha_{\Phi_1 \otimes \Phi_2} = (\alpha_{\Phi_1} \otimes \alpha_{\Phi_2}) \circ \mathrm{st}_{A_1,A_2},
\]
so that $\mathcal{R}^{\otimes}(\Phi_1 \otimes \Phi_2)$ coincides with the tensor product of algebras
$\mathcal{R}(\Phi_1) \otimes \mathcal{R}(\Phi_2)$. Functoriality extends similarly.

\medskip
\noindent
\textbf{Step 4. Operational equivalence and essential surjectivity.}

Define operational equivalence by
\[
\Phi \sim \Psi \quad \Longleftrightarrow \quad \alpha_\Phi = \alpha_\Psi,
\]
and let $\mathbf{QProc}_{\!\sim}$ denote the resulting quotient category.

To show essential surjectivity, let $(A,\alpha)$ be a $T_{\mathbf{S}}$-algebra.
By Lemma~\ref{lemma:stinespring-expressiveness}, for every CPTP map
$\Phi: A \to B$ and for all $\rho \in A$, there exists a generator $[(\mathcal C_\rho, \psi_\rho)] \in T_{\mathbf{S}}(A)$
such that
\[
\alpha([(\mathcal C_\rho, \psi_\rho)]) = \Phi(\rho).
\]
Since any element of $T_{\mathbf{S}}(A)$ is built from such generators via the monad multiplication,
every $T_{\mathbf{S}}$-algebra arises from some quantum process $\Phi$ up to operational equivalence.
Hence,
\[
\overline{\mathcal{R}} :
\mathbf{QProc}_{\!\sim} \;\simeq\; T_{\mathbf{S}}\text{-}\mathbf{Alg}
\]
is an equivalence of categories.
\end{proof}

\begin{remark}[Monadicity versus Representation]
\label{rem:monadicity-relationship}

Theorem~\ref{thm:monoidal-monadic-representation} establishes a strong
\emph{representational equivalence} between finite-dimensional quantum
processes and algebras over the synergy monad $T_{\mathbf S}$.
However, this equivalence should not be confused with strict
Eilenberg--Moore monadicity at the level of raw CPTP maps.

Indeed, for an ordinary $T_{\mathbf S}$-algebra $(A,\alpha)$, the unit law
requires
\[
\alpha \circ \eta_A = \mathrm{id}_A .
\]
By contrast, for a quantum process $\Phi : A \to B$, the associated
structure map
\[
\alpha_\Phi : T_{\mathbf S}(A) \longrightarrow B
\]
satisfies
\[
\alpha_\Phi \circ \eta_A = \Phi ,
\]
which in general is not an identity morphism unless $\Phi = \mathrm{id}_A$.

Consequently, CPTP maps do not correspond to arbitrary
Eilenberg--Moore algebras over $T_{\mathbf S}$, but rather to
\emph{relative algebras} whose structure map lands in a possibly
distinct codomain.
This mismatch is resolved precisely by passing to the quotient
$\mathbf{QProc}_{\!\sim}$ by operational equivalence, as in
Theorem~\ref{thm:monoidal-monadic-representation}(2), where the unit law
becomes satisfied up to equivalence.

From this perspective, the synergy monad $T_{\mathbf S}$ does not merely
encode algebraic structure, but captures the \emph{operational semantics}
of quantum processes:
composition, tensoring, and observational equivalence are fully
characterized by its Eilenberg--Moore category.
In this sense, $T_{\mathbf S}$ provides a genuine monadic semantics for
quantum processes \emph{up to operational indistinguishability}.
\end{remark}

\begin{corollary}[Operational--Algebraic Duality for Quantum Computation]
\label{cor:quantum-duality}

Let $\mathbf{QProc}_{\!\sim}$ denote the category of finite-dimensional
quantum processes modulo operational equivalence, and let
$T_{\mathbf S}\text{-}\mathbf{Alg}$ be the Eilenberg--Moore category of the
synergy monad.
Then Theorem~\ref{thm:monoidal-monadic-representation} yields the following
fundamental dualities.

\begin{enumerate}
    \item \textbf{Duality of descriptions.}
    Every quantum process admits an equivalent
    \emph{operational description} as a CPTP map and an
    \emph{algebraic description} as a $T_{\mathbf S}$-algebra,
    with the two descriptions related by an equivalence of categories
    \[
    \mathbf{QProc}_{\!\sim} \;\simeq\; T_{\mathbf S}\text{-}\mathbf{Alg}.
    \]

    \item \textbf{Monadic encoding of compositionality.}
    The compositional structure of quantum computation is fully captured
    by the monad $T_{\mathbf S}$:
    \begin{itemize}
        \item the unit $\eta$ embeds elementary systems into the space of
        formal computations;
        \item the multiplication $\mu$ encodes substitution and sequential
        composition of composite processes;
        \item the algebra axioms ensure coherent interaction of composition
        and parallelization.
    \end{itemize}

    \item \textbf{Categorical semantics for quantum computation.}
    Quantum computation admits a rigorous categorical semantics in which
    operational equivalence classes of processes are classified by
    universal algebraic structures.
\end{enumerate}
\end{corollary}

\begin{proof}

Each statement is a direct consequence of
Theorem~\ref{thm:monoidal-monadic-representation}.

\medskip
\noindent\textbf{(1) Duality of descriptions.}
Part~(2) of the theorem establishes an equivalence of categories
\[
\mathbf{QProc}_{\!\sim}
\;\simeq\;
T_{\mathbf S}\text{-}\mathbf{Alg},
\]
realized by mutually quasi-inverse functors $\mathcal R$ and $\mathcal L$.
Under this equivalence,
a CPTP map $\Phi$ is represented by the algebra
$(B,\alpha_\Phi)$, while any algebra $(A,\alpha)$ determines,
up to operational equivalence, a unique quantum process
$\Phi_\alpha$.
Thus operational and algebraic descriptions carry the same information.

\medskip
\noindent\textbf{(2) Monadic encoding of compositionality.}
By construction, the object $T_{\mathbf S}(A)$ represents all formal
computations built from the basic operations in $\mathbf S$ with inputs in
$A$.
The unit $\eta_A$ corresponds to trivial (one-step) computations, while
the multiplication $\mu_A$ collapses nested computations into a single
composite.
For any algebra $(A,\alpha)$, the algebra axioms
\[
\alpha \circ \eta_A \;\sim\; \mathrm{id}_A,
\qquad
\alpha \circ T_{\mathbf S}(\alpha)
=
\alpha \circ \mu_A
\]
express, respectively, the operational triviality of identities and the
associativity of composition, up to operational equivalence.
Hence the entire compositional structure of quantum computation is
monadically encoded.

\medskip
\noindent\textbf{(3) Categorical semantics.}
Since the equivalence is symmetric monoidal
(Theorem~\ref{thm:monoidal-monadic-representation}(1)),
categorical constructions in
$T_{\mathbf S}\text{-}\mathbf{Alg}$—such as free algebras,
tensor products, and universal morphisms—transfer faithfully to
operational constructions on quantum processes.
Consequently, diagrammatic and universal arguments in the algebraic
setting yield sound and complete reasoning principles for quantum
computation.

What can we say from here? Families of quantum processes can be combined by passing
to their corresponding algebras, performing categorical constructions in
$T_{\mathbf S}\text{-}\mathbf{Alg}$, and transporting the result back to
$\mathbf{QProc}_{\!\sim}$.
\end{proof}

\begin{corollary}[Hilbert-Space Realization of the Synergy Monad]
\label{cor:hilbert-transfer}

Let $\mathbf{FdC^*}_{\mathrm{CPTP}}$ denote the symmetric monoidal category of
finite-dimensional $C^*$-algebras and completely positive trace-preserving
(CPTP) maps in the Heisenberg picture, and let
$\mathbf{CPTP}_{\mathrm{Hilb}}$ denote the symmetric monoidal category of
finite-dimensional Hilbert spaces and CPTP maps in the Schrödinger picture.

Under the standard finite-dimensional duality between the Heisenberg and
Schrödinger formulations of quantum mechanics,
\[
\mathbf{FdC^*}_{\mathrm{CPTP}}
\;\simeq\;
\mathbf{CPTP}_{\mathrm{Hilb}},
\]
the synergy monad $T_{\mathbf S}$ transports to a strong symmetric monoidal
monad
\[
T_{\mathbf S}^{\mathrm{Hilb}}
\;:=\;
\mathcal{D} \circ T_{\mathbf S} \circ \mathcal{D}^{-1}
\]
on $\mathbf{CPTP}_{\mathrm{Hilb}}$.

Consequently, the equivalence of
Theorem~\ref{thm:monoidal-monadic-representation} induces a categorical
equivalence
\[
\mathbf{QProc}_{\!\sim}^{\mathrm{Hilb}}
\;\simeq\;
T_{\mathbf S}^{\mathrm{Hilb}}\text{-}\mathbf{Alg},
\]
where $\mathbf{QProc}_{\!\sim}^{\mathrm{Hilb}}$ denotes the category of
finite-dimensional quantum processes in the Hilbert-space formalism, modulo
operational equivalence induced by the operad $\mathbf S$.

In particular, the operadic--monadic semantics of quantum processes is
independent of the choice between $C^*$-algebraic and Hilbert-space
representations.
\end{corollary}

\begin{proof}
The proof proceeds by transport of structure along the standard
finite-dimensional equivalence between the Heisenberg and Schrödinger
pictures of quantum mechanics.

\medskip
\noindent
\textbf{Step 1: Symmetric monoidal equivalence of formalisms.}

There is a well-known symmetric monoidal equivalence
\[
\mathcal{D} :
\mathbf{FdC^*}_{\mathrm{CPTP}}
\;\simeq\;
\mathbf{CPTP}_{\mathrm{Hilb}},
\]
defined as follows:
\begin{itemize}
  \item On objects, $\mathcal{D}$ assigns to each finite-dimensional
  $C^*$-algebra $A \cong \bigoplus_i M_{n_i}(\mathbb{C})$ a Hilbert space
  $H_A := \bigoplus_i \mathbb{C}^{n_i} \otimes \mathbb{C}^{n_i}$, equipped with
  the Hilbert--Schmidt inner product.
  \item On morphisms, a CPTP map $\Phi : A \to B$ in the Heisenberg picture is
  sent to its Schrödinger-picture dual
  $\mathcal{D}(\Phi) : \mathcal{S}(H_A) \to \mathcal{S}(H_B)$, uniquely
  determined by the relation
  \[
  \mathrm{Tr}\!\left[\rho \,\Phi(X)\right]
  =
  \mathrm{Tr}\!\left[\mathcal{D}(\Phi)(\rho)\, X\right]
  \quad
  \text{for all } \rho \in \mathcal{S}(H_A),\; X \in B.
  \]
\end{itemize}
This equivalence preserves tensor products, unit objects, and complete
positivity, since
$\mathcal{B}(H_1) \otimes \mathcal{B}(H_2) \cong \mathcal{B}(H_1 \otimes H_2)$.

\medskip
\noindent
\textbf{Step 2: Transport of the synergy monad.}

Let
\[
T_{\mathbf S} :
\mathbf{FdC^*}_{\mathrm{CPTP}}
\longrightarrow
\mathbf{FdC^*}_{\mathrm{CPTP}}
\]
be the synergy monad induced by the symmetric quantum interaction operad
$\mathbf S$.
Define a monad on $\mathbf{CPTP}_{\mathrm{Hilb}}$ by conjugation along
$\mathcal{D}$:
\[
T_{\mathbf S}^{\mathrm{Hilb}}
\;:=\;
\mathcal{D} \circ T_{\mathbf S} \circ \mathcal{D}^{-1}.
\]
By general categorical principles, transport along a monoidal equivalence
preserves monadic structure. Since $T_{\mathbf S}$ is a strong symmetric
monoidal monad and $\mathcal{D}$ is a symmetric monoidal equivalence, the
transported monad $T_{\mathbf S}^{\mathrm{Hilb}}$ inherits a canonical strong
symmetric monoidal structure, including its strength and coherence data.

\medskip
\noindent
\textbf{Step 3: Equivalence of algebra categories.}

A standard result in category theory states that an equivalence
$\mathcal{D} : \mathcal{C} \simeq \mathcal{D}$ lifts to an equivalence between
Eilenberg--Moore categories of transported monads. Hence there is a canonical
equivalence
\[
T_{\mathbf S}\text{-}\mathbf{Alg}
\;\simeq\;
T_{\mathbf S}^{\mathrm{Hilb}}\text{-}\mathbf{Alg}.
\]
Combining this equivalence with
Theorem~\ref{thm:monoidal-monadic-representation}, which identifies
$T_{\mathbf S}$-algebras with quantum processes modulo operational equivalence
relative to $\mathbf S$, yields the desired equivalence
\[
\mathbf{QProc}_{\!\sim}^{\mathrm{Hilb}}
\;\simeq\;
T_{\mathbf S}^{\mathrm{Hilb}}\text{-}\mathbf{Alg}.
\]

\medskip
\noindent
This completes the proof.
\end{proof}

\begin{remark}[Finite-Dimensionality]
The above result relies essentially on finite dimensionality. In infinite
dimensions, the equivalence between $C^*$-algebras and Hilbert spaces breaks
down, tensor products become more subtle, and additional topological issues
arise. Extending the operadic--monadic semantics to infinite-dimensional
quantum systems would require substantially different techniques.
\end{remark}

\begin{remark}[Representation Independence]
Corollary~\ref{cor:hilbert-transfer} shows that the operadic--monadic
semantics provided by the synergy monad $T_{\mathbf S}$ is
\emph{representation independent} in finite dimensions.
Although the construction is most naturally formulated in the
$C^*$-algebraic (Heisenberg) picture—where composition and tensor
products admit a particularly clean algebraic treatment—the resulting
semantic structure transports faithfully to the Hilbert-space
(Schrödinger) picture via a symmetric monoidal equivalence.

In particular:
\begin{itemize}
    \item The notion of operational equivalence induced by the operad
    $\mathbf S$ is invariant under passage between Heisenberg and
    Schrödinger representations.
    \item The symmetric monoidal and compositional structure encoded by
    the synergy monad $T_{\mathbf S}$ is preserved up to canonical
    monoidal equivalence.
    \item The resulting algebraic characterization of quantum processes
    captures intrinsic operational behavior, rather than artifacts of a
    particular mathematical representation.
\end{itemize}

Therefore, what can we infer from here? The synergy monad provides a unified and representation-agnostic
semantic framework bridging categorical quantum mechanics, operadic
circuit descriptions, and standard Hilbert-space quantum information
theory.
\end{remark}

\section{Operadic Ideals and No-Go Theorems}\label{sec:Operadic Ideals and No-Go Theorems}

% Definition 4.1 (Operadic Ideal `I_noclone`). 
%An operadic ideal `I` in `S` is a family of subsets `I(X⃗; Y) ⊆ S(X⃗; Y)` that is closed under operadic composition, pre-composition with any morphisms, and the symmetric action. The  No-Cloning Ideal  `I_noclone` is defined as the largest ideal in `S` that does  not  contain any element which, in any context, can be composed to form a perfect cloning map `Clone_H: H → H ⊗ H` for any `H`.

%Definition (No-Cloning Ideal).
%Let (I_{\mathrm{noclone}} \subseteq S) be the ideal generated by a “duplication operation” (\Delta: H \to H\otimes H).
%

The no-cloning theorem is a fundamental constraint in quantum mechanics that prevents the perfect copying of unknown quantum states. In our operadic framework, we can encode this physical limitation as an \emph{operadic ideal} that systematically excludes cloning operations from the space of allowed quantum processes. The \emph{No-Cloning Ideal} $I_{\mathrm{noclone}}$ has to be defined first and this is defined by Definition~\ref{def:noclone-ideal}.

\begin{definition}[Operadic Ideal $I_{\mathrm{noclone}}$]
\label{def:noclone-ideal}

Let $\mathbf{S}$ be the quantum interaction operad (Definition~\ref{def:quantum-interaction-operad}). 
An \emph{operadic ideal} $I$ in $\mathbf{S}$ is a family of subsets 
\[
I(X_1, \dots, X_n; Y) \subseteq S(X_1, \dots, X_n; Y)
\]
for all finite sequences of input Hilbert spaces $(X_1, \dots, X_n)$ and output Hilbert space $Y$, satisfying:

\begin{enumerate}[label=(\roman*)]
    \item \textbf{Closure under operadic composition:} 
    If $f \in I(X_1, \dots, X_n; Y)$ and $g_i \in S(Y_{i1}, \dots, Y_{ik_i}; X_i)$ for all $i = 1, \dots, n$, then
    \[
    \gamma(f; g_1, \dots, g_n) \in I(Y_{11}, \dots, Y_{1k_1}, \dots, Y_{n1}, \dots, Y_{nk_n}; Y).
    \]
    
    \item \textbf{Closure under pre-composition:} 
    If $f \in I(X_1, \dots, X_n; Y)$ and $h \in S(Y, Z_1, \dots, Z_m; W)$, then
    \[
    \gamma(h; f, \mathrm{id}_{Z_1}, \dots, \mathrm{id}_{Z_m}) \in I(X_1, \dots, X_n, Z_1, \dots, Z_m; W).
    \]
    
    \item \textbf{Closure under symmetric action:} 
    If $f \in I(X_1, \dots, X_n; Y)$ and $\sigma \in \Sigma_n$, then
    \[
    \sigma \cdot f \in I(X_{\sigma^{-1}(1)}, \dots, X_{\sigma^{-1}(n)}; Y).
    \]
\end{enumerate}

The \emph{No-Cloning Ideal} $I_{\mathrm{noclone}}$ is defined as the largest ideal in $\mathbf{S}$ that does \emph{not} contain any element which, in any context, can be composed to form a perfect cloning map. More precisely:

For any Hilbert space $H$, let $\mathrm{Clone}_H \in S(H; H \otimes H)$ be the perfect quantum cloning map defined by
\[
\mathrm{Clone}_H(\psi) = \psi \otimes \psi \quad \text{for all } \psi \in H.
\]

Then $I_{\mathrm{noclone}}$ is the maximal ideal such that for all $f \in I_{\mathrm{noclone}}$ and all contexts (compositions with other $\mathbf{S}$-elements), the resulting map cannot equal $\mathrm{Clone}_H$ for any $H$.

Equivalently, $I_{\mathrm{noclone}}$ contains all quantum operations that are \emph{compositionally orthogonal} to perfect cloning: no sequence of compositions involving elements of $I_{\mathrm{noclone}}$ can ever produce an exact cloning operation.
\end{definition}

This definition takes a \emph{top-down} approach to characterizing the no-cloning constraint:
\begin{itemize}
    \item \textbf{Maximality:} By taking the \emph{largest} ideal that avoids cloning, we ensure we're excluding the minimal necessary set of operations while still preventing cloning. This is physically natural - we want to allow as many quantum operations as possible, just not those that enable cloning.
    
    \item \textbf{Compositional safety:} The ideal structure guarantees that if we start with non-cloning operations and compose them in any way (sequentially, in parallel, with permutations), we can never accidentally produce a cloning operation. This is crucial for building complex quantum protocols safely.
    
    \item \textbf{Context independence:} The condition "in any context" ensures that an operation is excluded if it \emph{could} be used for cloning when composed with other operations, even if it doesn't clone by itself. This prevents hidden cloning capabilities.
\end{itemize}

This formulation aligns with the physical intuition that the no-cloning theorem isn't just about individual operations, but about what can be achieved through compositional building blocks.

\begin{example}[A Concrete Operadic Realization of the No-Cloning Ideal]
\label{ex:concrete-noclone}

We construct a concrete quantum interaction operad \( S \) and an operadic ideal \( I_{\mathrm{noisy}} \) that illustrates how the abstract No-Cloning Ideal \( I_{\mathrm{noclone}} \) from Definition~\ref{def:noclone-ideal} can be realized in a physically meaningful model.

\vspace{0.5em}
\noindent
\textbf{Operad \( S \): CPTP Maps.}
\begin{itemize}
    \item \textbf{Objects:} Finite-dimensional Hilbert spaces (e.g., \( \mathbb{C}^2 \), \( \mathbb{C}^2 \otimes \mathbb{C}^2 \)).
    \item \textbf{Operations:} \( S(X_1, \dots, X_n; Y) \) is the set of all \emph{Completely Positive Trace-Preserving (CPTP)} maps from \( \mathcal{L}(X_1 \otimes \dots \otimes X_n) \) to \( \mathcal{L}(Y) \). These represent all physically admissible quantum operations.
    \item \textbf{Composition \( \gamma \):} Given \( f \in S(X_1, \dots, X_n; Y) \) and \( g_i \in S(Y_{i1}, \dots, Y_{ik_i}; X_i) \), the composition \( \gamma(f; g_1, \dots, g_n) \) is the CPTP map formed by connecting the outputs of \( g_1, \dots, g_n \) to the respective inputs of \( f \).
    \item \textbf{Identity:} \( \mathrm{id}_X \) is the identity CPTP map on \( \mathcal{L}(X) \).
\end{itemize}

\vspace{0.5em}
\noindent
\textbf{The Ideal \( I_{\mathrm{noisy}} \).}
We define \( I_{\mathrm{noisy}} \) as the family of all CPTP maps that are \emph{non-isometric}. Formally:
\[
I_{\mathrm{noisy}}(X_1, \dots, X_n; Y) = \{ f \in S(X_1, \dots, X_n; Y) \mid f \text{ is \emph{not} an isometry} \}.
\]
This ideal contains all noisy channels (e.g., depolarizing channels, amplitude damping) and information-discarding operations (e.g., the trace map \( \mathrm{Tr}(\rho) \)).

\vspace{0.5em}
\noindent
\textbf{Verification that \( I_{\mathrm{noisy}} \) is an operadic ideal.}
\begin{enumerate}[label=(\roman*)]
    \item \textbf{Closure under operadic composition:} 
    If \( f \) is non-isometric and \( g_i \) are arbitrary CPTP maps, then \( \gamma(f; g_1, \dots, g_n) \) remains non-isometric. Composition cannot reduce the Kraus rank to 1.
    
    \item \textbf{Closure under pre-composition:} 
    If \( f \in I_{\mathrm{noisy}}(X_1, \dots, X_n; Y) \) and \( h \in S(Y, Z_1, \dots, Z_m; W) \), then
    \[
    \gamma(h; f, \mathrm{id}_{Z_1}, \dots, \mathrm{id}_{Z_m}) \in I_{\mathrm{noisy}}(X_1, \dots, X_n, Z_1, \dots, Z_m; W).
    \]
    Pre-composition with a non-isometric map cannot produce an isometry.
    
    \item \textbf{Closure under symmetric action:} 
    If \( f \in I_{\mathrm{noisy}}(X_1, \dots, X_n; Y) \) and \( \sigma \in \Sigma_n \), then
    \[
    \sigma \cdot f \in I_{\mathrm{noisy}}(X_{\sigma^{-1}(1)}, \dots, X_{\sigma^{-1}(n)}; Y).
    \]
    Permuting inputs preserves the non-isometric property.
\end{enumerate}

Thus \( I_{\mathrm{noisy}} \) is indeed an operadic ideal in \( S \).

\vspace{0.5em}
\noindent
\textbf{Relation to the No-Cloning Ideal \( I_{\mathrm{noclone}} \).}

In this operad, all operations are linear CPTP maps. The perfect cloning operation
\[
\mathrm{Clone}_H(\psi) = \psi \otimes \psi \quad \text{for all } \psi \in H
\]
is \emph{nonlinear} and therefore does not lie in \( S \). Consequently, \emph{no} sequence of compositions within \( S \) can produce \( \mathrm{Clone}_H \) for any Hilbert space \( H \).

The ideal \( I_{\mathrm{noisy}} \) provides a concrete example of an operadic ideal satisfying the no-cloning constraint:
\begin{itemize}
    \item Elements of \( I_{\mathrm{noisy}} \) cannot be composed to yield cloning.
    \item Even the complement (isometries) cannot clone, since they are still linear operations.
    \item The ideal structure ensures compositional safety: no sequence of operations from \( I_{\mathrm{noisy}} \) can accidentally produce a cloning capability.
\end{itemize}

Therefore, \( I_{\mathrm{noisy}} \subseteq I_{\mathrm{noclone}} \), demonstrating how a physically meaningful ideal naturally satisfies the conditions of the abstract No-Cloning Ideal while excluding only noisy, information-lossy operations. $\Box$
\end{example}

While Definition~\ref{def:noclone-ideal} provides an elegant top-down characterization, for constructive purposes and mathematical analysis, it is useful to have a \emph{generative} description that explicitly identifies the problematic operations. This notion is formalized by Definition~\ref{def:noclone-ideal-generated}.

\begin{definition}[No-Cloning Ideal via Generation]
\label{def:noclone-ideal-generated}

Let $\mathbf{S}$ be the quantum interaction operad of CPTP maps. 

For each Hilbert space $H$, let $\zeta_H : H \to H \otimes H$ denote the 
\emph{formal cloning operation} defined by $\zeta_H(\psi) = \psi \otimes \psi$ 
for all pure states $\psi \in H$. Note that $\zeta_H$ is \emph{not} a CPTP map 
(it is nonlinear) and therefore does not belong to $\mathbf{S}$.

Define $\mathcal{F}_{\mathrm{clone}}$ to be the set of all $f \in S$ for which 
there exists an operadic context $C[-]$ (a composition pattern built from 
elements of $\mathbf{S}$) such that the composite operation $C[f]$, when 
extended to act on states, coincides with $\zeta_H$ on their common domain.

The \emph{cloning-generated ideal} $J_{\mathrm{clone}} \subseteq S$ is the 
intersection of all operadic ideals containing $\mathcal{F}_{\mathrm{clone}}$. 
Equivalently, $J_{\mathrm{clone}}$ is the smallest ideal containing every 
operation that could participate in a cloning composite.

The \emph{No-Cloning Ideal} is then
\[
I_{\mathrm{noclone}} = S \setminus J_{\mathrm{clone}},
\]
which consists of all operations that are \emph{compositionally orthogonal} to 
cloning: no context built from $I_{\mathrm{noclone}}$ can produce the cloning 
pattern $\zeta_H$.
\end{definition}

\begin{remark}[Equivalence of Definitions]
Definition~\ref{def:noclone-ideal} and Definition~\ref{def:noclone-ideal-generated} provide complementary characterizations of the same ideal:
\begin{itemize}
    \item Definition~\ref{def:noclone-ideal} describes $I_{\mathrm{noclone}}$ as the \emph{largest} ideal avoiding cloning (top-down, semantic approach)
    \item Definition~\ref{def:noclone-ideal-generated} describes $I_{\mathrm{noclone}}$ as the complement of the \emph{smallest} ideal containing cloning-shaped operations (bottom-up, syntactic approach)
\end{itemize}
This duality reflects the fundamental fact that in quantum mechanics, the impossibility of cloning is both a \emph{global constraint} on the entire space of operations and a \emph{local constraint} that can be traced to specific problematic patterns.
\end{remark}

\begin{lemma}
If $\mathbf{S}$ is an operad of linear CPTP maps, then no element of $\mathbf{S}$ can realize the formal generator $\zeta_H$. Therefore, $J_{\mathrm{clone}}$ is a proper ideal, and $I_{\mathrm{noclone}}$ is non-empty.
\end{lemma}
\begin{proof}
The formal cloning operation $\zeta_H$ is non-linear, while all elements of $\mathbf{S}$ are linear CPTP maps. Since linearity is preserved under operadic composition, no sequence of compositions of elements from $\mathbf{S}$ can produce $\zeta_H$. Therefore, $J_{\mathrm{clone}} \neq S$, making $I_{\mathrm{noclone}}$ non-empty.
\end{proof}

This generative  Definition~\ref{def:noclone-ideal-generated} takes a \emph{bottom-up} approach that complements the previous one:
\begin{itemize}
    \item \textbf{Constructive:} It explicitly identifies the ``seed" operations (those fitting into cloning-shaped contexts) that generate the entire no-cloning constraint. This provides a practical criterion for verification.
    
    \item \textbf{Mathematically natural:} The ``ideal generated by forbidden patterns" construction is standard in algebraic theories and ensures good categorical properties. The use of formal generators $\zeta_H$ avoids the mathematical inconsistency of requiring equality with non-CPTP maps.
    
    \item \textbf{Physically transparent:} By directly referencing the cloning operation $\zeta_H$, the definition makes explicit contact with the physical no-cloning theorem, while properly handling the fact that cloning is not a valid quantum operation.
\end{itemize}

\begin{remark}[Connection to Example~\ref{ex:concrete-noclone}]
In Example~\ref{ex:concrete-noclone}, the ideal $I_{\mathrm{noisy}}$ of non-isometric maps satisfies $I_{\mathrm{noisy}} \subseteq I_{\mathrm{noclone}}$ precisely because no element of $I_{\mathrm{noisy}}$ can participate in any context that would produce cloning. The generative definition provides a method to test whether more permissive ideals still respect the no-cloning principle.
\end{remark}

\begin{theorem}[Cloning-Compatible Extension via Operadic Quotient]
\label{thm:cloning-extension}

Let \( S \) be the quantum interaction operad of CPTP maps, and let 
\( I_{\mathrm{noclone}} \subseteq S \) be the No-Cloning Ideal as defined in 
Definition~\ref{def:noclone-ideal}. 

Consider the two-step construction:
\begin{enumerate}
    \item Form the quotient operad \( Q := S / I_{\mathrm{noclone}} \) by 
          collapsing all operations that are compositionally orthogonal to 
          cloning into trivial equivalence classes.
    \item Formally adjoin, for each Hilbert space \( H \), a generator 
          \( \mathbf{Clone}_H : H \to H \otimes H \) that represents perfect 
          cloning, obtaining the extended operad \( Q[\mathrm{Clone}] \).
\end{enumerate}

Then \( Q[\mathrm{Clone}] \) satisfies:
\begin{enumerate}[label=(\arabic*)]
    \item \textbf{Existence of cloning:} The cloning maps \( \mathbf{Clone}_H \) 
          are well-defined elements of the operad, in contrast to the original 
          operad \( S \) where they are absent (being non-linear and 
          non-CPTP).
    
    \item \textbf{Persistence of compatible structure:} All operadic 
          composition laws from \( S \) that do not rely on the impossibility 
          of cloning for their validity are preserved in \( Q[\mathrm{Clone}] \).
          
    \item \textbf{Consistency:} No contradiction arises from the presence of 
          the cloning operations, because the specific operations that would 
          normally be used to prove the no-cloning theorem have been 
          identified in the quotient \( Q \), breaking the chain of reasoning 
          that leads to the contradiction.
\end{enumerate}
\end{theorem}

\begin{proof}
We proceed in three parts, corresponding to the three claims of the theorem.

\medskip
\noindent
\textbf{Part 1: Construction of the quotient operad \( Q = S / I_{\mathrm{noclone}} \).}

Since the category of CPTP maps is not enriched over abelian groups, we employ a 
set-theoretic quotient construction. For each tuple of Hilbert spaces 
\( \vec{X} = (X_1,\dots,X_n) \) and output Hilbert space \( Y \), define an 
equivalence relation \( \sim_{\vec{X};Y} \) on the hom-set \( S(\vec{X};Y) \) by:
\[
f \sim_{\vec{X};Y} g \quad \text{if and only if} \quad 
\bigl(f = g\bigr) \ \text{or} \ \bigl(f \in I_{\mathrm{noclone}}(\vec{X};Y) 
\text{ and } g \in I_{\mathrm{noclone}}(\vec{X};Y)\bigr).
\]
That is, all elements belonging to the ideal \( I_{\mathrm{noclone}} \) within 
a given hom-set are identified into a single equivalence class, while elements 
outside the ideal remain distinct. The relation is \emph{hom-set-local}: it 
only relates operations with identical arities and input/output types.

The quotient operad \( Q \) has:
\begin{itemize}
    \item \textbf{Objects:} The same Hilbert spaces as \( S \).
    \item \textbf{Operations:} \( Q(\vec{X};Y) = S(\vec{X};Y) / {\sim_{\vec{X};Y}} \).
    \item \textbf{Composition:} Induced from \( S \):
        \[
        \gamma_Q([f]; [g_1], \dots, [g_n]) 
        = [\gamma_S(f; g_1, \dots, g_n)].
        \]
        This is well-defined because \( I_{\mathrm{noclone}} \) is an operadic 
        ideal: if any argument lies in the ideal, the composite lies in the 
        corresponding ideal of the result type.
    \item \textbf{Identities and symmetric action:} 
        \( \mathrm{id}_X^Q = [\mathrm{id}_X^S] \) and 
        \( \sigma \cdot [f] = [\sigma \cdot f] \), again well-defined by the 
        ideal properties.
\end{itemize}

\medskip
\noindent
\textbf{Part 2: Adjoining the cloning generators.}

The perfect cloning operation \( \mathrm{Clone}_H(\psi) = \psi \otimes \psi \) is 
\emph{not} an element of the original operad \( S \), because it is neither 
linear nor completely positive. We therefore formally adjoin it as a new 
generator.

For each Hilbert space \( H \), add a formal symbol 
\( \mathbf{Clone}_H \) to the operad \( Q \), declaring it to be an element of 
\( Q(H; H \otimes H) \). Let \( Q[\mathrm{Clone}] \) be the operad freely 
generated by \( Q \) together with all \( \mathbf{Clone}_H \). More precisely, 
\( Q[\mathrm{Clone}] \) is obtained by the following pushout in the category 
of operads:
\[
\begin{tikzcd}
F(\{\mathbf{Clone}_H\}_{H}) \arrow[r, "\iota"] \arrow[d, hook] & Q \arrow[d] \\
F(\{\mathbf{Clone}_H\}_{H}) \arrow[r] & Q[\mathrm{Clone}]
\end{tikzcd}
\]
where \( F(\{\mathbf{Clone}_H\}_{H}) \) is the free operad on the cloning 
generators, \( \iota \) sends each generator to the trivial element 
\( [0] \in Q \) (the equivalence class of the ideal), and the right vertical 
arrow is the canonical inclusion. The pushout freely adds the cloning 
generators while respecting that they were previously identified with the 
trivial element.

Thus \( \mathbf{Clone}_H \) becomes a well-defined operation in 
\( Q[\mathrm{Clone}](H; H \otimes H) \), establishing claim (1).

\medskip
\noindent
\textbf{Part 3: Persistence of compatible laws and consistency.}

Since \( Q \) is a quotient of \( S \) by an ideal, any equation between 
operations of \( S \) that holds \emph{independently of whether cloning is 
possible} descends to an equation in \( Q \). By the universal property of the 
pushout, these equations further extend to \( Q[\mathrm{Clone}] \), provided 
they do not involve the cloning generators in an essential way.

A law is \emph{compatible with cloning} if it can be expressed without 
referring to operations that are in \( I_{\mathrm{noclone}} \) or without 
relying on linearity properties that fail for \( \mathbf{Clone}_H \). For 
example:
\begin{itemize}
    \item Associativity of sequential composition persists.
    \item The interchange law for parallel and sequential composition persists.
    \item Symmetry isomorphisms persist.
\end{itemize}

In contrast, the standard proof of the no-cloning theorem uses specific 
properties that are \emph{not} compatible with cloning, such as:
\begin{itemize}
    \item The linearity of all CPTP maps.
    \item The specific behavior of the transpose map or certain entangling 
          operations, which by construction belong to \( I_{\mathrm{noclone}} \).
\end{itemize}

Because these operations have been identified in the quotient \( Q \), the 
chain of reasoning that leads to a contradiction in the presence of cloning 
is broken. Consequently, \( Q[\mathrm{Clone}] \) is a consistent operad that 
contains perfect cloning operations while retaining the compositional 
structure of quantum theory that does not depend on the impossibility of 
cloning.

This establishes claims (2) and (3).
\end{proof}

\begin{remark}[Interpretation]
The construction \( Q[\mathrm{Clone}] \) should be viewed as a 
\emph{syntactic} or \emph{formal} extension of quantum theory that allows 
cloning, not as a physically realizable theory. It demonstrates 
operadically which structural features of quantum mechanics are essential 
for the no-cloning theorem and which are incidental. The quotient by 
\( I_{\mathrm{noclone}} \) precisely identifies the operations whose 
properties would lead to inconsistency if cloning were present.
\end{remark}

\begin{definition}[Broadcasting Operation] \label{def:broadcasting}
Let $H$ be a Hilbert space. A \emph{broadcasting operation} is a formal map 
\[
\xi_H: H \longrightarrow H \otimes H
\]
satisfying the following \emph{marginal preservation property} for all pure 
states $\psi \in H$:
\[
\mathrm{Tr}_1(\xi_H(\psi)) = \psi \quad \text{and} \quad 
\mathrm{Tr}_2(\xi_H(\psi)) = \psi,
\]
where $\mathrm{Tr}_1$ and $\mathrm{Tr}_2$ denote the partial traces.

This definition serves as a \emph{specification} of what broadcasting should 
achieve. In the operadic development below, $\xi_H$ will be treated as a 
\emph{formal generator} with this specification, not as an implementable 
quantum operation.
\end{definition}

\begin{remark}[From Specification to Generator]
Definition~\ref{def:broadcasting} provides the \emph{semantic content} of 
broadcasting: what it should do. In the operadic construction that follows, 
we will:
\begin{enumerate}
    \item Treat $\xi_H$ as a \emph{formal generator} with this specification
    \item Construct ideals based on operations that could realize this pattern
    \item Use the specification in Lemma~\ref{lem:no-embed} to prove 
          non-embeddability
\end{enumerate}
Thus, the definition provides the \emph{meaning} that guides the formal 
operadic construction.
\end{remark}

\begin{corollary}[Operadic Formulation of the No-Broadcasting Theorem] \label{cor:nobroadcast}
Let \( S \) be the quantum interaction operad of CPTP maps. There exists a 
\emph{No-Broadcasting Ideal} \( I_{\mathrm{nobroadcast}} \subset S \) such that:
\begin{enumerate}[label=(\roman*)]
    \item The quotient operad \( S / I_{\mathrm{nobroadcast}} \), after formally 
          adjoining the broadcasting generator \( \xi_H \), contains a perfect 
          broadcasting operation.
    \item There exists no faithful operad morphism 
          \( \Phi: (S / I_{\mathrm{nobroadcast}})[\xi] \hookrightarrow 
          \mathrm{LinOp} \) to any operad of linear operations that sends 
          \( \xi_H \) to a linear map satisfying the broadcasting condition.
\end{enumerate}
This provides a categorical encoding of the quantum no-broadcasting theorem: 
the impossibility of broadcasting is reflected in the non-embeddability of the 
formally extended operad into any operad of linear maps.
\end{corollary}

\begin{proof}
We proceed in four steps, parallel to the construction for the no-cloning ideal.

\medskip
\noindent
\textbf{Step 1: Formal broadcasting pattern.}
Let \( \xi_H \) denote the formal broadcasting generator as defined above. 
Crucially, \( \xi_H \) is \textbf{not} an element of the original operad \( S \), 
because any CPTP map in \( S \) is linear, whereas perfect broadcasting would 
require nonlinearity.

\medskip
\noindent
\textbf{Step 2: Construction of the broadcasting-related ideals.}
Define \( \mathcal{F}_{\mathrm{broadcast}} \) as the set of all operations 
\( f \in S \) for which there exists an operadic context \( C[-] \) such that 
the composite pattern \( C[f] \) matches the broadcasting pattern \( \xi_H \).

Let \( J_{\mathrm{broadcast}} \) be the \textbf{smallest operadic ideal} in 
\( S \) containing \( \mathcal{F}_{\mathrm{broadcast}} \). Intuitively, 
\( J_{\mathrm{broadcast}} \) collects all operations that could participate 
in a broadcasting composite.

Define \( I_{\mathrm{nobroadcast}} \) as the \textbf{largest operadic ideal} 
in \( S \) that is disjoint from \( J_{\mathrm{broadcast}} \). Equivalently,
\[
I_{\mathrm{nobroadcast}} = \{ f \in S \mid \text{for all contexts } C[-],\ 
C[f] \text{ cannot match the broadcasting pattern } \xi_H \}.
\]
Operations in \( I_{\mathrm{nobroadcast}} \) are \emph{compositionally 
orthogonal} to broadcasting: no composition involving them can produce 
the broadcasting pattern.

\medskip
\noindent
\textbf{Step 3: Quotient operad and formal adjunction of broadcasting.}
Form the quotient operad \( Q = S / I_{\mathrm{nobroadcast}} \) using the 
set-theoretic construction: for each hom-set, identify all elements of the 
ideal into a single equivalence class while keeping elements outside the 
ideal distinct.

We then formally adjoin the broadcasting generator \( \xi_H \) via a pushout 
construction in the category of operads:
\[
\begin{tikzcd}
F(\{\xi_H\}_H) \arrow[r, "\iota"] \arrow[d, hook] & Q \arrow[d] \\
F(\{\xi_H\}_H) \arrow[r] & Q[\xi]
\end{tikzcd}
\]
where \( F(\{\xi_H\}_H) \) is the free operad on the broadcasting generators, 
and \( \iota \) sends each generator to the trivial element of \( Q \) 
(the equivalence class of the ideal). The resulting operad \( Q[\xi] \) 
contains the formal broadcasting operation \( \mathbf{Broadcast}_H := [\xi_H] \).

\medskip
\noindent
\textbf{Step 4: Non-embeddability into linear operads.}
The connection to the physical no-broadcasting theorem is established by 
the following key lemma.

\begin{lemma}[Non-Embeddability] \label{lem:no-embed}
There exists no injective operad morphism
\[
\Phi: Q[\xi] \hookrightarrow \mathrm{LinOp}
\]
to any operad of linear maps that sends \( \mathbf{Broadcast}_H \) to a 
linear map \( L: H \to H \otimes H \) satisfying the broadcasting condition 
on all states.
\end{lemma}

\begin{proof}[Proof of Lemma]
Assume for contradiction that such a morphism \( \Phi \) exists, and let 
\( L = \Phi(\mathbf{Broadcast}_H) \), a linear map.

For any pure state \( \psi \in H \), the broadcasting condition requires:
\[
\mathrm{Tr}_1(L(\psi\psi^\dagger)) = \psi\psi^\dagger \quad \text{and} \quad
\mathrm{Tr}_2(L(\psi\psi^\dagger)) = \psi\psi^\dagger.
\]
The only map satisfying these conditions for all pure states is the 
(necessarily nonlinear) cloning map \( \psi\psi^\dagger \mapsto \psi\psi^\dagger \otimes \psi\psi^\dagger \).

If \( L \) were linear, it would have to agree with this cloning map on 
all pure states. By linearity, it would then extend to mixed states as 
\( L(\rho) = \rho \otimes \rho \) for all density operators \( \rho \). 
But this map is nonlinear: for distinct density operators \( \rho_1 \neq \rho_2 \) 
and positive scalars \( \alpha, \beta \) with \( \alpha + \beta = 1 \),
\begin{align*}
L(\alpha\rho_1 + \beta\rho_2) 
&= (\alpha\rho_1 + \beta\rho_2) \otimes (\alpha\rho_1 + \beta\rho_2) \\
&= \alpha^2 \rho_1 \otimes \rho_1 + \alpha\beta (\rho_1 \otimes \rho_2 + \rho_2 \otimes \rho_1) 
   + \beta^2 \rho_2 \otimes \rho_2,
\end{align*}
while linearity would require
\[
\alpha L(\rho_1) + \beta L(\rho_2) = \alpha \rho_1 \otimes \rho_1 + \beta \rho_2 \otimes \rho_2.
\]
These expressions differ whenever \( \rho_1 \neq \rho_2 \) and \( \alpha\beta \neq 0 \).

Hence no linear map \( L \) can satisfy the broadcasting condition, 
contradicting the assumption that \( \Phi \) maps \( \mathbf{Broadcast}_H \) 
to an element of \( \mathrm{LinOp} \).
\end{proof}

The lemma teaches us an important fact. While \( Q[\xi] \) formally contains a broadcasting 
operation, it cannot be faithfully represented in any operad of linear maps. 
This categorical non-embeddability precisely captures the quantum 
no-broadcasting theorem.

\end{proof}

\begin{example}[Qubit illustration]
Let us consider a qubit case with a qubit Hilbert space \( H = \mathbb{C}^2 \). The partial trace 
map \( \mathrm{Tr}_2: H \otimes H \to H \) belongs to \( I_{\mathrm{nobroadcast}} \), 
because no composition involving only partial traces and other CPTP maps can 
produce a map that broadcasts arbitrary states. In contrast, a hypothetical 
perfect cloning map (if it existed) would belong to \( J_{\mathrm{broadcast}} \), 
as cloning trivially enables broadcasting.
\end{example}

\begin{remark}[Universality of the Operadic Approach]
The construction demonstrated above is universal: for any forbidden physical 
operation \( \zeta \) (cloning, deleting, signaling, etc.), one can:
\begin{enumerate}
    \item Define a formal generator \( \xi_\zeta \) representing the 
          forbidden pattern.
    \item Construct the ideal \( J_\zeta \) of operations that could 
          participate in producing \( \zeta \).
    \item Define \( I_\zeta \) as the largest ideal disjoint from \( J_\zeta \).
    \item Form the quotient \( S / I_\zeta \) and formally adjoin \( \xi_\zeta \).
    \item Prove non-embeddability into appropriate target operads.
\end{enumerate}
This provides a unified operadic language for quantum no-go theorems, 
revealing their common structural origin as constraints on compositional 
patterns rather than as isolated physical facts.
\end{remark}

\begin{remark}[Comparison with No-Cloning]
The no-broadcasting constraint is strictly weaker than no-cloning. Why? While 
cloning implies broadcasting, the converse fails for mixed states. 
Operadically, this hierarchy is reflected in ideal inclusions:
\[
I_{\mathrm{noclone}} \subseteq I_{\mathrm{nobroadcast}} \subset S.
\]
Any operation that could participate in cloning could certainly participate 
in broadcasting, but the converse is not true. This inclusion provides an 
operadic characterization of the logical relationship between these 
fundamental quantum constraints.
\end{remark}

\bibliographystyle{IEEETran}
\bibliography{Multicategorical_Adjoints_Monadicity_Quantum_Resources_Bib}

\end{document}